\documentclass[11pt,a4paper,leqno]{amsart}
\usepackage{amsmath,amsthm,amscd,amssymb,amsfonts, amsbsy}
\usepackage{latexsym}
\usepackage{txfonts}
\usepackage{exscale}
\usepackage{esint}

\usepackage{undertilde}

\usepackage[colorlinks,citecolor=red,pagebackref,hypertexnames=false]{hyperref}
\usepackage{bbm}
\usepackage{pgf}

\calclayout
\allowdisplaybreaks

\theoremstyle{plain}

\newtheorem{theorem}[equation]{Theorem}
\newtheorem{lemma}[equation]{Lemma}
\newtheorem{corollary}[equation]{Corollary}
\newtheorem{proposition}[equation]{Proposition}

\theoremstyle{definition}
\newtheorem{remark}[equation]{Remark}
\newtheorem{definition}[equation]{Definition}

\theoremstyle{remark}

\numberwithin{equation}{section}

\newcommand{\RR}{{\mathbb{R}}}
\newcommand{\ZZ}{{\mathbb{Z}}}
\newcommand{\PP}{\mathcal{P}}
\newcommand{\DD}{\mathbb{D}}
\newcommand{\CC}{{\mathbb{C}}}

\newcommand{\dist}{\operatorname{dist}}

\newcommand{\rn}{\mathbb{R}^n}

\newcommand{\reu}{\mathbb{R}^{n+1}_+}

\newcommand{\N}{\mathbb{N}}

\newcommand{\hm}{\omega}

\newcommand{\vp}{\varphi}

\newcommand{\ve}{v_{\epsilon}}

\newcommand{\Dom}{\mathrm{Dom}}
\newcommand{\loc}{\mathrm{loc}}
\newcommand{\ube}{\bar{u}_\epsilon}

\renewcommand{\emptyset}{\mbox{\textup{\O}}}

\DeclareMathOperator{\supp}{supp}

\DeclareMathOperator{\diam}{diam}

\def\div{\mathop{\operatorname{div}}\nolimits}

\begin{document}

\title[Carleson measure estimates and the Dirichlet problem]{Carleson measure estimates and the Dirichlet problem for degenerate elliptic equations}


\address{Steve Hofmann
\\
Department of Mathematics
\\
University of Missouri
\\
Columbia, MO 65211, USA.}
\email{hofmanns@missouri.edu}


\address{Phi L. Le
\\
Mathematics Department
\\
Syracuse University
\\
Syracuse, NY 13244, USA.}
\email{ple101@syr.edu}

\address{Andrew J. Morris
\\
School of Mathematics
\\
University of Birmingham
\\
Birmingham, B15 2TT, UK.}
\email{a.morris.2@bham.ac.uk}

\author{Steve Hofmann, Phi Le, Andrew J. Morris}

\thanks{
 }

\date{\today}

\begin{abstract}
We prove that the Dirichlet problem for degenerate elliptic equations $\div (A \nabla u) = 0$ in the upper half-space $(x,t)\in \RR^{n+1}_+$ is solvable when $n\geq2$ and the boundary data is in $L^p_\mu(\RR^n)$ for some $p<\infty$. The coefficient matrix $A$ is only assumed to be measurable, real-valued and $t$-independent with a degenerate bound and ellipticity controlled by an $A_2$-weight $\mu$. It is not required to be symmetric. The result is achieved by proving a Carleson measure estimate for all bounded solutions in order to deduce that the degenerate elliptic measure is in $A_\infty$ with respect to the $\mu$-weighted Lebesgue measure on $\RR^n$. The Carleson measure estimate allows us to avoid applying the method of $\epsilon$-approximability, which simplifies the proof obtained recently in the case of uniformly elliptic coefficients. The results have natural extensions to Lipschitz  domains.
\end{abstract}

\maketitle

\tableofcontents

\section{Introduction}\label{sec:intro}
We consider the Dirichlet boundary value problem for the degenerate elliptic equation $\div (A \nabla u)=0$ in the upper half-space $\RR^{n+1}_+$ when $n\geq2$ and which we make precise below. The boundary $\RR^n\times\{0\}$ is identified with $\RR^n$ and we adopt the notation $X=(x,t)$ for points $X\in\RR^{n+1}_+$ with coordinates $x\in\RR^n$ and $t\in(0,\infty)$. The gradient $\nabla:=(\nabla_x,\partial_t)$ and divergence $\div:=\div_x+\partial_t$ are with respect to all $(n+1)$-coordinates. The coefficient $A$ denotes an $(n+1)\times (n+1)$-matrix of measurable, real-valued and $t$-independent functions on $\RR^{n+1}_+$. The matrix $A(x):=A(x,t)$ is not required to be symmetric. We suppose that there exist constants $0<\lambda\leq\Lambda<\infty$ and an $A_2$-weight $\mu$ on $\RR^n$ such that the degenerate bound and ellipticity
\begin{equation}\label{eq:deg.def}
|\langle A(x)\xi, \zeta \rangle| \leq \Lambda \mu(x)|\xi| |\zeta|
\quad\text{and}\quad
\langle A(x)\xi, \xi \rangle \geq \lambda \mu(x)|\xi|^2
\end{equation}
hold for all $\xi,\zeta \in \RR^{n+1}$ and almost every $x\in \RR^n$. We use $\langle \cdot,\cdot\rangle$ and $|\cdot|$ to denote the Euclidean inner-product and norm. An $A_2$-weight $\mu$ on $\RR^n$ refers to a non-negative locally integrable function $\mu:\RR^n\rightarrow[0,\infty]$ such that
\[
[\mu]_{A_2(\RR^n)}
:=\sup_Q \left(\frac{1}{|Q|}\int_Q \mu(x)\,  dx \right) \left(\frac{1}{|Q|}\int_Q \frac{1}{\mu(x)}\ dx  \right)
<\infty,
\]
where $\sup_Q$ denotes the supremum over all cubes $Q$ in $\RR^n$ with volume~$|Q|$. We also use $\mu$ to denote the measure $\mu(Q):=\int_Q \mu(x)\ dx$ and consider the Lebesgue space $L^p_\mu(\RR^n)$ with the norm $\|f\|_{L^p_\mu(\RR^n)}:= (\int_{\RR^n} |f|^p\ d\mu)^{1/p}$ for all $p\in[1,\infty)$. There is also the notation $\fint_Q f\ d\mu:=\mu(Q)^{-1}\int_Q f\ d\mu$ whilst $\fint_Q f :=|Q|^{-1}\int_Q f(x)\ dx$.

If $\mu$ is identically 1, then $A$ is called uniformly elliptic. The solvability of the Dirichlet problem for general non-symmetric coefficients in that case was obtained only recently by Hofmann, Kenig, Mayboroda and Pipher in~\cite{HKMP1}. The result in dimension $n=1$ had been obtained previously by Kenig, Koch, Pipher and Toro in~\cite{KKoPT}. These results assert that for each  uniformly elliptic coefficient matrix $A$, there exists some $p<\infty$ for which the Dirichlet problem is solvable for $L^p$-boundary data. Conversely, counterexamples in \cite{KKoPT} show that for each $p<\infty$, there exists a uniformly elliptic coefficient matrix $A$ for which the Dirichlet problem is not solvable for $L^p$-boundary data. In contrast, solvability of the Dirichlet problem for symmetric coefficients in the uniformly elliptic case is well-understood, and we mention only that it was obtained by Jerison and Kenig in~\cite{JK} for $L^p$-boundary data when $2\leq p<\infty$.

The solvability of the Dirichlet problem in the uniformly elliptic case has also been established for a variety of complex coefficient structures (see, for instance, \cite{AS,HKMP1,HoMiMo}). A significant portion of that theory was recently extended to the degenerate elliptic case by Auscher, Ros\'{e}n and Rule in~\cite{ARR} for $L^2$-boundary data. That extension did not include, however, the results for general non-symmetric coefficients in \cite{HKMP1}. This paper complements the progress made in \cite{ARR} by extending the solvability obtained for the Dirichlet problem in \cite{HKMP1} to the degenerate elliptic case.

For solvability on the upper half-space $\RR^{n+1}_+$, the $A_2$-weight $\mu$ on $\RR^n$ is extended to the $t$-independent $A_2$-weight $\mu(x,t):=\mu(x)$ on $\RR^{n+1}$ (and $[\mu]_{A_2(\RR^{n+1})}=[\mu]_{A_2(\RR^n)}$). We then say that $u$ is a \textit{solution} of the equation $\div (A \nabla u)=0$ in an open set $\Omega\subseteq \RR^{n+1}$ when $u\in W^{1,2}_{\mu,\loc}(\Omega)$ and $\int_{\RR^{n+1}_+} \langle A\nabla u, \nabla \Phi \rangle =0$ for all smooth compactly supported functions $\Phi\in C^\infty_c(\Omega)$.  The solution space is the local $\mu$-weighted Sobolev space $W^{1,2}_{\mu,\loc}$ defined in Section~\ref{sec:prelims}. The convergence of solutions to boundary data is afforded by estimates for the non-tangential maximal function $N_*u$ of solutions $u$, defined by 
\[
(N_*u)(x) := \sup_{(y,t)\in\Gamma(x)} |u(y,t)|
\qquad \forall x\in\RR^n,
\]
where the cone $\Gamma(x) := \{(y,t)\in\RR^{n+1}_+ : |y-x|<t\}$. If $p\in(1,\infty)$, then the Dirichlet problem for $L^p_\mu(\RR^n)$-boundary data, or simply $(D)_{p,\mu}$, is said to be \textit{solvable} when for each $f\in L^p_\mu(\RR^n)$, there exists a solution $u$ such that
\begin{equation}\tag*{$(D)_{p,\mu}$}
\begin{cases}
\div (A \nabla u)=0 \textrm{ in } \RR^{n+1}_+,\\
N_* u \in L^p_{\mu}(\RR^n),\\
\lim_{t\rightarrow 0} u(\cdot,t) = f,
\end{cases}
\end{equation}
where the limit is required to converge in $L^p_\mu(\RR^n)$-norm and in the non-tangential sense whereby $\lim_{\Gamma(x)\ni(y,t)\rightarrow (x,0)} u(y,t) = f(x)$ for almost every $x\in\RR^n$. Note that this definition of solvability is distinct from \textit{well-posedness}, which requires that such solutions are unique. We are able to obtain a uniqueness result for solutions that converge uniformly to 0 at infinity, but the question of well-posedness more generally remains open (see Theorem~\ref{thm:DpSolvmain} and the preceding discussion).

A non-negative Borel measure $\omega$ on a cube $Q_0$ in $\RR^n$ is said to be in the $A_\infty$-class with respect to $\mu$, written $\omega \in A_\infty(\mu)$, when there exist constants $C,\theta>0$, which we call the $A_\infty(Q_0)$-constants, such that
\[
\omega(E) \leq C \left(\frac{\mu(E)}{\mu(Q)}\right)^\theta \omega(Q)
\]
for all cubes $Q\subseteq Q_0$ and all Borel sets $E\subseteq Q$. This is a scale-invariant version of the absolute continuity of $\omega$ with respect to $\mu$. It is well-known, at least in the uniformly elliptic case, that solvability of the Dirichlet problem for $L^p$-boundary data for some $p<\infty$ is equivalent to the property that an adapted harmonic measure (elliptic measure) belongs to $A_\infty$ with respect to the Lebesgue measure on $\RR^n$ (see Theorem 1.7.3 in \cite{K}). In the degenerate case, an adapted harmonic measure $\omega^X$, which we call degenerate elliptic measure, can also be defined at each $X\in\RR^{n+1}_+$ (see Section~\ref{sec:Dp}). We prove that this degenerate elliptic measure is in $A_\infty$ with respect to $\mu$ and then deduce the solvability of $(D)_{p,\mu}$ stated in the theorem below. This requires the notation associated with cubes $Q$ in $\RR^n$ whereby $x_Q$ and $\ell(Q)$ denote the centre and side length of $Q$, respectively, and $X_Q:=(x_Q,\ell(Q))$ denotes the corkscrew point in $\RR^{n+1}_+$ relative to $Q$. 

\begin{theorem}\label{thm:introAinfinity+Dp}
If $n\geq 2$ and the $t$-independent coefficient matrix $A$ satisfies the degenerate bound and ellipticity in~\eqref{eq:deg.def} for some constants $0<\lambda\leq\Lambda<\infty$ and an $A_2$-weight $\mu$ on $\RR^n$, then there exists $p\in(1,\infty)$ such that $(D)_{p,\mu}$ is solvable. Moreover, on each cube $Q$ in $\RR^n$, the degenerate elliptic measure $\omega:=\omega^{X_Q}\lfloor{Q}$ satisfies $\omega\in A_\infty(\mu)$ with $A_\infty(Q)$-constants that depend only on $n$, $\lambda$, $\Lambda$ and $[\mu]_{A_2}$.
\end{theorem}

In contrast to the proof of solvability in the uniformly elliptic case in \cite{HKMP1}, we avoid the need to apply the method of $\epsilon$-approximability by first establishing the Carleson measure estimate in the theorem below. This crucial estimate facilitates the main results of the paper. The connection between the Carleson measure estimate and solvability was first established in the uniformly elliptic case by Kenig, Kirchheim, Pipher and Toro in \cite{KKiPT}, and we follow their approach here, adapting it to the degenerate elliptic setting (see Lemma~\ref{28} below). In particular, the $A_\infty$-property of degenerate elliptic measure is obtained by combining the Carleson measure estimate \eqref{CME} with the notion of good $\epsilon$-coverings introduced in \cite{KKoPT}. 

\begin{theorem}\label{thm:introCME}
If $n\geq 2$ and the $t$-independent coefficient matrix $A$ satisfies the degenerate bound and ellipticity in~\eqref{eq:deg.def} for some constants $0<\lambda\leq\Lambda<\infty$ and an $A_2$-weight $\mu$ on $\RR^n$, then any solution $u\in L^\infty(\RR^{n+1}_+)$ of $\div (A \nabla u)=0$ in $\RR^{n+1}_+$ satisfies the Carleson measure estimate
\begin{equation}\label{CME}
\sup_Q \frac{1}{\mu(Q)}\int_{0}^{\ell(Q)}\int_{Q} |t\nabla u(x,t)|^2\ d\mu(x)\frac{dt}{t} \leq C \|u\|_\infty^2,
\end{equation}
where $C$ depends only on $n$, $\lambda$, $\Lambda$ and $[\mu]_{A_2}$.
\end{theorem}

Using the Carleson measure estimate in this way allows us to bypass the need to establish norm-equivalences between the non-tangential maximal function $N_*u$ and the square function $Su$ of solutions $u$, defined by
\[
(Su)(x):=\left(\iint_{\Gamma(x)} |t\nabla u(y,t)|^2 \frac{d\mu(y)}{\mu(\Delta(x,t))}\frac{dt}{t}\right)^{1/2}
\qquad \forall x\in\RR^n,
\]
where the surface ball $\Delta(x,t):=\{y\in\RR^n: |y-x|<t\}$. It was shown by Dahlberg, Jerison and Kenig in \cite{DJK}, however, that such estimates are a consequence of the $A_\infty$-property of degenerate elliptic measure, which provides the following result.

\begin{theorem}\label{thm:SNS}
If $n\geq 2$ and the $t$-independent coefficient matrix $A$ satisfies the degenerate bound and ellipticity in~\eqref{eq:deg.def} for some constants $0<\lambda\leq\Lambda<\infty$ and an $A_2$-weight $\mu$ on $\RR^n$, then any solution of $\div (A \nabla u)=0$ in $\RR^{n+1}_+$ satisfies
\[
\|Su\|_{L^p_\mu(\RR^n)} \leq C \|N_*u\|_{L^p_\mu(\RR^n)}
\qquad \forall p\in(0,\infty),
\]
and if, in addition, $u(X_0)=0$ for some $X_0\in\RR^{n+1}_+$, then
\[
\|N_*u\|_{L^p_\mu(\RR^n)} \leq C \|Su\|_{L^p_\mu(\RR^n)}
\qquad \forall p\in(0,\infty),
\]
where $C$ depends only on $X_0$, $p$, $n$, $\lambda$, $\Lambda$ and $[\mu]_{A_2}$.
\end{theorem}

The paper is structured as follows. Technical preliminaries concerning weights and degenerate elliptic operators are in Section~\ref{sec:prelims} whilst estimates for weighted maximal operators are in Section~\ref{sec:maxops}. The Carleson measure estimate in Theorem~\ref{thm:introCME} is obtained in Section~\ref{sec:CME}. The degenerate elliptic measure is constructed in Section~\ref{sec:Dp} and then the $A_\infty$-estimates in Theorem~\ref{thm:introAinfinity+Dp} are deduced as part of Theorem~\ref{thm:Ainftymain}. The square function and non-tangential maximal function estimates in Theorem~\ref{thm:SNS} are included in the more general result in Theorem~\ref{thm:DJKSNS} whilst the solvability of the Dirichlet problem in Theorem~\ref{thm:introAinfinity+Dp} is finally deduced in Theorem~\ref{thm:DpSolvmain}, where a uniqueness result is also obtained.

We state and prove our results in the upper half-space, but we note that they extend immediately to the case that the domain is the region above a Lipschitz graph, by a well-known pull-back technique which preserves the $t$-independence of the coefficients.  In turn, our results concerning the $A_\infty$-property of degenerate elliptic measure may then be extended to the case of a bounded star-like Lipschitz domain, with radially independent coefficients, by a standard localization argument using the maximum principle.

The convention is adopted whereby $C$ denotes a finite positive constant that may change from one line to the next. For $a,b\in\RR$, the notation $a\lesssim b$ means that $a\leq Cb$ whilst $a\eqsim b$ means that $a\lesssim b \lesssim a$. We write $a\lesssim_p b$ when $a\leq C b$ and we wish to emphasize that $C$ depends on a specified parameter $p$.

\section{Preliminaries}\label{sec:prelims}
We dispense with some technical preliminaries concerning general $A_p$-weights~$\mu$ for $p\in(1,\infty)$ and degenerate elliptic operators on $\RR^n$ for $n\in\N$. All cubes $Q$ and balls $B$ in $\RR^n$ are assumed to be open (except in Section~\ref{sect:dyadiccubes} where the standard dyadic cubes $S$ in $\DD(\RR^n)$ are assumed to be closed to provide genuine coverings of $\RR^n$). For $\alpha>0$, let $\alpha Q$ and $\alpha B$ denote the concentric dilates of $Q$ and $B$ respectively. For $x\in\RR^n$ and $r>0$, define the ball $B(x,r):=\{y\in\RR^n : |y-x| <r\}$. An $A_p$-weight refers to a non-negative locally integrable function $\mu$ on $\RR^n$ with the property that $[\mu]_{A_p(\RR^n)}:=\sup_Q \big(\fint_Q \mu \big) \big(\fint_Q \mu^{-1/(p-1)}\big)^{p-1}<\infty$. The measure associated with such a weight satisfies the doubling property
\begin{equation}\label{eq:D}
\mu(\alpha B) \leq [\mu]_{A_p} \alpha^{np} \mu(B)
\end{equation}
for all $\alpha\geq1$ (see, for instance, Section~1.5 in Chapter~V of~\cite{St2}). 

For an open set $\Omega \subseteq \RR^n$, the Sobolev space $W^{1,p}_{\mu}(\Omega)$ is defined as the completion, in the ambient space $L^p_\mu(\Omega)$, of the normed space of all $f\in C^\infty(\Omega)$ with finite norm
\begin{equation}\label{eq:weightedSob}
\|f\|_{W^{1,p}_\mu(\Omega)}^p
:= \int_{\Omega} |f|^p\ d\mu + \int_{\Omega} |\nabla f|^p\ d\mu < \infty.
\end{equation}
The embedding of the completion $W^{1,p}_{\mu}(\Omega)$ in $L^p_\mu(\Omega)$ relies on the $A_p$-property of the weight (to the extent that it implies both $\mu$ and $\mu^{-1/(p-1)}$ are in $L^1_\loc(\Omega)$), which ensures that if $(f_j)_j$ is a $W^{1,p}_\mu(\Omega)$-Cauchy sequence in $C^\infty(\Omega)$ converging to $0$ in $L^p_\mu(\Omega)$, then $(f_j)_j$ converges to $0$ in $W^{1,p}_\mu(\Omega)$-norm (see Section 2.1 in \cite{FKS}). Therefore, since $C^\infty(\Omega)$ is dense in $W^{1,p}_\mu(\Omega)$, the gradient extends to a bounded operator $\nabla:W^{1,p}_\mu(\Omega)\rightarrow L^p_\mu(\Omega,\RR^n)$, thereby extending \eqref{eq:weightedSob} to all $f\in W^{1,p}_\mu(\Omega)$. The Sobolev space $W^{1,p}_{0,\mu}(\Omega)$ is defined as the closure of $C_c^\infty(\Omega)$ in $W^{1,p}_\mu(\Omega)$. It can be shown that $W^{1,p}_{0,\mu}(\RR^n)=W^{1,p}_\mu(\RR^n)$ by following the proof in the unweighted case from Proposition~1 of Chapter~V in~\cite{St} but instead using Lemma~2.2 in \cite{ARR} to deduce the convergence of the regularization in $L^p_\mu(\RR^n)$. The local space $W^{1,p}_{\mu,\loc}(\Omega)$ is then defined as the set of all $f\in L^p_{\mu,\loc}(\Omega)$ such that $f\in W^{1,p}_{\mu}(\Omega')$ for all open sets $\Omega'$ with compact closure $\overline{\Omega'}\subset\Omega$ (henceforth denoted $\Omega'\subset\subset\Omega$). Finally, the weighted Sobolev and Poincar\'{e} inequalities obtained for continuous functions in Theorems~1.2 and~1.5 in~\cite{FKS} have the following immediate extensions. 

\begin{theorem}\label{thm:wSob+wPoinc}
Let $n\geq2$ and suppose that $B\subset\RR^n$ denotes a ball with radius $r(B)$. If $p\in(1,\infty)$ and $\mu$ is an $A_p$-weight on~$\RR^n$, then there exists $\delta>0$ such that
\begin{equation}\label{wSob}
\left(\fint_{B}|f|^{p(\frac{n}{n-1}+\delta)}\ d\mu \right)^{1/(p(\frac{n}{n-1}+\delta))} \lesssim r(B) \left(\fint_{B}|\nabla f|^p\ d\mu \right)^{1/p}
\end{equation}
for all $f\in W^{1,p}_{0,\mu}(B)$, and
\begin{equation}\label{wPoinc}
\left(\fint_{B}|f(x) - c_B|^p\ d\mu\right)^{1/p} \lesssim r(B) \left(\fint_{B}|\nabla f|^p\ d\mu \right)^{1/p}
\end{equation}
for all $f\in W^{1,p}_\mu(B)$ and $c_B \in \left\{\fint_{B}f\ d\mu, \fint_{B} f \right\}$, where the implicit constants depend only on $n$, $p$ and $[\mu]_{A_p}$. The estimates also hold when the ball $B$ and the radius $r(B)$ are replaced by a cube $Q$ and the sidelength $\ell(Q)$.
\end{theorem}

For $n\in\N$, constants $0<\lambda\leq\Lambda<\infty$ and an $A_2$-weight $\mu$ on $\RR^n$, let $\mathcal{E}(n,\lambda,\Lambda,\mu)$ denote the set of all $n\times n$-matrices $\mathcal{A}$ of measurable real-valued functions on $\RR^n$ satisfying the degenerate bound and ellipticity
\begin{equation}\label{eq:deg.def.bdry}
|\langle \mathcal{A}(x)\xi, \zeta \rangle| \leq \Lambda \mu(x)|\xi| |\zeta|
\quad\text{and}\quad
\langle \mathcal{A}(x)\xi, \xi \rangle \geq \lambda \mu(x)|\xi|^2
\end{equation}
for all $\xi,\zeta \in \RR^n$ and almost every $x\in \RR^n$. These properties allow us to define ${\mathcal{L}_{\mu,\Omega}:\Dom(\mathcal{L}_{\mu,\Omega})\subseteq L^2_\mu(\Omega) \rightarrow L^2_\mu(\Omega)}$ as the maximal accretive operator in $L^2_\mu(\Omega)$ associated with the bilinear form defined by
\begin{equation}\label{eq:formbdrydef}
\mathfrak{a}_\Omega(f,g) := \int_{\Omega} \langle \mathcal{A} \nabla f, \nabla g \rangle
= \int_{\Omega} \langle \tfrac{1}{\mu} \mathcal{A} \nabla f, \nabla g \rangle\ d\mu
\end{equation}
for all $f,g \in W^{1,2}_{0,\mu}(\Omega)$. The domain of $\mathcal{L}_{\mu,\Omega}$ is dense in $L^2_\mu(\Omega)$, and in particular
\[
\Dom(\mathcal{L}_{\mu,\Omega}) = \{f\in W^{1,2}_{0,\mu}(\Omega) : \sup\nolimits_{g\in C_c^\infty(\Omega)} |\mathfrak{a}_\Omega(f,g)| / \|g\|_{L^2_\mu(\Omega)} < \infty\},
\]
with
\begin{equation}\label{eq:Lbdrydef}
\int_{\Omega} (\mathcal{L}_{\mu,\Omega} f) g\ d\mu =  \mathfrak{a}_\Omega(f,g)
\end{equation}
for all $f\in\Dom(\mathcal{L}_{\mu,\Omega})$ and $g\in W^{1,2}_{0,\mu}(\Omega)$. It is equivalent to define $\mathcal{L}_{\mu,\Omega}$ as the composition $-\div_{\mu,\Omega} (\frac{1}{\mu}\mathcal{A}\nabla)$ of unbounded operators, where $-\div_{\mu,\Omega}$ is the adjoint~$\nabla^*$ of the closed densely-defined operator $\nabla:W^{1,2}_{0,\mu}(\Omega)\subseteq L^2_\mu(\Omega)\rightarrow L^2_\mu(\Omega,\RR^n)$, that is
\begin{equation}\label{eq:divbdrydef}
\int_{\Omega} (-\div_{\mu,\Omega} \mathbf{f}) g\ d\mu =  \int_{\Omega} \langle \mathbf{f}, \nabla g\rangle\ d\mu
\end{equation}
for all $\mathbf{f}\in\Dom(\div_{\mu,\Omega}):=\Dom(\nabla^*)$ and $g\in W^{1,2}_{0,\mu}(\Omega)$. In view of \eqref{eq:formbdrydef} and \eqref{eq:Lbdrydef}, we have the formal identities $\div_{\mu,\Omega}=\frac{1}{\mu}\div_\Omega \mu$ and $\mathcal{L}_{\mu,\Omega}=-\frac{1}{\mu}\div_\Omega (\mathcal{A}\nabla)$.

Now let $\Omega=Q$ for some cube $Q\subset\RR^n$ and denote the space of bounded linear functionals on $W^{1,2}_{0,\mu}(Q)$ by $W^{-1,2}_{0,\mu}(Q)$. The inclusions $W^{1,2}_{0,\mu}(Q) \subseteq L^2_\mu(Q) \subseteq W^{-1,2}_{0,\mu}(Q)$ are interpreted in the standard way by identifying $f\in L^2_\mu(Q)$ with the functional $\ell_f$ defined by $\ell_f(g):= \int_{Q} fg\ d\mu$ for all $g\in W^{1,2}_{0,\mu}(Q)$. Thus, setting
\[
\mathcal{L}_{\mu,Q} f(g):=\mathfrak{a}_Q(f,g)
\qquad{\textrm{and}}\qquad
-\div_{\mu,Q}\mathbf{f}(g):=\int_{Q} \langle \mathbf{f}, \nabla g\rangle\ d\mu
\]
for all $f,g\in W^{1,2}_{0,\mu}(Q)$ and $\mathbf{f}\in L^2(Q,\RR^n)$, we obtain an extension of $\mathcal{L}_{\mu,Q}$ from \eqref{eq:Lbdrydef} to a bounded invertible operator from $W^{1,2}_{0,\mu}(Q)$ onto $W^{-1,2}_{0,\mu}(Q)$, and an extension of $\div_{\mu,Q}$ from \eqref{eq:divbdrydef} to a bounded operator from $L^2_\mu(Q)$ into $W^{-1,2}_{0,\mu}(Q)$. The surjectivity of $\mathcal{L}_{\mu,Q}$ relies on \eqref{wSob} and the Lax--Milgram Theorem. These definitions imply that 
\[
\|\nabla \mathcal{L}^{-1}_{\mu,Q} \div_{\mu,Q} \mathbf{f}\|_{L^2_\mu(Q,\RR^n)} \lesssim \|\mathbf{f}\|_{L^2_\mu(Q,\RR^n)}
\]
for all $\mathbf{f}\in L^2_\mu(Q,\RR^n)$. The topological direct sum or $W^{1,2}_{0,\mu}(Q)$-Hodge decomposition
\begin{equation}\label{eq:W12Hodge}
L^2_\mu(Q,\RR^n) = \{\tfrac{1}{\mu}\mathcal{A}\nabla g : g \in W^{1,2}_{0,\mu}(Q)\} \oplus \{\mathbf{h} \in L^2_\mu(Q,\RR^n) : \div_{\mu,Q} \mathbf{h} = 0\}
\end{equation}
follows by writing $\mathbf{f}=-\tfrac{1}{\mu}\mathcal{A}\nabla \mathcal{L}_{\mu,Q}^{-1} \div_{\mu,Q} \mathbf{f} + (\mathbf{f} + \tfrac{1}{\mu}\mathcal{A}\nabla \mathcal{L}_{\mu,Q}^{-1} \div_{\mu,Q} \mathbf{f}) =: \tfrac{1}{\mu}\mathcal{A}\nabla g + \mathbf{h}$, since then $\div_{\mu,Q} \mathbf{h} = \div_{\mu,Q} \mathbf{f} -\mathcal{L}_{\mu,Q}\mathcal{L}_{\mu,Q}^{-1}\div_{\mu,Q} \mathbf{f}=0$. This decomposition also extends to $L^p_\mu(Q,\RR^n)$ for all $p \in [2,2+\epsilon)$ and some $\epsilon> 0$ by recent work of Le in~\cite{L}, although we do not need it here.

Now let $\Omega=\RR^n$ and consider $\div_\mu:=\div_{\mu,\RR^n}$ as in~\eqref{eq:divbdrydef} so $\mathcal{L}_\mu:=-\div_\mu(\frac{1}{\mu}\mathcal{A}\nabla)$ is maximal accretive, thus having a maximal accretive square root $\mathcal{L}_\mu^{1/2}$, in~$L^2_\mu(\RR^n)$. The solution of the Kato square root problem in \cite{AHLMT} was recently extended to degenerate elliptic equations by Cruz-Uribe and Rios in~\cite{CuR}. This shows that $\|\mathcal{L}_\mu^{1/2} f\|_{L^2_\mu(\RR^n)} \eqsim \|\nabla f\|_{L^2_\mu(\RR^n,\RR^n)}$ for all $f\in W^{1,2}_\mu(\RR^n)$, hence $\Dom(\mathcal{L}_\mu^{1/2}) = W^{1,2}_\mu(\RR^n)$.

The operator $\mathcal{L}_\mu$ is also injective and type-$S_{\omega+}$ in $L^2_\mu(\RR^n)$ for some $\omega\in(0,\pi/2)$, so it has a bounded $H^\infty(S_{\theta+}^o)$-functional calculus in $L^2_\mu(\RR^n)$ for each $\theta\in(\omega,\pi)$, where $S_{\theta+}^o:=\{z\in\CC\setminus\{0\} : |\arg z| < \theta\}$. See Section~2.2 of \cite{A} for the uniformly elliptic case and Theorems~F and~G in \cite{ADM} for the general theory. An equivalent property is the validity of the quadratic estimate
\begin{equation}\label{eq:QEest}
\int_0^\infty \|\psi(t\mathcal{L_\mu})f\|_{L^2_\mu(\RR^n)}^2 \frac{dt}{t} \eqsim \|f\|_{L^2_\mu(\RR^n)}^2
\qquad\forall f\in L^2_\mu(\RR^n)
\end{equation}
for each holomorphic $\psi$ on $S_{\theta+}^o$ satisfying $|\psi(z)| \lesssim \min\{|z|^\alpha,|z|^{-\beta}\}$ for some $\alpha,\beta>0$, where the bounded operator $\psi(t\mathcal{L_\mu})$ on $L^2_\mu(\RR^n)$ is defined by a Cauchy integral. More generally, the relationship between bounded holomorphic functional calculi and quadratic estimates is developed in the seminal articles \cite{Mc} and \cite{CDMY}.

The functional calculus then defines a bounded operator $\varphi(\mathcal{L}_\mu)$ on $L^2_\mu(\RR^n)$ for each bounded holomorphic function $\varphi$ on~$S_{\theta+}^o$ and $\|\varphi(\mathcal{L}_\mu)\|_{L^2_\mu(\RR^n)\rightarrow L^2_\mu(\RR^n)} \lesssim_\theta \|\varphi\|_\infty$. Another consequence is that $-\mathcal{L}_\mu$ generates a holomorphic contraction semigroup $(e^{-\zeta\mathcal{L}_\mu})_{\zeta\in S^o_{\pi/2-\omega}\cup\{0\}}$ on $L^2_\mu(\RR^n)$, thus $e^{-t\mathcal{L_\mu}} f \in \Dom(\mathcal{L}_\mu)$ and $\partial_t(e^{-t\mathcal{L_\mu}} f) = \mathcal{L}_\mu e^{-t\mathcal{L_\mu}} f $ for all $f\in L^2_\mu(\RR^n)$ and $t>0$. The functional calculus also extends to define an unbounded operator $\phi(\mathcal{L}_\mu)$ on $L^2_\mu(\RR^n)$ for each holomorphic function $\phi$ on $S_{\theta+}^o$ satisfying $|\phi(z)| \lesssim \max\{|z|^\alpha,|z|^{-\beta}\}$ for some $\alpha,\beta>0$, but the algebra homomorphism property of the functional calculus ($\phi_1(\mathcal{L}_\mu)\phi_2(\mathcal{L}_\mu)=(\phi_1\phi_2)(\mathcal{L}_\mu)$) must then be interpreted in the sense of unbounded linear operators. This allows us to interpret both the semigroup and the square root of $\mathcal{L}_\mu$ in terms of the functional calculus in order to justify some otherwise formal manipulations, beginning with~\eqref{eq:FCident} in the proof of the following corollary of the solution of the Kato problem in~\cite{CuR}.

\begin{theorem}\label{VSFbound}
Let $n \geq 1$ and suppose that $\mathcal{A}\in\mathcal{E}(n,\lambda,\Lambda,\mu)$ for some constants $0<\lambda\leq\Lambda<\infty$ and an $A_2$-weight $\mu$ on $\RR^n$. The operator $\mathcal{L}_\mu:=-\div_\mu (\frac{1}{\mu}\mathcal{A}\nabla)$ satisfies
\begin{equation}\label{eq:VSF1}
\int_0^\infty \|t \mathcal{L}_{\mu} e^{-t^2\mathcal{L}_{\mu}} f\|_{L^2_\mu(\RR^n)}^2 \frac{dt}{t}
\eqsim \|\nabla f\|_{L^2_\mu(\RR^n,\RR^n)}^2
\end{equation}
and
\begin{equation}\label{eq:VSF2}
\int_0^\infty \|t^2\nabla_{x,t} \mathcal{L}_{\mu} e^{-t^2\mathcal{L}_{\mu}} f\|_{L^2_\mu(\RR^n,\RR^{n+1})}^2 \frac{dt}{t}
\lesssim \|\nabla f\|_{L^2_\mu(\RR^n,\RR^n)}^2
\end{equation}
for all $f \in W^{1,2}_{\mu}(\RR^n)$, where the implicit constants depend only on $n$, $\lambda$, $\Lambda$ and~$[\mu]_{A_2}$.
\end{theorem}

\begin{proof}
The functional calculus of $\mathcal{L}_\mu$ justifies the identity 
\begin{equation}\label{eq:FCident}
\mathcal{L}_\mu e^{-t^2\mathcal{L}_\mu} f
= \mathcal{L}_\mu^{1/2} e^{-t^2\mathcal{L}_\mu} \mathcal{L}_\mu^{1/2} f 
= e^{-(t^2/2)\mathcal{L}_\mu} \mathcal{L}_\mu e^{-(t^2/2)\mathcal{L}_\mu} f
\end{equation}
for all $f\in \Dom(\mathcal{L}_\mu^{1/2})$ and $t>0$. The first equality in \eqref{eq:FCident}, the quadratic estimate in~\eqref{eq:QEest} and the solution of the Kato problem in \cite{CuR} imply that
\begin{align*}
\int_0^\infty \|t \mathcal{L}_{\mu} e^{-t^2\mathcal{L}_{\mu}} f\|_{L^2_\mu(\RR^n)}^2 \frac{dt}{t}
&= \int_0^\infty \|(\tau\mathcal{L}_{\mu})^{1/2} e^{-\tau\mathcal{L}_{\mu}} \mathcal{L}_{\mu}^{1/2}  f\|_{L^2_\mu(\RR^n)}^2 \frac{d\tau}{\tau} \\
&\eqsim \|\mathcal{L}_{\mu}^{1/2}  f\|_{L^2_\mu(\RR^n)}^2 \\
&\eqsim \|\nabla f\|_{L^2_\mu(\RR^n,\RR^n)}^2
\end{align*}
for all $f\in \Dom(\mathcal{L}_\mu^{1/2}) = W^{1,2}_\mu(\RR^n)$, which proves \eqref{eq:VSF1}.

The bounded $H^\infty(S_{\theta+}^o)$-functional calculus of $\mathcal{L}_\mu$ implies the uniform estimate
\begin{align*}
\|t\nabla_{x,t} e^{-t^2\mathcal{L}_{\mu}} g\|_{L^2_\mu(\RR^n,\RR^{n+1})}^2
&= \|t \partial_t e^{-t^2\mathcal{L}_{\mu}} g\|_{L^2_\mu(\RR^n)}^2
+ \|t \nabla_x e^{-t^2\mathcal{L}_\mu}\|_{L^2_\mu(\RR^n,\RR^n)}^2 \\
&\lesssim \|t^2\mathcal{L}_\mu e^{-t^2\mathcal{L}_\mu}g\|_{L^2_\mu(\RR^n)}^2
+  \int_{\RR^n} t^2 \langle \mathcal{A}\nabla_x e^{-t^2\mathcal{L}_\mu}g, \nabla_x e^{-t^2\mathcal{L}_\mu}g\rangle \\
&\lesssim \|g\|_{L^2_\mu(\RR^n)}^2
+ \|t^2\mathcal{L}_\mu e^{-t^2\mathcal{L}_\mu}g\|_{L^2_\mu(\RR^n)} \|e^{-t^2\mathcal{L}_\mu}g\|_{L^2_\mu(\RR^n)} \\
&\lesssim \|g\|_{L^2_\mu(\RR^n)}^2
\end{align*}
for all $g\in L^2_\mu(\RR^n)$ and $t>0$. Thus, the second equality in \eqref{eq:FCident} and the vertical square function estimate in \eqref{eq:VSF1}, which we have already proved, imply that
\[
\int_0^\infty \|t^2\nabla_{x,t} \mathcal{L}_{\mu} e^{-t^2\mathcal{L}_{\mu}} f\|_{L^2_\mu(\RR^n,\RR^{n+1})}^2 \frac{dt}{t}
\lesssim \int_0^\infty \|t \mathcal{L}_{\mu} e^{-(t^2/2)\mathcal{L}_{\mu}} f\|_{L^2_\mu(\RR^n)}^2 \frac{dt}{t} \lesssim \|\nabla f\|_{L^2_\mu(\RR^n,\RR^n)}^2
\]
for all $f \in W^{1,2}_\mu(\RR^n)$, which proves \eqref{eq:VSF2}.
\end{proof}

Now let us return to the case when $\Omega \subseteq \RR^n$ is an arbitrary open set and suppose that $\mathbf{f}:\Omega\rightarrow \RR^n$ is a measurable function for which $\frac{1}{\mu}\mathbf{f} \in L^\infty(\Omega)$. A \textit{solution} of the inhomogeneous equation $\div (\mathcal{A} \nabla u)= \div \mathbf{f}$ in $\Omega \subseteq \RR^n$ refers to any function $u\in W^{1,2}_{\mu,\loc}(\Omega)$ such that $\int_{\RR^n} \langle \mathcal{A} \nabla u - \mathbf{f}, \nabla \Phi \rangle = 0$ for all $\Phi\in C_c^\infty(\Omega)$. All solutions $u$ of the homogeneous equation $\div (\mathcal{A} \nabla u)=0$ in $\Omega$ are locally bounded and H\"{o}lder continuous in the sense that 
\begin{equation}\label{eq:M}
\|u\|_{L^\infty(B)} \lesssim \left(\fint_{2B} |u|^2\ d\mu\right)^{1/2} 
\end{equation}
and there exists $\alpha>0$ such that
\begin{equation}\label{eq:dGN}
|u(x)-u(y)| \lesssim \left(\frac{|x-y|}{r(B)}\right)^{\alpha}\left(\fint_{2B}|u|^{2} \ d\mu \right)^{1/2} \qquad \forall x,y\in B,
\end{equation}
and if, in addition, $u\geq0$ almost everywhere on $\Omega$, there is the Harnack inequality
\begin{equation}\label{eq:H}
\sup_{B} u \lesssim \inf_{B} u,
\end{equation}
for all balls $B$ of radius $r(B)$ such that $2B \subseteq \Omega$, where $\alpha$ and the implicit constants depend only on $n$, $\lambda$, $\Lambda$ and $[\mu]_{A_2}$. These properties follow from Corollary 2.3.4, Lemma 2.3.5 and Theorem~2.3.12 in~\cite{FKS} by observing that the proofs do not use the  assumption therein that $\mathcal{A}$ is symmetric. The estimates also hold when the balls $B$ are replaced by (open) cubes $Q$, and also when the dilate $2B$ is replaced by $C_0B$ for any $C_0>1$, provided the implicit constants are understood to depend on $C_0$.  

The following local boundedness estimate for solutions of the inhomogeneous equation is needed in Lemma~\ref{lem:CoV}, although only for $p=2$. This is a simpler version of Theorem~8.17 in~\cite{GT}, which we have adapted to degenerate elliptic equations. In fact, the result for $p\geq2$ is already proven in \cite{FKS} by combining Corollary~2.3.4 with estimates~(2.3.7) and (2.3.13) therein. The proof is included here for the readers convenience and since it implies \eqref{eq:M} as a specical case, which in turn is the well-known starting point for establishing \eqref{eq:dGN}.

\begin{theorem}\label{MoEsttype}
Let $n\geq2$ and suppose that $\mathcal{A}\in\mathcal{E}(n,\lambda,\Lambda,\mu)$ for some  constants $0<\lambda\leq\Lambda<\infty$ and an $A_2$-weight $\mu$ on $\RR^n$. Let $\Omega\subseteq\RR^n$ denote an open set and suppose that $\mathbf{f}:\Omega\rightarrow \RR^n$ is a measurable function such that $\frac{1}{\mu}\mathbf{f} \in L^\infty(\Omega)$. If $p\in(1,\infty)$ and $\div (\mathcal{A} \nabla u) = \div \mathbf{f}$ in $\Omega$, then
\begin{equation}\label{eq:MoserInh}
\|u\|_{L^\infty(B)} \lesssim \left(\fint_{2B} |u|^p\ d\mu\right)^{1/p} 
+ r(B) \|\tfrac{1}{\mu} \mathbf{f} \|_{L^\infty(\Omega)}
\end{equation}
for all balls $B$ of radius $r(B)>0$ such that $2B \subseteq \Omega$, where the implicit constant depends only on $p$, $n$, $\lambda$, $\Lambda$ and $[\mu]_{A_2}$.
\end{theorem}

\begin{proof} 
Suppose that $\div (\mathcal{A} \nabla u )= \div \mathbf{f}$ in $\Omega$ and consider a ball $B$ such that $2B \subseteq \Omega$. First, assume that $u$ is non-negative and in $L^\infty(2B)$. Let~$\epsilon>0$, set $k=r(B)\|\frac{1}{\mu}\mathbf{f}\|_{L^\infty(\Omega)}$ and $\ube:=u + k + \epsilon$. Let $B_r$ denote the ball concentric to $B$ with radius $r>0$ and recall the index $\delta>0$ from the Sobolev inequality in Theorem~\ref{thm:wSob+wPoinc}. We claim that if $\gamma\in[p,\infty)$ and $r(B)\leq r_1<r_2\leq2r(B)$, then
\begin{equation}\label{eq:MIclaim}
\left(\fint_{B_{r_1}} \ube^{\gamma(\frac{n}{n-1}+\delta)} \ d\mu\right)^{1/(\gamma(\frac{n}{n-1}+\delta))}
\lesssim \left(\gamma \frac{r_1}{r_2-r_1}\right)^{2/\gamma} \left(\fint_{B_{r_2}} \ube^\gamma\ d\mu \right)^{1/\gamma},
\end{equation}
where the implicit constant depends only on $p$, $n$, $\lambda$, $\Lambda$ and $[\mu]_{A_2}$. To prove \eqref{eq:MIclaim}, fix $\eta\in C^\infty_c(\Omega)$ such that $\eta:\Omega\rightarrow[0,1]$, $\eta\equiv 1$ on $B_{r_1}$, $\eta\equiv 0$ on $\Omega\setminus B_{r_2}$ and $\|\nabla\eta\|_\infty \leq 2/(r_2-r_1)$. Set $\beta:=\gamma-1$ and $\nu:=\eta^2\ube^\beta$. Note that $\nu\in W^{1,2}_{0,\mu}(\Omega)$ with
\[
\nabla v = 2\eta \nabla \eta \ube^\beta + \beta \eta^2 \ube^{\beta-1} \nabla u,
\]
since $0<\epsilon \leq \ube(x) \leq \|u\|_{L^\infty(2B)} + k + \epsilon<\infty$ for almost every $x\in 2B$, thus
\[
\int_{\RR^n} \langle \mathcal{A} \nabla u - \mathbf{f}, 2\eta \nabla \eta \ube^\beta \rangle
 = - \int_{\RR^n} \langle \mathcal{A} \nabla u - \mathbf{f}, \beta \eta^2 \ube^{\beta-1} \nabla u \rangle.
\]
We then use this identity and Cauchy's inequality with $\sigma>0$ to obtain
\begin{align*}
\int_{\RR^n} \eta^2 \ube^{\beta-1} |\nabla u|^2 \ d\mu
&\lesssim_\lambda \int_{\RR^n} \eta^2 \ube^{\beta-1} \langle \mathcal{A}\nabla u, \nabla u\rangle \\
&= -2\beta^{-1} \int_{\RR^n} \eta \ube^\beta \langle \mathcal{A} \nabla u - \mathbf{f}, \nabla \eta \rangle
+\!\int_{\RR^n} \eta^2 \ube^{\beta-1} \langle \mathbf{f}, \nabla u\rangle \\
&\lesssim_\Lambda (p\!-\!1)^{-1}\! \int_{\RR^n} \eta \ube^\beta (|\nabla u| + |\tfrac{1}{\mu}\mathbf{f}|) |\nabla \eta|\ d\mu
+\!\int_{\RR^n} \eta^2 \ube^{\beta-1} |\tfrac{1}{\mu}\mathbf{f}| |\nabla u| \ d\mu \\
&\lesssim_p \sigma \int_{\RR^n} \eta^2 \ube^{\beta-1} |\nabla u|^2\ d\mu
+ \sigma^{-1} \int_{\RR^n} \ube^{\beta+1} |\nabla \eta|^2\ d\mu \\
&\qquad + \int_{\RR^n} \ube^{\beta+1} |\nabla \eta|^2\ d\mu
+ \int_{\RR^n} (\eta/r(B))^2 \ube^{\beta+1} \ d\mu \\
&\qquad + \sigma \int_{\RR^n} \eta^2 \ube^{\beta-1} |\nabla u|^2 \ d\mu 
+ \sigma^{-1} \int_{\RR^n} (\eta/r(B))^2 \ube^{\beta+1} \ d\mu,
\end{align*}
where in the second inequality we used the assumption that $\beta:=\gamma-1\geq p-1$ and in the final inequality we used the fact that $|\tfrac{1}{\mu}\mathbf{f}| \leq k/r(B) \leq \ube /r(B)$ on $\Omega$. Next, choose $\sigma>0$ small enough, depending only on $p$, $\lambda$ and $\Lambda$, to deduce that
\[
\int_{B_{r_1}} \ube^{\beta-1} |\nabla u|^2\ d\mu
\lesssim_{p,\lambda,\Lambda} \int_{\RR^n} \ube^{\beta+1} (|\nabla \eta|^2\
+ (\eta/r(B))^2)\ d\mu
\lesssim \frac{1}{(r_2-r_1)^2} \int_{B_{r_2}} \ube^{\beta+1}\ d\mu,
\]
where in the final inequality we used the fact that $r(B)\geq r_2-r_1$. Now combine this estimate with the Sobolev inequality \eqref{wSob} and recall that $\beta:=\gamma-1$ to obtain
\begin{align*}
\left(\fint_{B_{r_1}} \ube^{\gamma(\frac{n}{n-1}+\delta)}\ d\mu\right)^{1/(\frac{n}{n-1}+\delta)}
&\lesssim r_1^2 \fint_{B_{r_1}} |\nabla(\ube^{(\beta+1)/2})|^2\ d\mu \\
&\lesssim ((\beta +1) r_1)^2 \fint_{B_{r_1}} \ube^{\beta-1} |\nabla u|^2\ d\mu \\
&\lesssim \left(\gamma \frac{r_1}{r_2-r_1}\right)^{2} \fint_{B_{r_2}} \ube^{\gamma}\ d\mu,
\end{align*}
where the implicit constants depend only on $p$, $n$, $\lambda$, $\Lambda$ and $[\mu]_{A_2}$, proving~\eqref{eq:MIclaim}.

We now apply the Moser iteration technique to prove \eqref{eq:MoserInh}. Set $\chi:=\frac{n}{n-1}+\delta$ and define $\Phi(q,r):= \left(\fint_{B_r} \ube^q\ d\mu \right)^{1/q}$ for $q,r>0$. Estimate~\eqref{eq:MIclaim} implies that
\[
\Phi(\gamma\chi,r_1) \leq \left(C\gamma \frac{r_1}{r_2-r_1}\right)^{2/\gamma} \Phi(\gamma,r_2)
\]
where $C$ depends only on $p$, $n$, $\lambda$, $\Lambda$ and $[\mu]_{A_2}$, and it follows by induction that
\[
\Phi(p\chi^m,(1+2^{-m})r(B)) \leq (4Cp)^{\frac{2}{p}\sum_{k=0}^{m-1} \chi^{-k}} (2\chi)^{\frac{2}{p}\sum_{k=0}^{m-1} k\chi^{-k}} \Phi(p,2r(B)) \lesssim \Phi(p,2r(B))
\]
for all $m\in\N$. This shows that
\[
\|\ube\|_{L^\infty(B)} = \lim_{m\rightarrow\infty} \Phi(p\chi^m,r(B))
\lesssim \Phi(p,2r(B))
= \left(\fint_{2B} \ube^p\ d\mu\right)^{1/p}
\]
and therefore
\[
\|u\|_{L^\infty(B)} \leq \|\ube\|_{L^\infty(B)}
\lesssim \left(\fint_{2B} \ube^p\ d\mu\right)^{1/p}
\lesssim \left(\fint_{2B} u^p\ d\mu\right)^{1/p} + r(B)\|\tfrac{1}{\mu}\mathbf{f}\|_{L^\infty(\Omega)} + \epsilon
\]
for all $\epsilon>0$, which implies \eqref{eq:MoserInh}.

Finally, it remains to remove the assumption that $u$ is non-negative and bounded. This is achieved by setting $\ube:=\max\{u,0\}+k+\epsilon$ and $\ube:=-\min\{u,0\}+k+\epsilon$ respectively and in each case adjusting the proof above to incorporate the truncated test function $\nu := \eta^2 h_N(\ube) \ube$, where 
\[
h_N(x) :=
\begin{cases}
x^{\beta-1}, & x \leq N+k+\epsilon, \\
(N+k+\epsilon)^{\beta-1}, & x > N+k+\epsilon.
\end{cases}
\]
We leave the standard details to the reader.
\end{proof}

The following self-improvement property for Carleson measures will be used in conjunction with the local H\"{o}lder continuity estimate for solutions in~\eqref{eq:dGN}. The result is proved in the unweighted case in Lemma~2.14 in~\cite{AHLT}. In that proof, the Lebesgue measure on $\RR^n$ can in fact be replaced by any doubling measure, since the Whitney decomposition of open sets can be adapted to any such measure (see, for instance, Lemma~2 in Chapter~I of~\cite{St2}). The result below then follows.

\begin{lemma}\label{JNforCar} Let $n\geq1$ and suppose that $\mu$ is an $A_2$-weight on $\RR^n$. Let $\alpha$, $\beta_0>0$ and suppose that $(v_t)_{t>0}$ is a collection of H\"{o}lder continuous functions on a cube $Q\subset\RR^n$ satisfying 
\begin{equation*}
0 \leq v_t(x) \leq \beta_0
\quad{and}\quad
|v_t(x) - v_t(y)| \leq \beta_0\left(\frac{|x - y|}{t}\right)^{\alpha}
\end{equation*}
for all $x, y \in Q$. If there exists $\eta \in (0,1]$, $\beta>0$ and, for each cube $Q' \subseteq Q$, a measurable set $F' \subset Q'$ such that
\begin{equation*}
\mu(F') \geq \eta \mu(Q')
\quad{and}\quad
\frac{1}{\mu(Q')} \int_{0}^{l(Q')} \int_{F'} v_t(x)\ d\mu(x) \frac{dt}{t} \leq \beta,
\end{equation*}
then
\begin{equation*}
\frac{1}{\mu(Q)}\int_{0}^{\ell(Q)} \int_{Q} v_{t}(x)\ d\mu(x) \frac{dt}{t} \lesssim_{\alpha,\eta} \beta + \beta_0,
\end{equation*}
where the implicit constant depends only on $\alpha$, $\eta$, $n$ and $[\mu]_{A_2}$.
\end{lemma}

\section{Estimates for Maximal Operators}\label{sec:maxops}
We obtain estimates for a variety of maximal operators ($M_{\mu}$, $D_{*,\mu}$, $N_*^{\eta}$ and $\widetilde{N}_{*,\mu}^{\eta}$) adapted to an $A_2$-weight~$\mu$ and degenerate elliptic operators $\mathcal{L}_\mu:=-\div_\mu(\frac{1}{\mu}\mathcal{A}\nabla)$ on $\RR^n$ for $n\geq 2$. These will be used to prove the Carleson measure estimate from Theorem~\ref{thm:introCME} in Section~\ref{sec:CME}. We first define the maximal operators $M_{\mu}$ and $D_{*,\mu}$ by
\begin{align*}
M_{\mu}f(x) &:= \sup_{r>0}\fint_{B(x,r)}|f(y)|\ d\mu(y),  \\
D_{*,\mu}g(x) &:= \sup_{r>0}\left( \fint_{B(x,r)}\left( \frac{|g(x) -g(y)|}{|x-y|}\right)^2 d\mu(y)\right)^{1/2}
\end{align*}
for all $f\in L^1_{\mu,\loc}(\RR^n)$, $g\in W^{1,2}_{\mu,\loc}(\RR^n)$ and $x\in\RR^n$. The usual unweighted and centred Hardy--Littlewood maximal operator is abbreviated by $M$. The maximal operator $M_\mu$ is bounded on $L^p_\mu(\RR^n)$ for all $p\in(1,\infty)$ and satisfies the weak-type estimate
\begin{equation}\label{eq:wkdef}
\mu\big(\{x\in\RR^n : |M_\mu f(x)| > \kappa\}\big) \lesssim \kappa^{-1} \|f\|_{L^1_\mu(\RR^n)}
\qquad \forall \kappa >0
\end{equation}
for all $f\in L^1_\mu(\RR^n)$ (see, for instance, Theorem 1 in Chapter~I of~\cite{St2}). There is also the following weak-type estimate for the maximal operator $D_{*,\mu}$.

\begin{lemma}\label{lem:Dpbound} Let $n\geq 2$. If $\mu$ is an $A_2$-weight on $\RR^n$, then
\begin{align}\label{Dptilde}
\mu\big(\{x\in\RR^n : |D_{*,\mu} f(x)| > \kappa\}\big) \lesssim \kappa^{-2} \|\nabla f\|_{L^2_\mu(\RR^n,\RR^n)}^2
\qquad \forall \kappa >0
\end{align}
for all $f\in W^{1,2}_\mu(\RR^n)$, where the implicit constant depends only on $n$ and~$[\mu]_{A_2}$.
\end{lemma}
\begin{proof}
If $f\in C_c^\infty(\RR^n)$, then a version of Morrey's inequality (see, for instance, Theorem~3.5.2 in~\cite{Mo}) shows that
\[
\frac{|f(x) - f(y)|}{|x-y|} \lesssim M(\nabla f)(x) + M(\nabla f)(y)
\]
for almost every $x,y\in\RR^n$, hence
\[
D_{*,\mu}f(x) \lesssim M(\nabla f)(x) + \left(M_\mu[M(\nabla f)]^2(x)\right)^{1/2}.
\]
Estimate~\eqref{Dptilde} then follows from the weak-type bound for $M_\mu$ in \eqref{eq:wkdef}, the fact that $M$ is bounded on $L^2_\mu(\RR^n)$ (see, for instance, Theorem~1 in Chapter~V of \cite{St2}) and the density of $C_c^\infty(\RR^n)$ in $W^{1,2}_\mu(\RR^n)$.
\end{proof}

We now define the non-tangential maximal operators $N_*^{\eta}$ and $\widetilde{N}_{*,\mu}^{\eta}$, for $\eta>0$, by
\begin{align*}
N_*^{\eta}u(x) := \sup_{(y,t)\in \Gamma_\eta(x)} |u(y,t)|,\qquad
\widetilde{N}_{*,\mu}^{\eta}v(x) := \sup_{(y,t)\in \Gamma_\eta(x)} \left(\fint_{B(y,\eta t)} |v(z,t)|^2\ d\mu(z) \right)^{1/2}
\end{align*}
for all measurable functions $u, v$ on $\RR^{n+1}_+$ (such that $v(\cdot,t)\in L^2_{\mu,\loc}(\RR^n)$ for a.e. $t>0$) and $x\in\RR^n$, where $\Gamma_{\eta}(x):=\{(y,t)\in\RR^{n+1}_+ : |y-x|<\eta t \}$ is the conical non-tangential approach region in $\RR^{n+1}_+$ with vertex at~$x$ and aperture~$\eta$.

Now suppose that $\mathcal{A} \in \mathcal{E}(n,\lambda,\Lambda,\mu)$, as defined by \eqref{eq:deg.def.bdry}. In particular, since $\mathcal{A}$ has real-valued coefficients, there exists an integral kernel $W_t(x,y)$ such that
\begin{equation}\label{eq:G1}
e^{-t\mathcal{L}_{\mu}}f(x) = \int_{\RR^n} W_t(x,y) f(y)\ d\mu(y),
\end{equation}
for all $f\in L^2_\mu(\RR^n)$, and there exists constants $C_1, C_2 >0$ such that
\begin{equation}\label{eq:G2}
|W_t(x,y)| \leq \frac{C_1}{\mu(B(x,\sqrt{t}))} \exp\left(-C_2\frac{|x-y|^2}{t}\right)
\end{equation}
for all $t>0$ and $x,y\in\RR^n$. This was proved by Cruz-Uribe and Rios for $f\in C_c^
\infty(\RR^n)$ under the assumption that $\mathcal{A}$ is symmetric (see Theorem~1 and Remark~3 in~\cite{CuRc}). The symmetry assumption can be removed, however, by following their proof and applying the  Harnack inequality for degenerate parabolic equations obtained by Ishige in Theorem~A of \cite{I}, which does not require symmetric coefficients, instead of the version recorded in Proposition~3.8 of~\cite{CuRg}. The results also extend to $f\in L^2_\mu(\RR^n)$ by density, Schur's Lemma and the doubling property of $\mu$. 

We now consider the semigroup generated by $\mathcal{L}_\mu:=-\div_\mu(\frac{1}{\mu}\mathcal{A}\nabla)$ with elliptic homogeneity ($t$ replaced by $t^2$) and denoted by $\mathcal{P}_t := e^{-t^2\mathcal{L}_{\mu}}$ in the estimates below.

\begin{lemma}\label{Lpsegr}
Let $n\geq2$ and suppose that $\mathcal{A}\in\mathcal{E}(n,\lambda,\Lambda,\mu)$ for some constants $0<\lambda\leq\Lambda<\infty$ and an $A_2$-weight $\mu$ on $\RR^n$. Let $p\in(1,\infty)$ and suppose that $\mu$ is also an $A_p$-weight on $\RR^n$. If $x\in\RR^n$, $\eta>0$ and $\alpha\geq1$, then
\begin{equation}\label{NonA}
\sup_{(y,t)\in\Gamma_\eta(x)} |(\eta t)^{-1}[\PP_{\eta t}(f - c_{B(x,\alpha \eta t)})](y)|^2
\lesssim_\alpha [M_{\mu}(|\nabla f|^p)(x)]^{2/p}
\end{equation}
for all $f\in W^{1,p}_\mu(\RR^n)$ and $c_{B(x,\alpha \eta t)} \in \left\{\fint_{B(x,\alpha \eta t)} f\ d\mu, \fint_{B(x,\alpha \eta t)} f \right\}$, and
\begin{align}
\label{NonB}
|N_*^\eta(\partial_t\mathcal{P}_t f)(x)|^2
&\lesssim_\eta  [M_{\mu}(|\nabla f|^p)(x)]^{2/p},\\
\label{NonBb}
|\eta^{-1}N_*^\eta(\partial_t\mathcal{P}_{\eta t} f)(x)|^2
&\lesssim [M_{\mu}(|\nabla f|^p)(x)]^{2/p},\\
\label{NonC}
|\widetilde{N}_{*,\mu}^\eta(\nabla_x\mathcal{P}_{\eta t}f)(x)|^2
&\lesssim M_{\mu}\big([M_{\mu}(|\nabla f|^p)]^{2/p}\big)(x) + M_{\mu}(|\nabla f|^2)(x),
\end{align}
for all $f\in W^{1,2}_\mu(\RR^n) \cap W^{1,p}_{\mu,\loc}(\RR^n)$, where the implicit constants depend only on $n$, $\lambda$, $\Lambda$, p, $[\mu]_{A_2}$ and $[\mu]_{A_p}$, as well as on $\alpha$ in \eqref{NonA} and on $\eta$ in \eqref{NonB}.
\end{lemma}

\begin{proof}
Let $x\in\RR^n$, $(y,t)\in \Gamma^\eta (x)$, $f\in W^{1,2}_\mu(\RR^n) \cap W^{1,p}_{\mu,\loc}(\RR^n)$, $f_{B(x,t)} := \fint_{B(x,t)} f$ and $\tilde{f}_{B(x,t)} := \fint_{B(x,t)} f\ d\mu$. To prove~\eqref{NonA}, it suffices to assume that $\eta=1$ and $\alpha\geq 1$. We set $C_0(t):= B(x,\alpha t)$ and define the dyadic annulus $C_j(t) := B(x, 2^j \alpha t)\setminus B(x, 2^{j-1} \alpha t)$ for all $j\in\N$. The Gaussian kernel estimates in~\eqref{eq:G1} and \eqref{eq:G2} imply that
\begin{multline*}
|t^{-1}[\PP_t(f - f_{B(x,\alpha t)})](y)|  = t^{-1}\left|\int_{\RR^n} W_{t^2}(y,z)[f(z)- f_{B(x,\alpha t)}]\ d\mu(z) \right|\\
\leq \sum_{j=0}^{\infty}  t^{-1} \frac{C_1}{\mu(B(y,t))} \int_{C_j(t)} \exp\left({-C_2\frac{|y-z|^2}{t^2}}\right)|f (z)- f_{B(x,\alpha t)}|\ d\mu(z)
=: \sum_{j=0}^\infty I_j.
\end{multline*}
To estimate $I_0$, note that $B(x,\alpha t)\subseteq B(y,(1+\alpha)t)$ and apply the doubling property of~$\mu$, followed by the $L^p_\mu$-Poincar\'{e} inequality in~\eqref{wPoinc} with $c_B=\fint_{B(x,\alpha t)} f$, to obtain
\begin{align*}
I_0
\lesssim_\alpha t^{-1} \fint_{B(x,\alpha t)} |f (z)- f_{B(x,\alpha t)}|\ d\mu(z)
\lesssim \left(\fint_{B(x,\alpha t)} |\nabla f|^p\ d\mu \right)^{1/p} \!\!\!
\lesssim [M_{\mu}(|\nabla f|^p)(x)]^{1/p}.
\end{align*}
To estimate $I_j$, for each $j\in\N$, expand $f(z) - f_{B(x,\alpha t)}$ as a telescoping sum to write
\begin{align*}
I_j
&\leq C_1 e^{-C_2(2^{j-1}\alpha -1)^2} \frac{\mu(B(x,2^j \alpha t))}{\mu(B(y, t))} t^{-1} \Bigg(\fint_{B(x,2^j \alpha t)} |f - \tilde{f}_{B(x,2^j \alpha t)}|\ d\mu \\
&\qquad\qquad\qquad\qquad\qquad\qquad
+ \sum_{i=1}^{j} |\tilde{f}_{B(x,2^i \alpha t)} - \tilde{f}_{B(x,2^{i-1} \alpha t)}|
+ |\tilde{f}_{B(x,\alpha t)}- f_{B(x,\alpha t)}|\Bigg)\\
&\lesssim e^{-C_2(2^{j-1}\alpha-1)^2} \frac{\mu(B(y,(1+2^j\alpha) t))}{\mu(B(y,t))} \sum_{i=0}^j t^{-1} \fint_{B(x,2^i \alpha t)} |f - \tilde{f}_{B(x,2^i \alpha t)}|\ d\mu \\
&\lesssim e^{-C_2(2^{j-1}\alpha-1)^2} (1+2^j\alpha)^{2n}  \sum_{i=0}^j 2^i\alpha \left(\fint_{B(x,2^i \alpha t)} |\nabla f|^p\ d\mu\right)^{1/p} \\
&\lesssim_\alpha e^{-C4^j} 4^{nj} [M_\mu(|\nabla f|^p)(x)]^{1/p},
\end{align*}
where the second inequality relies on the inclusion $B(x,2^j \alpha t)\subseteq B(y,(1+2^j\alpha) t)$, whilst the third inequality uses the doubling property of $\mu$ in \eqref{eq:D} with $p=2$, and the $L^p_\mu$-Poincar\'{e} inequality in~\eqref{wPoinc} with $c_B=\fint_{B(x,2^i \alpha t)} f\ d\mu$. Altogether, we have
\[
|t^{-1}[\PP_t(f - f_{B(x,\alpha t)})](y)|
\lesssim_\alpha \bigg(\sum_{j=0}^\infty e^{-C 4^j} 4^{nj} \bigg) [M_\mu(|\nabla f|^p)(x)]^{1/p}
\lesssim [M_\mu(|\nabla f|^p)(x)]^{1/p},
\]
which proves~\eqref{NonA} when $c_{B(x,\alpha t)}= \fint_{B(x,\alpha t)} f$. The proof when $c_{B(x,\alpha t)} = \fint_{B(x,\alpha t)} f\ d\mu$ follows as above by replacing $f_{B(x,\alpha t)}$ with $\tilde{f}_{B(x,\alpha t)}$, since~\eqref{wPoinc} can still be applied.

To prove~\eqref{NonB} and \eqref{NonBb}, suppose that $\eta>0$. The Gaussian kernel estimate for $e^{-t\mathcal{L}_\mu}$ in~\eqref{eq:G2} implies that $t \partial_t \mathcal{P}_t f (y)$ has an integral kernel $\widetilde{W}_{t^2}(y,z)$ satisfying 
\[
|\widetilde{W}_{t^2}(y,z)| \leq \frac{C_1}{\mu(B(y,t))} \exp\left(-C_2\frac{|y-z|^2}{t^2}\right)
\]
and the conservation property $\int_{\RR^n} \widetilde{W}_{t^2}(y,z)\ d\mu(y)=0$ for all $z\in \RR^n$ and $t>0$. This follows from Theorem~5 in~\cite{CuRc}, where the assumption that $\mathcal{A}$ is symmetric can be removed as per the remarks preceding this lemma. Therefore, we may write
\[
|\partial_t \mathcal{P}_t f(y)|
= t^{-1} \left|\int_{\RR^n} \widetilde{W}_{t^2}(y,z)[f(z) - f_{B(x,\eta t)}]\ d\mu(z)\right|
\]
and a change of variables implies that
\[
\sup_{(y,t)\in\Gamma^\eta(x)} |\partial_t \mathcal{P}_t f(y)|
= \sup_{(y,t)\in\Gamma (x)} t^{-1} \left|\int_{\RR^n} \eta \widetilde{W}_{(t/\eta)^2}(y,z)[f(z) - f_{B(x,t)}]\ d\mu(z)\right|.
\]
We can then obtain~\eqref{NonB} by following the proof of~\eqref{NonA} with $\alpha=1$ in order to show that this is bounded by $[M_{\mu}(|\nabla f|^p)(x)]^{1/p}$, since the doubling property of $\mu$ ensures that
\[
|\eta \widetilde{W}_{(t/\eta)^2}(y,z)| \leq \frac{C_{1,\eta}}{\mu(B(y,t))} \exp\left(-C_{2,\eta} \frac{|y-z|^2}{t^2}\right)
\]
for some positive constants $C_{1,\eta}$ and $C_{2,\eta}$ that depend on $\eta$. We obtain~\eqref{NonBb} as an immediate consequence of \eqref{NonB} and the fact that $\eta^{-1}\partial_t\mathcal{P}_{\eta t} = (\partial_{s}\mathcal{P}_{s})|_{s=\eta t}$.

To prove \eqref{NonC}, let $\eta>0$, set $u_{\eta t} := \mathcal{P}_{\eta t}f$ and choose a non-negative function $\Phi\in C_c^\infty(B(y,2\eta t))$ such that $\Phi\equiv 1$ on $B(y,\eta t)$ and $|\nabla_x \Phi| \lesssim (\eta t)^{-1}$. Let $c>0$ denote a constant that will be chosen later. The definition of $\mathcal{L}_\mu$ implies that
\begin{align*}
\fint_{B(y, \eta t)} &|\nabla_x \mathcal{P}_{\eta t}f|^2\ d\mu
\leq \frac{1}{\mu(B(y, \eta t))} \int_{\RR^n} |\nabla_x u_{\eta t}|^2\Phi^2 \ d\mu  \\
&\lesssim \frac{1}{\mu(B(y, \eta t))} \int_{\RR^n} \langle \mathcal{A}\nabla_x u_{\eta t}, \nabla_x (u_{\eta t} - c)\rangle \Phi^2 \\
&=\frac{1}{\mu(B(y,\eta t))}\int_{\RR^n} \left\{\langle \mathcal{A}\nabla_{x}u_{\eta t}, \nabla_x[(u_{\eta t} - c)\Phi^2]\rangle 
\!-\!2 \langle \mathcal{A}\nabla_{x}u_{\eta t}, \nabla_x\Phi(u_{\eta t} - c)\rangle \Phi\right\}\\
&\lesssim \frac{1}{\mu(B(y,\eta t))} \int_{\RR^n} \left\{(\mathcal{L}_{\mu} u_{\eta t}) (u_{\eta t} - c)\Phi^2
+ |\nabla_{x}u_{\eta t}||\nabla_x\Phi||(u_{\eta t} - c)\Phi| \right\} d\mu\\
&\leq \frac{1}{\mu(B(y,\eta t))} \int_{B(y,2\eta t)} \left( \frac{1}{2\eta^2t} |\partial_t u_{\eta t}| |u_{\eta t} - c|\Phi^2
+ |\nabla_{x}u_{\eta t}||\nabla_x\Phi||u_{\eta t} - c|\Phi\right) d\mu\\
&=: I + II.
\end{align*}
Now fix $c := \tilde{f}_{B(x, 3\eta t)}$. To estimate $I$, we use Cauchy's inequality and the doubling property of $\mu$, combined with the fact that $B(x,\eta t) \subseteq B(y, 2\eta t) \subseteq B(x, 3\eta t)$, to obtain
\begin{align*}
I &\lesssim \fint_{B(x,3\eta t)} \! \left(|\eta^{-1} \partial_t u_{\eta t}|^2
+ (\eta t)^{-2} |u_{\eta t} - f|^2
+ (\eta t)^{-2} |f - \tilde{f}_{B(x, 3\eta t)}|^2\right) d\mu =:I_1 + I_2 + I_3.
\end{align*}
It is immediate that $I_1 \leq M_\mu(|\eta^{-1} N_{*}^{\eta}(\partial_t\mathcal{P}_{\eta t}f)|^2)(x)$, whilst the semigroup property
\[
|u_{\eta t}(z) - f(z)|
= \left|\int_{0}^{\eta t} \partial_s u_s(z)\ ds\right|
\leq \eta t N_*(\partial_su_s)(z)
\]
implies that $I_2 \lesssim M_{\mu}(|N_*(\partial_su_s)|^2)(x)$, and the $L^2_\mu$-Poincar\'{e} inequality in~\eqref{wPoinc} shows that $I_3 \lesssim M_{\mu}(|\nabla f|^2)(x)$, hence
\begin{equation*}
I \leq M_\mu(|\eta^{-1} N_{*}^{\eta}(\partial_t\mathcal{P}_{\eta t}f)|^2)(x) + M_{\mu}(|N_*(\partial_su_s)|^2)(x) + M_{\mu}(|\nabla f|^2)(x).
\end{equation*}
To estimate $II$, we use Cauchy's inequality with $\epsilon>0$ to obtain
\begin{align*}
II \lesssim \frac{\epsilon}{\mu(B(y,\eta t))} \int_{\RR^n} |\nabla_x u_{\eta t}|^2 \Phi^2\ d\mu
+ \epsilon^{-1} (I_2+I_3).
\end{align*}
A sufficiently small choice of $\epsilon>0$ allows the $\epsilon$-term to be subtracted, yielding
\[
\fint_{B(y, \eta t)} |\nabla_x \mathcal{P}_{\eta t}f|^2\ d\mu
\lesssim I + II
\lesssim  M_{\mu}(|\eta^{-1}N_{*}^{\eta}(\partial_t\mathcal{P}_{\eta t}f)|^2 + |N_*(\partial_t\mathcal{P}_t f)|^2 + |\nabla f|^2)(x),
\]
which, combined with \eqref{NonB} and \eqref{NonBb}, implies~\eqref{NonC}.
\end{proof}

The pointwise estimates in Lemma~\ref{Lpsegr} have the following corollary.

\begin{corollary}\label{cor:wkest}
Let $n\geq2$ and suppose that $\mathcal{A}\in\mathcal{E}(n,\lambda,\Lambda,\mu)$ for some constants $0<\lambda\leq\Lambda<\infty$ and an $A_2$-weight $\mu$ on $\RR^n$. If $\eta>0$, then
\begin{align}
\label{Non1}
\mu\big(\{x\in\RR^n : |N_*^{\eta}(\partial_t\mathcal{P}_t f)(x)| > \kappa \}\big)
&\lesssim_\eta \kappa^{-2} \|\nabla f\|_{L^2_\mu(\RR^n,\RR^n)}^2,\\
\label{Non2}
\mu\big(\{x\in\RR^n : |\eta^{-1}N_*^{\eta}(\partial_{t}\mathcal{P}_{\eta t} f)(x)| > \kappa \}\big)
&\lesssim \kappa^{-2} \|\nabla f\|_{L^2_\mu(\RR^n,\RR^n)}^2,\\
\label{Non3}
\mu\big(\{x\in\RR^n : |\widetilde{N}_{*,\mu}^{\eta}(\nabla_x\mathcal{P}_{\eta t} f)(x)| > \kappa \}\big)
&\lesssim\kappa^{-2} \|\nabla f\|_{L^2_\mu(\RR^n,\RR^n)}^2,
\end{align}
for all $\kappa>0$ and $f\in W^{1,2}_{\mu}(\RR^n)$, where the implicit constants depend only on $n$, $\lambda$, $\Lambda$ and $[\mu]_{A_2}$, as well as on $\eta$ in \eqref{Non1}.
\end{corollary}

\begin{proof}
Estimates \eqref{Non1} and \eqref{Non2} follow respectively from \eqref{NonB} and \eqref{NonBb}, in the case $p=2$, since $M_\mu$ satisfies the weak-type estimate in \eqref{eq:wkdef}. To prove \eqref{Non3}, note that there exists $1<q<2$ such that $\mu$ is an $A_q$-weight on $\RR^n$ (see, for instance, Section 3 in Chapter~V of~\cite{St2}). Therefore, combining \eqref{NonC} in the case $p=q$ with \eqref{eq:wkdef} and noting that $2/q>1$, we obtain
\begin{align*}
\mu\big(\{x\in\RR^n : |\widetilde{N}_{*,\mu}^{\eta}(\nabla_x\mathcal{P}_{\eta t} f)(x)| > \kappa \}\big) 
& \lesssim \kappa^{-2} \big(\|M_{\mu}(|\nabla f|^q)\|^{2/q}_{L^{2/q}_\mu(\RR^n)} + \|\nabla f\|_{L^2_\mu(\RR^n,\RR^n)}^2\big) \\
& \lesssim \kappa^{-2} \|\nabla f\|_{L^2_\mu(\RR^n,\RR^n)}^2
\end{align*}
for all $\kappa>0$ and $f\in W^{1,2}_\mu(\RR^n)$ (since $W^{1,2}_\mu(\RR^n) \subseteq W^{1,q}_{\mu,\loc}(\RR^n)$), as required.
\end{proof}

\section{The Carleson Measure Estimate}\label{sec:CME}
The purpose of this section is to prove the Carleson measure estimate \eqref{CME} in Theorem~\ref{thm:introCME}. We adopt the strategy outlined at the end of Section~3.1 in~\cite{HKMP1}, although the crucial technical estimate, stated here as Theorem~\ref{thm:BPF}, is not at all an obvious extension of the uniformly elliptic case. Moreover, establishing the Carleson measure estimate directly allows us to avoid ``good-$\lambda$'' inequalities and thus apply a change of variables based on the $W^{1,2}_{0,\mu}$-Hodge decomposition in~\eqref{eq:W12Hodge}, instead of the $W^{1,2+\epsilon}_0$-version (for a sufficiently small $\epsilon>0$) required in~\cite{HKMP1}.

The technical result in Theorem~\ref{thm:BPF} establishes \eqref{CME} on certain ``big pieces'' of all cubes. The passage to the general estimate ultimately follows from the self-improvement property for Carleson measures in Lemma~\ref{JNforCar}. This requires, however, that the Carleson measure estimate on the full gradient $\nabla u$ of a solution $u$ can be controlled by the same estimate on its transversal derivate $\partial_t u$, which is the content of  Lemma~\ref{lem:tangrad}. We briefly postpone the statement and proof of Lemma~\ref{lem:tangrad} and Theorem~\ref{thm:BPF}, however, in order to deduce Theorem~\ref{thm:introCME} from those results below.

In contrast to the previous two sections, the results here concern solutions of the equation $\div (A \nabla u)=0$ in open sets $\Omega\subseteq \RR^{n+1}_+$ when $n\geq2$ and $A$ is a $t$-independent coefficient matrix that satisfies \eqref{eq:deg.def} for some $0<\lambda\leq\Lambda<\infty$ and an $A_2$-weight $\mu$ on~$\RR^n$. In particular, in Section~\ref{sec:prelims}, weighted Sobolev spaces were defined on open sets in $\RR^d$ and matrix coefficients $\mathcal{A}\in\mathcal{E}(d,\lambda,\Lambda,\mu)$ were considered for all $d\in\N$. Those results also hold here on open sets in the upper half-space with the weight $\mu(x,t):=\mu(x)$ and the coefficients $A(x,t):=A(x)$ for all $(x,t)\in\RR^{n+1}$, since then $[\mu]_{A_2(\RR^{n+1})}=[\mu]_{A_2(\RR^n)}$ and $A\in\mathcal{E}(n\!+\!1,\lambda,\Lambda,\mu)$. In particular, the solution space $W^{1,2}_{\mu,\loc}(\Omega)$ is defined and the regularity estimates in \eqref{eq:M}, \eqref{eq:dGN} and \eqref{eq:H} hold when $\Omega \subseteq \RR^{n+1}_+$.

We will also use, without reference, the well-known fact that if $u$ is a solution of $\div (A \nabla u)=0$ in $\Omega \subseteq \RR^{n+1}_+$, then $\partial_t u$ is also a solution in $\Omega$. In particular, to see that $\partial_t u$ is in $W^{1,2}_{\mu,\loc}(\Omega)$, a Whitney decomposition of $\Omega$ reduces matters to showing that $\partial_t u$ is in $W^{1,2}_{\mu}(R)$ for all cubes $R\subset\Omega$ satisfying $\ell(R) < \frac{1}{2}\dist(R,\partial\Omega)$. To this end, define the difference quotients $D^h_i u(X) := \tfrac{1}{h}[u(X+he_i) -u(X)]$ for all $X\in R$ and $h\!<\!\dist(R,\partial\Omega)$, where $e_i$ is the unit vector in the $i$th-coordinate direction in $\RR^{n+1}$. The $t$-independence of the coefficients implies that $D^h_{n+1}u$ is a solution in $R$, so we use the identity $D^h_{n+1}(\partial_i u) = \partial_i (D^h_{n+1}u)$ and Caccioppoli's inequality to obtain 
\begin{align*}
\iint_{R} |D^h_{n+1}(\partial_i u)|^2\ d\mu
&\leq \iint_{R} |\nabla (D^h_{n+1}u)|^2\ d\mu
\lesssim \ell(R)^2 \iint_{2R} |D^h_{n+1}u|^2\ d\mu \\
&\leq \ell(R)^2 \iint_{2R} |\partial_tu|^2\ d\mu =: K
\qquad \forall h<\dist(R,\partial\Omega),
\end{align*}
where the implicit constant depends only on $n$, $\lambda$, $\Lambda$ and $[\mu]_{A_2}$, and the final bound holds uniformly in $h$ because $u$ is in $W^{1,2}_\mu(R)$ (see Lemma 7.23 in \cite{GT}). We can then use Lemma~7.24 in \cite{GT} to deduce that $\partial_t u$ is in $W^{1,2}_\mu(R)$ with the estimate $\|\partial_i\partial_tu\|_{L^2_\mu(R)}^2=\|\partial_t\partial_iu\|_{L^2_\mu(R)}^2\leq K$ for all $i\in\{1,\ldots,n\!+\!1\}$, as required. Note that the proofs of Lemmas 7.23 and 7.24 in \cite{GT} extend immediately to the weighted context considered here because $C^\infty(R)$ is still dense in $W^{1,2}_\mu(R)$.

\begin{proof}[Proof of Theorem~\ref{thm:introCME} from Lemma~\ref{lem:tangrad} and Theorem~\ref{thm:BPF}] Let $Q\subset\RR^n$ denote a cube and suppose that $u\in L^\infty(\RR^{n+1}_+)$ solves $\div (A \nabla u)=0$ in $\RR^{n+1}_+$. It follows \textit{a~fortiori} from Theorem~\ref{thm:BPF} that there exist constants $C, c_0>0$ and, for each cube $Q'\subseteq Q$, a measurable set $F' \subset Q'$ such that $\mu(F') \geq c_0\mu(Q')$ and
\[
\frac{1}{\mu(Q')}\int_0^{l(Q')}\int_{F'}|t\partial_t u(x,t)|^2\ d\mu(x)\frac{dt}{t} \leq C \|u\|_{{\infty}}^2,
\]
where $C$ and $c_0$ depend only on $n$, $\lambda$, $\Lambda$ and $[\mu]_{A_2}$.

The coefficient matrix $A$ is $t$-independent, so $\partial_t u$ is also a solution and thus the degenerate version of Moser's estimate in \eqref{eq:M}, followed by Caccioppoli's inequality, shows that $\|t\partial_t u\|_\infty \lesssim \|u\|_\infty$. Moreover, the degenerate version of the de~Giorgi--Nash H\"{o}lder regularity for solutions in \eqref{eq:dGN} shows that  
\[
|t\partial_tu(x,t)-t\partial_tu(y,t)| 
\lesssim \left(\frac{|x-y|}{t}\right)^{\alpha} \|t\partial_tu\|_\infty 
\lesssim \|u\|_\infty \left(\frac{|x-y|}{t}\right)^{\alpha}
\]
for all $x,y \in Q$ and $t>0$, where all of the implicit constants and the exponent $\alpha>0$ depend only on $n$, $\lambda$, $\Lambda$ and $[\mu]_{A_2}$. Therefore, we may apply Lemma~\ref{JNforCar} with $\{v_t,\alpha,\beta_0,\eta,\beta\}:=\{(t\partial_t u)^2,\alpha,C \|u\|_\infty^2,c_0,C\|u\|_\infty^2\}$ to obtain
\begin{equation}\label{eq:bdrctyest}
\frac{1}{\mu(Q)}\int_0^{\ell(Q)}\int_Q |t\partial_t u(x,t)|^2\ d\mu(x)\frac{dt}{t} \lesssim \|u\|_{{\infty}}^2,
\end{equation}
where the implicit constant depends only on $n$, $\lambda$, $\Lambda$ and $[\mu]_{A_2}$. This estimate holds for all cubes $Q$, so by Lemma~\ref{lem:tangrad}, we conclude that \eqref{CME} holds.
\end{proof}


We now dispense with the following lemma, which was used in the proof of Theorem~\ref{thm:introCME} above to reduce to a Carleson measure estimate on the transversal derivative of solutions. The proof is adapted from Section~3.1 of ~\cite{HKMP1}.

\begin{lemma}\label{lem:tangrad}
Let $n\geq 2$ and consider a cube $Q\subset\RR^n$. If $A$ is a $t$-independent coefficient matrix that satisfies the degenerate bound and ellipticity in~\eqref{eq:deg.def} for some constants $0<\lambda\leq\Lambda<\infty$ and an $A_2$-weight $\mu$ on $\RR^n$, then any solution $u\in L^\infty(4Q\times(0,4\ell(Q)))$ of $\div (A \nabla u)=0$ in $4Q\times(0,4\ell(Q))$ satisfies
\[
\int_{0}^{\ell(Q)}\int_{Q} |t\nabla u(x,t)|^2\ d\mu(x)\frac{dt}{t} 
\lesssim \int_{0}^{4\ell(Q)}\int_{4Q} |t\partial_t u(x,t)|^2\ d\mu(x)\frac{dt}{t} + \mu(Q)\|u\|_\infty^2,
\]
where the implicit constant depends only on $n$, $\lambda$, $\Lambda$ and $[\mu]_{A_2}$.
\end{lemma}

\begin{proof}
Let $0<\delta<1/2$ and set $\Phi_Q(t):= \Phi\left({t}/{\ell(Q)}\right)$, where $\Phi:\RR\rightarrow[0,1]$ denotes a $C^\infty$-function such that $\Phi(t) = 1$ for all $2\delta \leq t \leq 1$ whilst $\Phi(t) = 0$ for all $t\leq\delta$ and $t \geq 2$. Integrating by parts with respect to the $t$-variable and noting that $\|\partial_t\Phi\|_{L^\infty([1,2])} \lesssim 1$ whilst $\|\partial_t\Phi\|_{L^\infty([\delta,2\delta])} \lesssim 1/\delta$, we obtain
\begin{align*}
\mathbf{I}:&=\int_Q\int_{0}^{2\ell(Q)} \left|\nabla u(x,t) \right|^2 \Phi_Q(t)t\ dt d\mu(x) \\
&\eqsim \int_Q\int_{0}^{2\ell(Q)}\partial_t \left( \left| \nabla u(x,t)\right|^2\Phi_Q(t)\right) t^2 dt d\mu(x)\\
&\lesssim \int_Q\int_{0}^{2\ell(Q)} \left\langle \nabla\partial_t u(x,t), \nabla u(x,t)\right\rangle \Phi_Q(t) t^2 dt d\mu(x) \\
& \quad+ \int_Q\fint_{\ell(Q)}^{2\ell(Q)} \left| \nabla u(x,t)\right|^2 t^2 dt d\mu(x) + \int_Q\fint_{\delta \ell(Q)}^{2\delta \ell(Q)} \left| \nabla u(x,t)\right|^2 t^2 dt d\mu(x) \\
&=: \mathbf{I}'+\mathbf{I}''+\mathbf{I}'''.
\end{align*}
For the term $\mathbf{I}'$, we apply Cauchy's inequality with an arbitrary $\epsilon>0$, to obtain
\begin{align*}
\mathbf{I}' \leq \epsilon \mathbf{I} + \frac{1}{\epsilon}\int_Q\int_0^{2\ell(Q)}\left|\nabla \partial_t u(x,t) \right|^2t^3 dt d\mu(x).
\end{align*}
For the term $\mathbf{I}''$, we apply Caccioppoli's inequality, the doubling property of $\mu$ and the fact that $t\approx \ell(Q)$ in the domain of the integration, to obtain
\begin{align*}
\mathbf{I}'' &\eqsim \ell(Q) \int_Q\int_{\ell(Q)}^{2\ell(Q)} | \nabla u(x,t)|^2 dt d\mu(x)\\
&\lesssim \frac{1}{\ell(Q)}\int_{2Q}\int_{\ell(Q)/2}^{5\ell(Q)/2} |u(x,t)|^2 dt d\mu(x) \\
&\lesssim \mu(Q) \|u\|_{{\infty}}^2.
\end{align*}
For the term $\mathbf{I}'''$, the same reasoning shows that $\mathbf{I}'''\lesssim \mu(Q) \|u\|_{{\infty}}^2$. We now fix $\epsilon>0$, depending only on allowable constants, such that altogether
\[
\mathbf{I} \lesssim \int_Q\int_{0}^{2\ell(Q)} \left| \nabla \partial_t u(x,t)\right|^2 t^3dt d\mu(x) + \mu(Q) \|u\|^2_{\infty},
\]
which is justified since $\mathbf{I}<\infty$ by Caccioppoli's inequality and the support of $\Phi_Q$.

To complete the estimate, we let $\{W_j: j \in J\}$ denote a collection of Whitney boxes (from a Whitney decomposition of $\RR^{n+1}_+$) such that $W_j \cap \left( Q\times (0,2\ell(Q))\right)\neq\emptyset$ and $\sum_{j\in J}\mathbbm{1}_{2W_j}(x,t)\lesssim 1$. The coefficient matrix $A$ is $t$-independent, so $\partial_t u$ is also a solution of $\div (A\nabla u)=0$ in each set $W_j$, hence we may apply Caccioppoli's inequality in combination with the fact that $t \eqsim l(W_j)$ in $W_j$, to obtain
\begin{align*}
\int_{2\delta \ell(Q)}^{\ell(Q)} \int_Q \left|t\nabla u(x,t) \right|^2  d\mu(x) \frac{dt}{t} 
&\lesssim\sum_{j\in J}\iint_{W_j} \left| \nabla \partial_t u(x,t)\right|^2t^3 dt d\mu(x) + \mu(Q) \|u\|^2_{\infty}\\
&\lesssim \sum_{j\in J}l(W_j)\iint_{2W_j} \left| \partial_t u(x,t)\right|^2 dt d\mu(x)+ \mu(Q) \|u\|^2_{\infty}\\
&\lesssim \int_0^{4\ell(Q)} \int_{4Q} \left|t\partial_t u(x,t)\right|^2 d\mu(x) \frac{dt}{t} + \mu(Q) \|u\|^2_{\infty},
\end{align*}
where the implicit constants do not depend on $\delta$. The final result is then obtained by applying Fatou's lemma to estimate the limit as $\delta$ approaches $0$.
\end{proof}

The remainder of this section is dedicated to the proof of the crucial technical estimate, Theorem~\ref{thm:BPF}, that was used to prove Theorem~\ref{thm:introCME}. The proof adapts the change of variables from Section~3.2 of~\cite{HKMP1} to the degenerate elliptic case. This is used to pull-back solutions to certain sawtooth domains where the Carleson measure estimate can be verified by reducing matters to the vertical square function estimates in Theorem~\ref{VSFbound}, which we recall were obtained from the solution of the Kato problem in~\cite{CuR}. The following technical lemma, which reprises the notation $\mathcal{P}_t := e^{-t^2\mathcal{L}_\mu}$ for $\mathcal{L}_\mu:=-\div_\mu(\frac{1}{\mu}\mathcal{A}\nabla)$ and $\mathcal{A}\in\mathcal{E}(n,\lambda,\Lambda,\mu)$ as in~\eqref{eq:deg.def.bdry} and Lemma~\ref{Lpsegr}, will be used to justify these changes of variables.

\begin{lemma}\label{lem:CoV}
Let $n\geq2$ and suppose that $\mathcal{A}\in\mathcal{E}(n,\lambda,\Lambda,\mu)$ for some  constants $0<\lambda\leq\Lambda<\infty$ and an $A_2$-weight $\mu$ on $\RR^n$. Let $Q\subset\RR^n$ denote a cube and suppose that $\mathbf{f}:5Q\rightarrow \RR^n$ is a measurable function such that $\frac{1}{\mu}\mathbf{f} \in L^\infty(5Q)$. Let $\phi \in W^{1,2}_{0,\mu}(5Q)$ and suppose that $\div (\mathcal{A} \nabla \phi) = \div \mathbf{f}$ in $5Q$. If $\kappa_0>0$, $0<\eta<1/2$ and $x_0\in Q$ satisfy $\Lambda(\eta,\phi,\mathcal{A})(x_0) \leq \kappa_0$, where
\begin{equation}\label{eq:LambdaDef}
\Lambda(\eta,\phi,\mathcal{A}):=\eta^{-1}N_*^{\eta}(\partial_t\PP_{\eta t} \phi) + N_*(\partial_t\PP_{ t} \phi) + [M_\mu(|\nabla_x \phi |^2)]^{1/2} + D_{*,\mu}\phi,
\end{equation}
then
\begin{equation}\label{eq:partialPest}
|\partial_t \PP_{\eta t} \phi(x)| \leq \eta\kappa_0
\qquad \forall (x,t)\in \Gamma_\eta(x_0)
\end{equation}
and
\begin{equation}\label{eq:I-Pest}
|(I-\PP_{\eta t})\phi(x)| \lesssim \eta (\kappa_0 + \|\tfrac{1}{\mu}\mathbf{f}\|_\infty) t
\qquad \forall (x,t)\in \Gamma_\eta(x_0) \cap (2Q \times (0,4\ell(Q))),
\end{equation}
where the implicit constant depends only on $n$, $\lambda$, $\Lambda$ and $[\mu]_{A_2}$.
\end{lemma}

\begin{proof}
Suppose that $\kappa_0>0$, $0<\eta<1/2$ and $x_0\in Q$ satisfy $\Lambda(\eta,\phi,\mathcal{A})(x_0) \leq \kappa_0$. It follows \textit{a~fortiori} that $\eta^{-1}N_*^{\eta}(\partial_t\PP_{\eta t} \phi)(x_0) \leq \kappa_0$, so \eqref{eq:partialPest} holds for all $(x,t)\in \Gamma_\eta(x_0)$.

To prove \eqref{eq:I-Pest}, first note that the properties of the semigroup imply that
\begin{equation}\label{chvar1}
|(I - \PP_{\eta t})\phi(x_0)| = \left| \int_{0}^{\eta t} \partial_s\PP_{s}\phi(x_0)\ ds\right|
\leq \eta t \kappa_0
\end{equation}
for all $t>0$, since $N_*(\partial_s\PP_s\phi)(x_0) \leq \kappa_0$. Now let $(x,t) \in \Gamma_\eta(x_0) \cap (2Q \times (0,4\ell(Q)))$. We set $\phi_{x_0, \eta t}:=\fint_{B(x_0,2 \eta t)}\phi(y) dy$ and apply estimate (\ref{NonA}) with $\alpha=2$ to obtain
\begin{equation}\label{chvar2}
|\PP_{\eta t}(\phi - \phi_{x_0,\eta t})(x)| \lesssim \eta t [M_{\mu}(|\nabla_x\phi|^2)(x_0)]^{1/2} \leq \eta t \kappa_0.
\end{equation}
Next, since $\div(\mathcal{A}\nabla(\phi-\phi(x_0))) = \div(\mathcal{A}\nabla\phi)= \div \mathbf{f}$ in $5Q$, and since $0<\eta<1/2$ ensures that $B(x_0,2 \eta t) \subseteq 5Q$, we may apply the degenerate version of Moser's estimate for inhomogeneous equations in \eqref{eq:MoserInh} to obtain
\begin{align}\begin{split}\label{MoChvar}
|\phi(x) - \phi(x_0)| &\lesssim \left(\fint_{B(x_0,2 \eta t)}|\phi(y) - \phi(x_0)|^2\ d\mu(y) \right)^{1/2} + \eta t \|\tfrac{1}{\mu}\mathbf{f}\|_\infty \\
&\lesssim \eta t (D_{*,\mu}\phi(x_0) + \|\tfrac{1}{\mu}\mathbf{f}\|_\infty) \\
&\lesssim \eta t (\kappa_0 + \|\tfrac{1}{\mu}\mathbf{f}\|_\infty).
\end{split}\end{align}
Combining estimates (\ref{chvar1}), (\ref{chvar2}) and (\ref{MoChvar}), we obtain
\begin{align*}
|(I - \PP_{\eta t})\phi (x)| &\leq |\phi(x) - \phi(x_0)| + |(I - \PP_{\eta t})\phi(x_0)| \\ 
&\qquad + |\PP_{\eta t}(\phi - \phi_{x_0, \eta t})(x_0)| + |\PP_{\eta t}(\phi - \phi_{x_0, \eta t})(x)| \\
&\lesssim \eta (\kappa_0 + \|\tfrac{1}{\mu}\mathbf{f}\|_\infty) t,
\end{align*}
which proves \eqref{eq:I-Pest}, as the implicit constant depends only on $n$, $\lambda$, $\Lambda$ and $[\mu]_{A_2}$.
\end{proof}

We now present the main technical result of this section. The proof is adapted from Section~3.2 of \cite{HKMP1}, although some arguments have been simplified as detailed at the beginning of this section, and the additional justification required in the degenerate elliptic case has been emphasized.

The strategy of the original proof in \cite{HKMP1} was motivated in-part by the fact that integration by parts is sufficient to establish the required estimate in the case when $A$ has a certain block upper-triangular structure. A key idea in \cite{HKMP1} was to account for the presence of lower-triangular coefficients $\mathbf{c}$ (and upper-triangular coefficients) by decomposing them according to a $W^{1,2+\epsilon}_0$-Hodge decomposition. This was done locally on a given cube $Q$ and the idea has been adapted  here. First, the $W^{1,2}_{0,\mu}$-Hodge decomposition $\mathbf{c}\mathbbm{1}_{5Q} = \mu \mathbf{h}-A_{||}^*\nabla \varphi$ is introduced in \eqref{Hodgeforvarphi}, where $A_{||}$ is the $n\times n$ submatrix of $A$ shown in \eqref{eq:Abcd}. After integrating by parts, the divergence-free component $\mu \mathbf{h}$ provides valuable cancellation, whilst the adapted gradient vector field $A_{||}^*\nabla \varphi$ facilitates a reduction to the square function estimates in Theorem~\ref{VSFbound}, which are implied by the solution to the Kato problem in~\cite{CuR}, for the boundary operator $L_{||,\mu}^*:=-\div_\mu (\tfrac{1}{\mu}A_{||}^*\nabla_x)$.\vspace{-3pt}

The latter estimates, however, require that $L_{||,\mu}^*$ acts on the range of $P_{t}^*:=e^{-t^2L_{||,\mu}^*}$ and this is arranged by initially making the Dhalberg--Kenig--Stein-type pull-back $t \mapsto t - (I-P_{\eta t}^*)\varphi(x)$ so that the lower-triangular coefficients become $\mu\mathbf{h} - A_{||}^*\nabla_xP_{\eta t}^*\varphi$. This change of variables is justified by choosing $\eta>0$ small enough so that the pull-back is bi-Lipschitz in $t$. Once this is in place, a set $F$ is introduced that contains a ``big piece'' of $Q$ and on which the various maximal functions in Lemma~\ref{lem:CoV} are bounded. The integration on $F\times (0,\ell(Q))$ is then performed by introducing a smooth test function $\Psi_\delta$ that equals 1 on $F\times (2\delta\ell(Q), 2\ell(Q))$ and is supported on a certain truncated sawtooth domain $\Omega_{\eta/8,Q,\delta}$ over $F$, where $\delta>0$ is an arbitrary (small) parameter that provides for a smooth truncation in the $t$-direction near the boundary of $\RR^{n+1}_+$. The main integration by parts is then performed in \eqref{mainest}. The two principal terms $\mathbf{S}_1$ and $\mathbf{S}_2$  arise from the tangential and transversal integration by parts, respectively, where the former is taken with respect to the measure $\mu$ and thus requires additional justification from the uniformly elliptic case. These and numerous error terms are then shown to be appropriately under control.

\begin{theorem}\label{thm:BPF}
Let $n\geq 2$ and consider a cube $Q\subset\RR^n$. If $A$ is a $t$-independent coefficient matrix that satisfies the degenerate bound and ellipticity in \eqref{eq:deg.def} for some constants $0<\lambda\leq\Lambda<\infty$ and an $A_2$-weight $\mu$ on $\RR^n$, then for any solution $u\in L^\infty(4Q\times(0,4\ell(Q)))$ that solves $\div (A \nabla u)=0$ in $4Q\times(0,4\ell(Q))$, there exist constants $C, c_0>0$ and a measurable set $F \subset Q$  such that $\mu(F) \geq c_0\mu(Q)$ and
\begin{equation}\label{eq:CMEforFinPf}
\frac{1}{\mu(Q)}\int_0^{\ell(Q)}\int_F|t\nabla u(x,t)|^2\ d\mu(x)\frac{dt}{t} \leq C \|u\|_{{\infty}}^2,
\end{equation}
where $C$ and $c_0$ depend only on $n$, $\lambda$, $\Lambda$ and $[\mu]_{A_2}$.
\end{theorem}

\begin{proof}
We begin by expressing the matrix $A$ and its adjoint $A^*$ (which is just the transpose $A^{\mathrm{t}}$, since the matrix coefficients are real-valued) in the following form
\begin{equation}\label{eq:Abcd}
A = \left[
\begin{array}{c|c}
A_{||} & \mathbf{b} \\
\hline
\mathbf{c}^{\mathrm{t}} & d
\end{array}
\right],
\qquad
A^* = \left[
\begin{array}{c|c}
A_{||}^* & \mathbf{c} \\
\hline
\mathbf{b}^{\mathrm{t}} & d
\end{array}
\right],
\end{equation}
where $A_{||}$ denotes the $n\times n$ submatrix of $A$ with entries $(A_{||})_{i,j}:= A_{i,j}, 1\leq i,j \leq n$, whilst $\mathbf{b} := (A_{i, n+1})_{1\leq i \leq n}$ is a column vector, $\mathbf{c}^{\mathrm{t}}:=(A_{n+1, j})_{1\leq j \leq n}$ is a row vector and $d:=A_{n+1, n+1}$ is a scalar. 

Now consider a cube $Q\subset\RR^n$. The aim is to construct a set $F\subset Q$ with the required properties. To this end, we apply the Hodge decomposition from \eqref{eq:W12Hodge} to the space $L^2_{\mu}(5Q,\RR^n)$ in order to write
\begin{equation}\label{Hodgeforvarphi}
\tfrac{1}{\mu}\mathbf{c}\mathbbm{1}_{5Q} = -\tfrac{1}{\mu}A_{||}^*\nabla \varphi + \mathbf{h},
\qquad \tfrac{1}{\mu}\mathbf{b}\mathbbm{1}_{5Q} = -\tfrac{1}{\mu} A_{||}\nabla \widetilde{\varphi} + \widetilde{\mathbf{h}},
\end{equation}
where $\varphi, \widetilde{\varphi} \in W^{1,2}_{\mu,0}(5Q)$ and $\mathbf{h},\widetilde{\mathbf{h}} \in L^2_\mu(5Q,\RR^n)$ are such that $\div_
\mu\mathbf{h} = \div_\mu\widetilde{\mathbf{h}}=0$ and
\begin{align}\label{HodEst1}
\begin{split}
\fint_{5Q}\Big(|\nabla \varphi(x)|^{2} + |\mathbf{h}(x)|^{2}\Big) d\mu(x)
\lesssim \fint_{5Q}\left|\frac{\mathbf{c}(x)}{\mu}\right|^{2} d\mu(x) \lesssim 1,
\end{split}
\end{align}
\begin{align}\label{HodEst2}
\begin{split}
\fint_{5Q}\Big(|\nabla \widetilde\varphi(x)|^{2} + |\mathbf{\widetilde{h}}(x)|^{2}\Big) d\mu(x)
\lesssim \fint_{5Q}\left|\frac{\mathbf{b}(x)}{\mu}\right|^{2} d\mu(x) \lesssim 1.
\end{split}
\end{align}
We extend each of $\varphi, \widetilde{\varphi}, \mathbf{h}, \widetilde{\mathbf{h}}$ to functions on $\RR^n$ by setting them equal to 0 on $\RR^n\setminus 5Q$.

In Sections~\ref{sec:prelims} and~\ref{sec:maxops}, we investigated the operators $\mathcal{L}_\mu:=-\div_\mu (\tfrac{1}{\mu}\mathcal{A}\nabla)$ and $\mathcal{P}_t:=e^{-t^2\mathcal{L}_\mu}$ for arbitrary coefficient matrices $\mathcal{A}$ in $\mathcal{E}(n,\lambda,\Lambda,\mu)$. We now set
\begin{align}\begin{split}\label{eq:L||mudef}
L_{||,\mu}:=-\div_\mu (\tfrac{1}{\mu}A_{||}\nabla_x),\qquad
&P_t:=e^{-t^2L_{||,\mu}},\\
L_{||,\mu}^*:=-\div_\mu (\tfrac{1}{\mu}A_{||}^*\nabla_x),\qquad
&P_t^*:=e^{-t^2L_{||,\mu}^*}
\end{split}\end{align}
in order to apply those results in the cases $\mathcal{A}=A_{||}$ and $\mathcal{A}=A_{||}^*$.

We now introduce two constants $\kappa_0,\eta>0$, which will be fixed shortly, and recall the function $\Lambda(\eta,\phi,\mathcal{A})$ from \eqref{eq:LambdaDef} to define the set $F\subset Q$ by
\begin{multline}\label{F}
F := \Big\{x\in Q: \Lambda(\eta,\varphi,A_{||}^*)(x) + \Lambda(\eta,\widetilde{\varphi},A_{||})(x) \\
+ \widetilde{N}_{*,\mu}^{\eta}(\nabla_x P^*_{\eta t} \varphi)(x) + \widetilde{N}_{*,\mu}^{\eta}(\nabla_x P_{\eta t} \widetilde{\varphi})(x) \leq \kappa_0 \Big\}.
\end{multline}
Applying the weak-type bounds in \eqref{eq:wkdef}, \eqref{Dptilde}, \eqref{Non2} and \eqref{Non3} followed by the estimates from the Hodge decomposition in \eqref{HodEst1} and \eqref{HodEst2}, we obtain
\[
\mu(Q\setminus F) \lesssim \kappa_0^{-2} \left(\|\nabla \varphi\|_{L^2_\mu(\RR^n,\RR^n)}^2+\|\nabla \widetilde\varphi\|_{L^2_\mu(\RR^n,\RR^n)}^2\right) \lesssim \kappa_0^{-2} \mu(Q),
\]
where the implicit constants depend only on $n$, $\lambda$, $\Lambda$ and $[\mu]_{A_2}$. This allows us to now fix $\kappa_0>1$ and some constant  $c_0>0$ such that $\mu(F) \geq c_0\mu(Q)$, where both $\kappa_0$ and $c_0$ depend only on the allowed constants, and thus are independent of $\eta$.

We now fix the value of $\eta$ as follows. First, for $0\leq\alpha\leq 4$ and $\beta>0$, let
\[
\Omega_{\beta} := \bigcup\nolimits_{x \in F}\Gamma_\beta(x),
\quad
\Omega_{\beta,Q,\alpha} := \Omega_{\beta} \cap \Big(2Q \times (\alpha\ell(Q),4\ell(Q))\Big)
\quad\textrm{and}\quad
\Omega_{\beta,Q} := \Omega_{\beta,Q,0}
\]
denote the sawtooth domains in $\RR^{n+1}_+$ spanned by cones centered on~$F$ of aperture~$\beta$. Next, note that the properties of the Hodge decomposition in~\eqref{Hodgeforvarphi} imply that $-\div (A_{||}^*\nabla\varphi)=\div (\mathbf{c}\mathbbm{1}_{5Q})$ and $-\div (A_{||}\nabla\widetilde\varphi)=\div (\mathbf{b}\mathbbm{1}_{5Q})$ in $5Q$. Therefore, we now fix $0<\eta<1/2$ in accordance with \eqref{eq:partialPest} and \eqref{eq:I-Pest} such that
\begin{equation}\label{eq:CoV1}
\max\left\{|\partial_t P_{\eta t}^* \varphi(x)|,|\partial_t P_{\eta t} \widetilde\varphi(x)|\right\} \leq \eta\kappa_0 < 1/8
\qquad \forall (x,t)\in \Omega_\eta
\end{equation}
and
\begin{align}\begin{split}\label{eq:CoV2}
&\max\big\{|(I-P_{\eta t}^*)\varphi(x)|,|(I-P_{\eta t})\widetilde\varphi(x)|\big\} \\
&\lesssim \eta \big(\kappa_0 + \max\big\{\|\tfrac{1}{\mu}\mathbf{c}\|_\infty,\|\tfrac{1}{\mu}\mathbf{b}\|_\infty\big\}\big) t
\lesssim \eta \kappa_0 t
< t/8
\qquad \forall (x,t)\in \Omega_{\eta,Q},
\end{split}\end{align}
where $\eta$ and the implicit constants depend only on $n$, $\lambda$, $\Lambda$ and $[\mu]_{A_2}$.

It remains to prove \eqref{eq:CMEforFinPf}. We will achieve this by changing variables in the transversal direction using the mapping $t \mapsto \tau(x,t)$, with $x\in\RR^n$ fixed, defined by
\begin{equation*}
\tau(x,t):=t - (I-P_{\eta t}^*)\varphi(x)
\end{equation*}
and having Jacobian denoted by
\begin{equation}\label{eq:Jdef}
J(x,t):=  \partial_t\tau(x,t) 
= 1+ \partial_tP_{\eta t}^*\varphi(x).
\end{equation}
In order to justify such changes of variables, we note from \eqref{eq:CoV1} and \eqref{eq:CoV2} that 
\begin{equation}\label{334}
\frac{7t}{8} < \tau(x,t) < \frac{9t}{8}
\qquad\textrm{and}\qquad
\frac{7}{8} < J(x,t)< \frac{9}{8}
\qquad \forall (x,t)\in \Omega_{\eta,Q}.
\end{equation}
In particular, for each $x\in F$ and $0\leq\alpha\leq1/8$, this implies that the mapping $t \mapsto \tau(x,t)$ is bi-Lipschitz in $t$ on $(2\alpha\ell(Q),2\ell(Q))$ with range 
\begin{equation}\label{334tau}
(4\alpha\ell(Q),\ell(Q)) \subseteq \tau(x,\cdot)\big((2\alpha\ell(Q),2\ell(Q))\big) \subseteq (\alpha\ell(Q),4\ell(Q)).
\end{equation}
Moreover, for each $0<\beta \leq \eta$, the mapping $(x,t) \mapsto \rho(x,t)$ defined by
\begin{equation*}
\rho(x,t):=(x, \tau(x,t))=(x, t +P_{\eta t}^*\varphi(x) - \varphi(x))
\end{equation*}
is bi-Lipschitz in $t$ on $\Omega_{\beta,Q}$ with range 
\begin{equation}\label{335}
\Omega_{{8\beta}/{9},Q} \subseteq \rho(\Omega_{\beta,Q}) \subseteq \Omega_{{8\beta}/{7},Q}.
\end{equation}

Now consider a bounded solution $u$ satisfying $\div (A \nabla u)=0$ in $4Q\times(0,4\ell(Q))$. The pull-back $u_1 := u\circ \rho$ is in $L^\infty(\Omega_{\eta,Q})$ and $\div (A_1\nabla u_1)=0$ in $\Omega_{\eta,Q}$, where
\[
A_1 := \left[
\begin{array}{c|c}
JA_{||} & \mathbf{b}  + A_{||}\nabla_x\varphi - A_{||}\nabla_xP_{\eta t}^*\varphi \\
\hline
(\mu\mathbf{h} - A_{||}^*\nabla_xP_{\eta t}^*\varphi)^{\mathrm{t}} & {\left\langle A\mathbf{p}, \mathbf{p}\right\rangle}/{J}
\end{array}
\right]
\]
and
\begin{equation}\label{eq:pdef}
\mathbf{p}(x,t) := 
\begin{bmatrix} \nabla_x\tau(x,t) \\ -1 \end{bmatrix}
= \begin{bmatrix} \nabla_xP_{\eta t}^*\varphi(x) - \nabla_x\varphi(x) \\ -1 \end{bmatrix}.
\end{equation}
Our statement that $\div (A_1\nabla u_1)=0 $ in $\Omega_{\eta,Q}$ does not mean that $A_1$ satisfies \eqref{eq:deg.def}, only that $u_1\in W^{1,2}_{\mu,\loc}(\Omega_{\eta,Q})$ and that $\int_{\RR^{n+1}_+} \langle A_1\nabla u_1, \nabla \Phi\rangle =0$ for all $\Phi\in C^\infty_c(\Omega_{\eta,Q})$. To prove this, we combine the pointwise identity
\begin{equation}\label{210}
\langle A \left( (\nabla u)\circ\rho\right), (\nabla v)\circ\rho\rangle J = \langle A_1\nabla (u \circ \rho) , \nabla (v \circ \rho)\rangle
\qquad \forall
v \in W^{1,2}_{0,\mu}(\rho(\Omega_{\eta,Q}))
\end{equation}
with the change of variables $(x,t) \mapsto \rho(x,t)$ on $\Omega_{\eta,Q}$, which is justified because $\rho$ is bi-Lipschitz in $t$ on $\Omega_{\eta,Q}$ with range $\rho(\Omega_{\eta,Q}) \subset 4Q\times(0,4\ell(Q))$ by \eqref{335}. Also, we note for later use that $\|\mathbbm{1}_{\Omega_{\eta,Q}}u_1\|_\infty \leq \|u\|_\infty$ and, using~\eqref{334}, that
\begin{equation}\label{eq:gradu1est}
|\nabla u_1| 
\lesssim \left| \begin{bmatrix} \nabla_x u_1 - (\nabla_x \tau)(\partial_t u_1)/J \\(\partial_t u_1)/J \end{bmatrix}\right| + |\nabla_x \tau| |\partial_t u_1|
= \left|(\nabla u)\circ \rho\right| + \left|\nabla_x \tau\right| \left|\partial_t u_1\right|
\end{equation}
on $\Omega_{\eta,Q}$.

Next, in order to work with the pull-back solution $u_1$, we consider an arbitrary constant $0<\delta\leq1/8$ and define a smooth cut-off function $\Psi_\delta$ adapted to $\Omega_{\eta,Q}$ as follows. Let $\delta_F(x):=\dist(x, F)$, fix a $C^\infty$-function $\Phi:\RR\rightarrow[0,1]$ satisfying $\Phi(t)=1$ when $t < \frac{1}{16}$ and $\Phi(t) = 0$ when $t \geq \frac{1}{8}$, and then define
\begin{equation*}
\Psi_\delta(x,t):= \Phi\left( \frac{\delta_F(x)}{\eta t}\right)\Phi\left(\frac{t}{32\ell(Q)}\right)\left(1-\Phi\left(\frac{t}{16\delta\ell(Q)}\right)\right)
\qquad\forall(x,t)\in\RR^{n+1}_+.
\end{equation*}
This function is designed so that $\Psi_\delta \equiv 1$ on $F\times (2\delta\ell(Q), 2\ell(Q))$, and since $\eta<1/2$, we have $\supp \Psi_\delta \subseteq \Omega_{\eta/8,Q,\delta}$ and
\begin{equation}\label{eq.gradpsi}
|\nabla_{x,t}\Psi_\delta(x,t)| \lesssim \frac{\mathbbm{1}_{E_1}(x,t)}{t}
+ \frac{\mathbbm{1}_{E_2}(x,t)}{\ell(Q)}
+ \frac{\mathbbm{1}_{E_3}(x,t)}{\delta\ell(Q)}
\qquad\forall(x,t)\in\Omega_{\eta/8,Q,\delta},
\end{equation}
where
\begin{align*}
E_1&:=\big\lbrace (x,t) \in 2Q\times (0, 4\ell(Q)): {\eta t}/{16} \leq \delta_F(x) \leq {\eta t}/{8}\big\rbrace,\\
E_2&:= 2Q \times (2\ell(Q), 4\ell(Q)),\\
E_3&:= 2Q \times (\delta\ell(Q), 2\delta\ell(Q)).
\end{align*}

In contrast to Section~3.2 in~\cite{HKMP1}, the cut-off function $\Psi_\delta$ introduced here incorporates an additional truncation in the $t$-direction at the boundary. This is done to simplify subsequent integration by parts arguments, since it ensures that $\Psi_\delta$ vanishes on the boundary of $\RR^{n+1}_+$. For later purposes, it is also convenient to isolate the following general fact here.

\begin{remark}\label{rem:nabla.psi.est}
For each $k\in\mathbb{Z}$, let $\mathbb{D}_{k}^{\eta}$ denote the grid of dyadic cubes $Q'\subset\RR^n$ such that $\eta2^{-k}/64 \leq \diam Q' < \eta 2^{-k}/32$. If $C_0>0$ and $(v_t)_{t>0}$ is a collection of non-negative measurable functions such that
\[
\sup_{t\in[2^{-k},2^{-k+1}]} \fint_{Q'} v_t(x)\ d\mu(x) \leq C_0
\qquad \forall k\in\mathbb{Z},\ \forall Q'\in\mathbb{D}_{k}^{\eta},
\]
then
\begin{equation}\label{eq:E123est}
\iint_{\RR^{n+1}_+} 
\left(\frac{\mathbbm{1}_{E_1}(x,t)}{t}
+ \frac{\mathbbm{1}_{E_2}(x,t)}{\ell(Q)}
+ \frac{\mathbbm{1}_{E_3}(x,t)}{\delta\ell(Q)}\right) v_t(x)\ d\mu(x)dt \lesssim C_0 \mu(Q),
\end{equation}
where the implicit constant depends only on $n$, $\lambda$, $\Lambda$ and $[\mu]_{A_2}$. To see this, first observe that since $\delta_F$ is a Lipschitz mapping with constant 1, we have
\begin{align*}
Q^{(1)} \times [2^{-k}, 2^{-k+1}] &\subseteq \widetilde{E}_1:=\left\lbrace (x,t) \in 4Q\times (0, 4\ell(Q)): \frac{\eta t}{C} \leq \delta_F(x) \leq C \eta t\right\rbrace,\\
Q^{(2)} \times [2^{-k}, 2^{-k+1}] &\subseteq 4Q\times (\ell(Q), 8\ell(Q)),\\
Q^{(3)} \times [2^{-k}, 2^{-k+1}] &\subseteq 4Q\times ((\delta/2)\ell(Q), 4\delta\ell(Q))
\end{align*}
whenever $E_i\cap (Q^{(i)}\times [2^{-k}, 2^{-k+1}]) \neq \emptyset$ and $i\in\{1,2,3\}$. The   estimate in \eqref{eq.gradpsi} and the doubling property of $\mu$ then imply that the left side of \eqref{eq:E123est} is bounded by
\begin{align*}
C_0 &\left(\sum_{k\in\mathbb{Z}}\sum_{Q' \in
 \mathbb{D}_{k}^{\eta}} \int_{2^{-k}}^{2^{-k+1}} \int_{Q'} \mathbbm{1}_{\widetilde{E}_1}\ d\mu\frac{dt}{t} + C\fint_{\ell(Q)}^{8\ell(Q)}\mu(Q)\ dt + C\fint_{(\delta/2)\ell(Q)}^{4\delta\ell(Q)} \mu(Q)\ dt \right) \\
&\lesssim C_0 \left(\int_{4Q}\int_{\frac{1}{C\eta}\delta_F(x)}^{\frac{C}{\eta}\delta_F(x)} \ \frac{dt}{t}d\mu(x) + \mu(Q) \right)
\lesssim C_0\mu(Q),
\end{align*}
as required.
\end{remark}

We now proceed to prove \eqref{eq:CMEforFinPf}. First, note that it suffices to show that
\begin{equation}\label{eq:CMEforFinPf.trunc0}
\sup_{0<\delta\leq1/8}\int_{4\delta\ell(Q)}^{\ell(Q)}\int_F |t\nabla u(x,t)|^2\ d\mu(x)\frac{dt}{t}
\lesssim \|u\|^2_{{\infty}} \mu(Q),
\end{equation}
since we may then obtain~\eqref{eq:CMEforFinPf} by using Fatou's lemma to pass to the limit as $\delta$ approaches $0$. To this end, we use \eqref{334tau}, followed by the bi-Lipschitz in $t$ change of variables $t \mapsto \tau(x,t)$ on $(\delta\ell(Q),2\ell(Q))$ for each $x\in F$, estimate \eqref{334} and identity \eqref{210}, to obtain
\begin{align*}
\int_{4\delta\ell(Q)}^{\ell(Q)}\int_F |t\nabla u(x,t)|^2\ d\mu(x)\frac{dt}{t}
&\lesssim \int_F\int_{4\delta\ell(Q)}^{\ell(Q)} \langle A\nabla u , \nabla u\rangle \ t dt dx \\
&\lesssim \int_{F}\int_{2\delta\ell(Q)}^{2\ell(Q)} \langle A_1\nabla u_1, \nabla u_1\rangle \ tdtdx \\
&\leq \iint_{\RR^{n+1}_+} \langle A_1\nabla u_1 , \nabla u_1\rangle \Psi_\delta^2\ t dx dt.
\end{align*}
Thus, in order to prove \eqref{eq:CMEforFinPf.trunc0} and ultimately \eqref{eq:CMEforFinPf}, it suffices to show that
\begin{equation}\label{eq:CMEforFinPf.trunc}
\iint_{\RR^{n+1}_+} \langle A_1\nabla u_1, \nabla u_1\rangle \Psi_\delta^2\ t dx dt
\lesssim \|u\|^2_{{\infty}} \mu(Q)
\qquad \forall 0<\delta\leq1/8,
\end{equation}
where the implicit constant depends only on $n$, $\lambda$, $\Lambda$ and $[\mu]_{A_2}$.

Next, we recall that $\div (A_1\nabla u_1)=0 $ in $\Omega_{\eta,Q}$, noting that $u_1\Psi_\delta^2t\in W^{1,2}_{0,\mu}(\Omega_{\eta,Q})$, and then integrate by parts to obtain
\begin{align}\begin{split}\label{mainest}
&\iint_{\RR^{n+1}_+} \langle A_1\nabla u_1 , \nabla u_1\rangle \Psi_\delta^2\ t dx dt 
=-\frac{1}{2}\iint_{\RR^{n+1}_+} \langle A_1\nabla(u_1^2),\nabla(\Psi_\delta^2 t)\rangle\ dx dt \\
&=-\frac{1}{2}\iint_{\RR^{n+1}_+} \langle \nabla(u_1^2), \tfrac{1}{\mu} A_1^*e_{n+1}\rangle \Psi_\delta^2 \ d\mu dt
-\frac{1}{2}\iint_{\RR^{n+1}_+} \langle A_1\nabla(u_1^2), \nabla(\Psi_\delta^2)\rangle \ tdx dt \\
&=\frac{1}{2}\iint_{\RR^{n+1}_+}u_1^2 
(L_{||,\mu}^*P_{\eta t}^*\varphi) \Psi_\delta^2\ d\mu dt 
+ \frac{1}{2}\iint_{\RR^{n+1}_+}u_1^2 
\partial_t(\langle A\mathbf{p},\mathbf{p}\rangle/J) \Psi_\delta^2\ dxdt\\
&\quad - \frac{1}{2}\iint_{\RR^{n+1}_+} \langle A_1\nabla(u_1^2), \nabla(\Psi_\delta^2)\rangle\ t dxdt  + \frac{1}{2}\iint_{\RR^{n+1}_+}u_1^2 \langle {e_{n+1}} , A_1\nabla(\Psi_\delta^2)\rangle \ dxdt \\
&=: \mathbf{S}_1 + \mathbf{S}_2 + \mathbf{E}_1 + \mathbf{E}_2,
\end{split}\end{align}
where $e_{n+1}:=(0, ..., 0, 1)$ denotes the unit vector in the $t$-direction. In particular, note that the tangential integration by parts
\[
\int_{\RR^n} \langle \nabla_x(u_1^2), \mathbf{h}-\tfrac{1}{\mu}A_{||}^*\nabla_xP_{\eta t}^*\varphi\rangle \Psi_\delta^2 \ d\mu
=\int_{\RR^n} u_1^2 \div_\mu[(\mathbf{h}-\tfrac{1}{\mu}A_{||}^*\nabla_xP_{\eta t}^*\varphi) \Psi_\delta^2] \ d\mu,
\]
with respect to the measure $\mu$, is justified by the definition of the operator $\div_\mu$, since $P_{\eta t}^*\varphi \in \Dom(L_{||,\mu}^*)$ and $\div_\mu \mathbf{h} = 0$ imply that $(\mathbf{h}-\tfrac{1}{\mu}A_{||}^*\nabla_xP_{\eta t}^*\varphi) \Psi_\delta^2 \in \Dom(\div_\mu)$ (recall~\eqref{eq:Lbdrydef}, \eqref{eq:divbdrydef} and \eqref{eq:L||mudef}). Meanwhile, the transversal integration by parts
\[
\int_0^\infty \partial_t(u_1^2)(\langle A\mathbf{p},\mathbf{p}\rangle/J) \Psi_\delta^2 \ dt
= - \int_0^\infty u_1^2 \partial_t[(\langle A\mathbf{p},\mathbf{p}\rangle/J) \Psi_\delta^2] \ dt 
\]
is justified because $\Psi_\delta$ vanishes on the boundary of $\RR^{n+1}_+$.

We proceed to prove that, for all $\sigma\in(0,1)$, each term in \eqref{mainest} is controlled by
\begin{equation}\label{eq:SEwithsigma}
\mathbf{S}_1 + \mathbf{S}_2 + \mathbf{E}_1 + \mathbf{E}_2
\lesssim \sigma\iint_{\RR^{n+1}_+}\langle A_1\nabla u_1, \nabla u_1\rangle \Psi_\delta^2\ tdxdt + \sigma^{-1}\|u\|_\infty^2\mu(Q),
\end{equation}
where the implicit constant depends only on $n$, $\lambda$, $\Lambda$ and $[\mu]_{A_2}$. Estimate \eqref{eq:CMEforFinPf.trunc} will then follow by fixing a sufficiently small $\sigma\in(0,1)$, depending only on allowed constants, to move the integral in \eqref{eq:SEwithsigma} to the left side of \eqref{mainest}. This is justified because the integral in \eqref{eq:SEwithsigma} is finite by Caccioppoli's inequality and the fact that $\Psi_\delta$ vanishes in a neighbourhood of the boundary of $\RR^{n+1}_+$ ($\supp \Psi_\delta \subseteq \Omega_{\eta/8,Q,\delta}$).

We now prove \eqref{eq:SEwithsigma} in three steps to complete the proof.

\noindent\textbf{Step 1:} Estimates for the error terms $\mathbf{E}_1$ and $\mathbf{E}_2$ in \eqref{mainest}.

We first apply Cauchy's inequality with $\sigma$ to write
\begin{align*}
\mathbf{E}_1 &\leq \left| \frac{1}{2}\iint_{\RR^{n+1}_+} \langle A_1\nabla(u_1^2),\nabla(\Psi_\delta^2)\rangle\ tdxdt\right|\\
&= 2\left|\iint_{\RR^{n+1}_+} \langle A_1\nabla u_1 ,\nabla \Psi_\delta\rangle u_1\Psi_\delta\ tdxdt\right|\\
&\lesssim \sigma\iint_{\RR^{n+1}_+}\langle A_1\nabla u_1, \nabla u_1\rangle \Psi_\delta^2\ tdxdt + \sigma^{-1} \iint_{\RR^{n+1}_+}u_1^2 \langle A_1\nabla \Psi_\delta, \nabla \Psi_\delta\rangle \ tdxdt\\
&=: \sigma\iint_{\RR^{n+1}_+}\langle A_1\nabla u_1, \nabla u_1\rangle \Psi_\delta^2\ tdxdt + \sigma^{-1}\mathbf{E}_1'.
\end{align*}
We then use $\mu\mathbf{h} = \mathbf{c}\mathbbm{1}_{5Q} + A_{||}^*\nabla \varphi$ from \eqref{Hodgeforvarphi}, the degenerate bound in \eqref{eq:deg.def} for $A$, the bound $\|\mathbbm{1}_{\Omega_{\eta,Q}}u_1\|_\infty\lesssim\|u\|_\infty$ and the estimate for $\nabla\Psi_\delta$ from \eqref{eq.gradpsi} to obtain
\[
\mathbf{E}_1' + \mathbf{E}_2
\lesssim \|u\|_\infty^2 \iint_{\Omega_{\eta/8,Q}} \left(\frac{\mathbbm{1}_{E_1}}{t}
+ \frac{\mathbbm{1}_{E_2}}{\ell(Q)}
+ \frac{\mathbbm{1}_{E_3}}{\delta\ell(Q)}\right) \left(1 + |\nabla_x(I- P_{\eta t}^*)\varphi|^2\right) d\mu dt ,
\]
where \eqref{eq.gradpsi} ensures that $|\nabla(\Psi_\delta^2)|$ and $|\nabla\Psi_\delta|^2t$ can be controlled in the same manner. In order to apply Remark~\ref{rem:nabla.psi.est} with $v_t=\mathbbm{1}_{\Omega_{\eta/8,Q}}(1 + |\nabla_x(I- P_{\eta t}^*)\varphi|^2)$, we observe that if $k\in\mathbb{Z}$, $Q'\in\mathbb{D}_{k}^{\eta}$ and $\Omega_{\eta/8,Q,\delta}\cap (Q'\times [2^{-k}, 2^{-k+1}]) \neq \emptyset$, then there exists $x_0 \in F$ such that $Q' \subseteq \Delta(x_0,\eta 2^{-k}) \subseteq C Q'$, where $\Delta$ is used to  denote balls in $\RR^n$, hence
\begin{equation}\label{eq:sawtoothcontainment}
Q'\times [2^{-k}, 2^{-k+1}] \subseteq \Omega_{\eta,2Q,\delta/4}
\end{equation}
and the doubling property of $\mu$ implies that
\begin{multline}\label{eq:gradtauest}
\fint_{Q'}|\nabla_x (I - P_{\eta t}^*)\varphi|^2\ d\mu
\lesssim \fint_{\Delta(x_0,\eta t)}| \nabla_x P_{\eta t}^{*}\varphi|^2\ d\mu + \fint_{\Delta(x_0,\eta 2^{-k})}| \nabla_x \varphi|^2\ d\mu\\
\lesssim \big[\widetilde{N}_{*,\mu}^{\eta}(\nabla_xP_{\eta t}^*\varphi)(x_0)\big]^2 + M_{\mu}( |\nabla_x \varphi|^2)(x_0) \lesssim \kappa_0^2 \lesssim 1
\qquad \forall t \in [2^{-k}, 2^{-k+1}],
\end{multline}
where in the last line we used the definition of the set $F$ in \eqref{F} and the weighted maximal operators $\widetilde{N}_{*,\mu}$ and $M_{\mu}$ from Section~\ref{sec:maxops}. It thus follows from \eqref{eq:E123est} that $\mathbf{E}_1' + \mathbf{E}_2 \lesssim \|u\|^2_{\infty} \mu(Q)$, so altogether we have
\begin{equation}\label{eq:E1E2}
\mathbf{E}_1 + \mathbf{E}_2
\lesssim \sigma\iint_{\RR^{n+1}_+}\langle A_1\nabla u_1, \nabla u_1\rangle \Psi_\delta^2\ tdxdt + \sigma^{-1}\|u\|_\infty^2\mu(Q)
\quad\forall\sigma\in(0,1).
\end{equation}

\noindent\textbf{Step 2:} Estimates for the term $\mathbf{S}_1$ in  \eqref{mainest}.

We note that $\partial_t P_{\eta t}^* = -2\eta^2 t L_{||,\mu}^*P_{\eta t}^*$ on $L^2_\mu(\RR^n)$ and integrate by parts in $t$ to write
\begin{align*}
\mathbf{S}_1 &= \frac{1}{2}\iint_{\RR^{n+1}_+}u_1^2 
(L_{||,\mu}^*P_{\eta t}^*\varphi) \Psi_\delta^2\ d\mu dt \\
&=-\frac{1}{2}\iint_{\RR^{n+1}_+}u_1^2\partial_t(L_{||,\mu}^*P_{\eta t}^*\varphi)\Psi_\delta^2\ td\mu dt + \frac{1}{2\eta^2}\iint_{\RR^{n+1}_+}(u_1\partial_t u_1)(\partial_tP_{\eta t}^*\varphi) \Psi_\delta^2\ d\mu dt \\
&\quad+\frac{1}{2\eta^2}\iint_{\RR^{n+1}_+}u_1^2(\partial_tP^{*}_{\eta t}\varphi) \Psi_\delta\partial_t\Psi_\delta\ d\mu dt
=: \mathbf{S}_1'+ \mathbf{S}_1'' + \mathbf{S}_1''',
\end{align*}
where there is no boundary term because $\Psi_\delta$ vanishes on the boundary of $\RR^{n+1}_+$.

To estimate $\mathbf{S}_1'''$, we use the definition of the set $F$ in \eqref{F}, the estimate for $|\nabla\Psi_\delta|$ from~\eqref{eq.gradpsi}, and Remark~\ref{rem:nabla.psi.est} in the case $v_t\equiv1$, to obtain
\[
\mathbf{S}_1'''
\lesssim \|u\|^2_{{\infty}}\iint_{\Omega_{\eta/8,Q}} N_*^\eta(\partial_t P_{\eta t}^* \varphi)\, |\partial_t \Psi_\delta| \ d\mu dt
\lesssim \eta\kappa_0\|u\|^2_{{\infty}}\mu(Q) \lesssim \|u\|^2_{{\infty}}\mu(Q).
\]

To estimate $\mathbf{S}_1'$, we observe that $\partial_t(L_{||,\mu}^*P_{\eta t}^*\varphi) = L_{||,\mu}^* (\partial_t P_{\eta t}^*\varphi)$, since $\varphi\in W^{1,2}_{\mu,0}(\RR^n)$ and  $\partial_t P_{\eta t}^* = -2\eta^2 t P_{\eta t}^*L_{||,\mu}^*$ on the dense subset $\Dom(L_{||,\mu}^*)$ of $W^{1,2}_{0,\mu}(\RR^n)$ (note also that $t\nabla_xP_{\eta t}^*$ and hence its adjoint are bounded operators on $L^2_\mu$, as can be seen from the proof of Theorem~\ref{VSFbound}). We then apply Cauchy's inequality with $\sigma$ to write
\begin{align} \label{355}
\begin{split}
\mathbf{S}_1' &\leq \left| \iint_{\RR^{n+1}_+} L_{||,\mu}^* (\partial_t P_{\eta t}^*\varphi) u_1^2\Psi_\delta^2\ td\mu dt\right|\\
&\lesssim
\left| \iint_{\RR^{n+1}_+} \langle \tfrac{1}{\mu}A_{||}^*\nabla_x (\partial_t P_{\eta t}^*\varphi) , \nabla_x u_1\rangle   u_1\Psi_\delta^2\ td\mu dt\right|  \\
&\quad+ 
 \left| \iint_{\RR^{n+1}_+} \langle\tfrac{1}{\mu} A_{||}^*\nabla_x (\partial_t P_{\eta t}^*\varphi) , \nabla_x\Psi_\delta\rangle u_1^2\Psi_\delta\ td\mu dt\right|
=: \mathbf{J} + \mathbf{K} \\
&\lesssim \sigma \iint_{\RR^{n+1}_+} \left| \nabla_x u_1\right|^2\Psi_\delta^2\ t d\mu dt
+ (\sigma^{-1} +1)\iint_{\RR^{n+1}_+}u_1^2|\nabla_x\partial_tP_{\eta t}^*\varphi|^2 \Psi_\delta^2\ t d\mu dt\\
&\quad+ \iint_{\RR^{n+1}_+}u_1^2\left| \nabla_x\Psi_\delta\right|^2\ t d\mu dt
=: \sigma\mathbf{S}_{11}' + (\sigma^{-1}+1)\mathbf{S}_{12}' + \mathbf{S}_{13}',
\end{split}
\end{align}
where the integration by parts in $x$, with respect to the measure $\mu$, is justified by the definition of the operator $L_{||,\mu}^*$ (recall~\eqref{eq:Lbdrydef}, \eqref{eq:divbdrydef} and \eqref{eq:L||mudef}). The terms~$\mathbf{J}$ and~$\mathbf{K}$ are highlighted above for reference in Step 3.

To estimate $\mathbf{S}_{13}'$, we use the estimate for $|\nabla\Psi_\delta|$ from~\eqref{eq.gradpsi} and Remark~\ref{rem:nabla.psi.est} in the case $v_t\equiv1$, to obtain $ \mathbf{S}_{13}' \lesssim \|u\|^2_{{\infty}}\mu(Q)$.
 
To estimate $\mathbf{S}_{12}'$, we observe that $\nabla_x\partial_t P_{\eta t}^* = -2\eta^2 t \nabla_xL_{||,\mu}^*P_{\eta t}^*$ on $L^2_\mu(\RR^n)$ and then apply the vertical square function estimate from~\eqref{eq:VSF2} followed by the $W^{1,2}_{0,\mu}(5Q)$-Hodge estimate for $\varphi$ from \eqref{HodEst1} to obtain
\begin{align*}
\mathbf{S}_{12}' &\lesssim \iint_{\RR^{n+1}_+} u_1^2|\nabla_x\partial_tP_{\eta t}^*\varphi|^2\Psi_\delta^2\ t d\mu dt
\lesssim \|u\|^2_{{\infty}}\iint_{\RR^{n+1}_+} |t^2 \nabla_x L_{||,\mu}^*P_{\eta t}^*\varphi|^2\ d\mu\frac{dt}{t} \\
&\lesssim \|u\|^2_{{\infty}} \|\nabla \varphi\|_{L^2_{\mu}(\RR^n,\RR^n)}^2
\lesssim \|u\|^2_{{\infty}} \mu(Q).
\end{align*}

The terms $\mathbf{S}_{11}'$ and $\mathbf{S}_{1}''$ will now be estimated together. We again apply Cauchy's inequality with $\sigma$, followed by the vertical square function estimate from~\eqref{eq:VSF1} with $\mathcal{L}_\mu = L_{||,\mu}^*$  and the $W^{1,2}_{0,\mu}(5Q)$-Hodge estimate for $\varphi$ from \eqref{HodEst1} to obtain
\begin{align*}
&\sigma\mathbf{S}_{11}'+ \mathbf{S}_{1}'' \lesssim  \sigma \iint_{\RR^{n+1}_+} \left| \nabla_x u_1\right|^2\Psi_\delta^2\ t d\mu dt
+ \left|\iint_{\RR^{n+1}_+} (u_1\partial_t u_1) (\partial_tP_{\eta t}^{*}\varphi)\Psi_\delta^2\ d\mu dt \right|
\\
&\lesssim \sigma \iint_{\RR^{n+1}_+} |\nabla u_1|^2 \Psi_\delta^2\ t d\mu dt + \sigma^{-1} \|u\|^2_{{\infty}} \iint_{\RR^{n+1}_+} |\partial_t P_{\eta t}^{*}\varphi|^2 \ d\mu\frac{dt}{t}\\
&\lesssim \sigma \iint_{\RR^{n+1}_+} \langle A_1\nabla u_1,\! \nabla u_1\rangle \Psi_\delta^2\ t dxdt +\sigma \iint_{\RR^{n+1}_+} \left|\nabla_x \tau\right|^2\! \left|\partial_t u_1\right|^2\! \Psi_\delta^2\ t d\mu dt
+\sigma^{-1}\|u\|^2_{\infty} \mu(Q),
\end{align*}
where we combined the pointwise estimates for $\nabla u_1$ and $J$ from \eqref{eq:gradu1est} and \eqref{334} with identity \eqref{210} and the ellipticity of $A$ to deduce the final inequality.

We use the dyadic decomposition from Remark~\ref{rem:nabla.psi.est} to write
\begin{align}\label{362}
\iint_{\RR^{n+1}_+} |\nabla_x \tau|^2 |\partial_t u_1|^2 \Psi_\delta^2\ t d\mu dt
\leq\sum_{k\in\mathbb{Z}}\sum_{Q' \in \mathbb{D}_k^{\eta}}\int_{2^{-k}}^{2^{-k+1}}\!\!\!\! \int_{Q'}\!\!\mathbbm{1}_{\Omega_{\eta,Q,\delta}}|\nabla_x \tau|^2 |\partial_t u_1|^2 \ t d\mu dt.
\end{align}
Observe that if $k\in\mathbb{Z}$, $Q'\in\mathbb{D}_{k}^{\eta}$ and $\Omega_{\eta/8,Q,\delta}\cap (Q'\times [2^{-k}, 2^{-k+1}]) \neq \emptyset$, then as in \eqref{eq:sawtoothcontainment} and \eqref{eq:gradtauest}, it holds that $Q'\times [2^{-k}, 2^{-k+1}] \subseteq \Omega_{\eta,2Q,\delta/4}$ and
\[
\fint_{Q'}|\nabla_x \tau(x,t)|^2\ d\mu(x)
\lesssim \kappa_0^2
\qquad \forall t \in [2^{-k}, 2^{-k+1}].
\]
Also, we have $\frac{7}{8}t < \tau(x,t) < \frac{9}{8}t$ and $J(x,t)\eqsim 1$ on $Q'\times [2^{-k}, 2^{-k +1}]$ by \eqref{334}, so the degenerate version of Moser's estimate in \eqref{eq:M} and $t$-independence show that
\begin{align*}
\sup_{x \in Q'}|\partial_t u_1(x,t)|^2
= \sup_{x \in Q'}\left|J(x,t) \partial_{\tau} u(x, \tau(x,t))\right|^2
\lesssim \fint_{2Q'}\fint_{{t}/{2}}^{2t}\left| \partial_s u(y,s)\right|^2\ dsd\mu(y)
\end{align*}
for all $t \in [2^{-k}, 2^{-k+1}]$. In particular, note that
\[
2Q' \times [2^{-k-1}, 2^{-k+2}] \subseteq \Omega^* :=\left\lbrace (y,s) \in \RR^{n+1}_+: \delta_F(y) < \frac{5\eta s}{8},\ \frac{\delta}{2}\ell(Q) < s < 8\ell(Q)\right\rbrace,
\]
since there exists $(x_0,t_0) \in Q'\times [2^{-k}, 2^{-k+1}]$ satisfying $\delta_F(x_0) < \frac{1}{8}\eta t_0$, whence
\begin{equation*}
\delta_F(y) < \diam (2Q') + \frac{1}{8}\eta t_0 \leq \frac{5}{16}\eta 2^{-k} \leq \frac{5}{8}\eta s
\qquad \forall y \in 2Q' \mbox{ and } s \geq 2^{-k-1},
\end{equation*}
whilst $\delta\ell(Q)<t_0 < 4\ell(Q)$ implies that $[2^{-k}, 2^{-k+1}] \subseteq ((\delta/2)\ell(Q),8\ell(Q))$.

The observations in the preceding paragraph show that \eqref{362} is bound by
\begin{align*}
\begin{split}
&\sum_{k\in\mathbb{Z}}\sum_{Q'\in \mathbb{D}_k^{\eta}}\int_{2^{-k}}^{2^{-k+1}}\left(\fint_{Q'}|\nabla_x \tau|^2 d\mu\right)\left(\int_{2Q'}\int_{{t}/{2}}^{2t}\left| \partial_s u(y,s)\right|^2\mathbbm{1}_{\Omega^*}(y,s)\ dsd\mu(y)\right)dt\\
&\quad \lesssim \sum_{k\in\mathbb{Z}}\sum_{Q'\in \mathbb{D}_k^{\eta}}\int_{2^{-k-1}}^{2^{-k+2}}\int_{2Q'}\left|\partial_s u(y,s)\right|^2\mathbbm{1}_{\Omega^*}(y,s)\ s d\mu(y)ds\\
&\quad \lesssim \left( \iint_{\Omega^{**}}\left| \partial_s u(y,s)\right|^2\ s d\mu(y)ds + \iint_{\Omega^*\setminus \Omega^{**}}\left| \partial_s u(y,s)\right|^2\ s d\mu(y)ds\right)
:= \mathbf{M} + \mathbf{E},
\end{split}
\end{align*}
where we used the fact that $\sum_{k\in\mathbb{Z}}\sum_{Q'\in \mathbb{D}_k^{\eta}} \mathbbm{1}_{2Q' \times [2^{-k-1}, 2^{-k+2}]} \lesssim \mathbbm{1}_{\RR^{n+1}_+}$ and  introduced
\[
\Omega^{**}:=\left\lbrace (y,s) \in \RR^{n+1}_+: \delta_F(y) < \eta s/18,\ 4\delta\ell(Q) < s < \ell(Q)\right\rbrace.
\]

To estimate the main term $\mathbf{M}$, we use \eqref{334}-\eqref{335} to observe that
\[
\rho^{-1}(\Omega^{**}) \subseteq \Omega_{\frac{\eta}{16}} \cap (2Q \times (2\delta\ell(Q), 2 \ell(Q))).
\]
Thus, since $\Psi_\delta \equiv 1$ on these sets, the change of variables $(y,s)\mapsto \rho(y,s)$ gives
\begin{equation*}
\mathbf{M} \lesssim \iint_{\RR^{n+1}_+}\left| (\partial_t u)\circ \rho \right|^2 J\, \Psi_\delta^2\ t d\mu dt
\lesssim \iint_{\RR^{n+1}_+} \langle A_1\nabla u_1, \nabla u_1\rangle \Psi_\delta^2\ t dx dt,
\end{equation*}
where we used identity \eqref{210} and the ellipticity of $A$ to deduce the final inequality.

To estimate the error term $\mathbf{E}$, recall that the degenerate version of Moser's estimate in \eqref{eq:M}, followed by Caccioppoli's inequality, ensures that $\|s\partial_s u\|_\infty \lesssim \|u\|_\infty$. Thus, by the definition of $\Omega^* \setminus \Omega^{**}$ and the doubling property of $\mu$, we obtain
\begin{align*} 
\mathbf{E} \lesssim \|u\|_{\infty}^2 \int_{2Q}\left( \int_{\frac{8}{5\eta}\delta_F(y)}^{\frac{18}{\eta}\delta_F(y)} \frac{ds}{s} + \int_{\ell(Q)}^{8 \ell(Q)}\frac{ds}{s} + \int_{(\delta/2)\ell(Q)}^{4\delta\ell(Q)}\frac{ds}{s}\right) d\mu(y)
\lesssim \|u\|_{\infty}^2\mu(Q).
\end{align*} 

This shows that $\sigma \mathbf{S}_{11}'+ \mathbf{S}_{1}''
\lesssim \sigma \iint_{\RR^{n+1}_+} \langle A_1\nabla u_1, \nabla u_1\rangle \Psi_\delta^2\ t dx dt +\sigma^{-1} \|u\|^2_{\infty} \mu(Q)$, hence
\begin{equation}\label{eq:S1}
\mathbf{S}_1
\lesssim \sigma \iint_{\RR^{n+1}_+} \langle A_1\nabla u_1, \nabla u_1\rangle \Psi_\delta^2\ t dx dt 
+\sigma^{-1} \|u\|^2_{\infty} \mu(Q) \quad\forall\sigma\in(0,1).
\end{equation}

\noindent\textbf{Step 3:} Estimates for the term $\mathbf{S}_2$ in  \eqref{mainest}.

We observe that since $A$ is $t$-independent it is possible to write
\begin{align*}
2\mathbf{S}_2 &= \iint_{\RR^{n+1}_+} u_1^2\partial_t (\langle A\mathbf{p}, \mathbf{p}\rangle/J)\Psi_\delta^2\ dx dt\\
&=\iint_{\RR^{n+1}_+} u_1^2\partial_t (1/J)\langle A\mathbf{p}, \mathbf{p}\rangle\Psi_\delta^2\ dx dt
+ \iint_{\RR^{n+1}_+} (u_1^2/J) \langle \partial_t\mathbf{p}, A^*\mathbf{p}\rangle\Psi_\delta^2\ dx dt\\
&\quad+\iint_{\RR^{n+1}_+} (u_1^2/J) \langle A\mathbf{p},  \partial_t\mathbf{p}\rangle\Psi_\delta^2\ dx dt
=: \mathbf{I} + \mathbf{II} + \mathbf{III}.
\end{align*}

To estimate $\mathbf{I}$, we recall the Jacobian $J(x,t) = 1 +\partial_t P_{\eta t}^* \varphi(x)$ from \eqref{eq:Jdef} and then integrate by parts in $t$ to write
\begin{align*}
\mathbf{I} &= -\iint_{\RR^{n+1}_+} u_1^2 \frac{\partial_t^2P_{\eta t}^*\varphi}{J^2}\langle A\mathbf{p}, \mathbf{p}\rangle\Psi_\delta^2\ dx dt\\
&= \iint_{\RR^{n+1}_+} \partial_t (u_1^2) 
\frac{\partial_tP_{\eta t}^*\varphi}{J^2}\langle A\mathbf{p}, \mathbf{p}\rangle\Psi_\delta^2\ dx dt
+ \iint_{\RR^{n+1}_+} u_1^2 \frac{\partial_tP_{\eta t}^*\varphi}{J^2}\partial_t(\langle A\mathbf{p}, \mathbf{p}\rangle)\Psi_\delta^2\ dx dt \\
&\quad +  \iint_{\RR^{n+1}_+} u_1^2 \partial_tP_{\eta t}^*\varphi\,\partial_t (J^{-2})\langle A\mathbf{p}, \mathbf{p}\rangle\Psi_\delta^2\ dx dt  +  \iint_{\RR^{n+1}_+} u_1^2 \frac{\partial_tP_{\eta t}^*\varphi}{J^2}\langle A\mathbf{p}, \mathbf{p}\rangle\partial_t(\Psi_\delta^2)\ dx dt \\
&=: \mathbf{I}_1 + \mathbf{I}_2 + \mathbf{I}_3 + \mathbf{I}_4,
\end{align*}
where there is no boundary term because $\Psi_\delta$ vanishes on the boundary of $\RR^{n+1}_+$.

To estimate $\mathbf{I}_1$, we recall that $J\eqsim 1$ on $\supp \Psi_\delta \subseteq \Omega_{\eta/8,Q,\delta}$ by \eqref{334} and then apply Cauchy's inequality with $\sigma$ to obtain
\begin{align}\begin{split}\label{365}
|\mathbf{I}_1|
&\lesssim \sigma\iint_{\RR^{n+1}_+}|\partial_t u_1|^2|\mathbf{p}|^2\Psi_\delta^2\ td\mu dt \\
&\quad + \sigma^{-1}\iint_{\RR^{n+1}_+}u_1^2|\partial_tP_{\eta t} \varphi|^2|\mathbf{p}|^2\Psi_\delta^2 \ d\mu\frac{dt}{t}
=: \sigma\mathbf{I}_1' + \sigma^{-1}\mathbf{I}_1''.
\end{split}\end{align}

To estimate $\mathbf{I}_1'$, recall that $|\mathbf{p}|^2=1+|\nabla_x\tau|^2$ by the definition of $\mathbf{p}$ in \eqref{eq:pdef}, so we follow the treatment of \eqref{362} above to obtain
\[
\mathbf{I}_1' 
\lesssim  \iint_{\RR^{n+1}_+} \langle A_1\nabla u_1, \nabla u_1\rangle \Psi_\delta^2\ t dx dt + \|u\|^2_{\infty} \mu(Q).
\]

To estimate $\mathbf{I}_1''$, recall that $\|\mathbbm{1}_{\Omega_{\eta,Q}}u_1\|_\infty\lesssim\|u\|_\infty$ and use the dyadic decomposition from Remark~\ref{rem:nabla.psi.est} to obtain
\begin{align}\begin{split}\label{eq:partialPtEst}
\mathbf{I}_1''
&\lesssim \|u\|_{\infty}^2 \sum_{k\in\ZZ}\sum_{Q' \in \mathbb{D}_{k}^{\eta}}
 \|\partial_t P_{\eta t}^*\varphi\|^2_{L^\infty(Q'\times[2^{-k},2^{-k+1}])}
\int_{2^{-k}}^{2^{-k +1}} \int_{Q'} |\mathbf{p}|^2\Psi_\delta^2\ d\mu\frac{dt}{t} \\
&\lesssim \|u\|_{\infty}^2 \sum_{k\in\ZZ}\sum_{Q' \in \mathbb{D}_{k}^{\eta}}
\mu(Q') \|\partial_t P_{\eta t}^*\varphi\|^2_{L^\infty(Q'\times[2^{-k},2^{-k+1}])}
\\
&\lesssim \|u\|_{\infty}^2 \sum_{k\in\ZZ}\sum_{Q' \in \mathbb{D}_{k}^{\eta}}
\mu(Q')\fint_{2^{-k-1}}^{2^{-k+2}} \fint_{2Q'} |\partial_tP_{\eta t}^*\varphi |^2\ d\mu dt \\
&\lesssim \|u\|_{\infty}^2 \sum_{k\in\ZZ}\sum_{Q' \in \mathbb{D}_{k}^{\eta}} \int_{2^{-k-1}}^{2^{-k+2}} \int_{2Q'} |\partial_tP_{\eta t}^*\varphi |^2\ d\mu\frac{dt}{t}\\
&\lesssim \|u\|_{\infty}^2\iint_{\RR^{n+1}_+}|tL_{||,\mu}^*e^{-t^2L_{||,\mu}^*}\varphi|^2\ d\mu \frac{dt}{t} \\
&\lesssim \|u\|_{\infty}^2\|\nabla \varphi\|^2_{L^2_{\mu}(\RR^n,\RR^n)} 
\lesssim \|u\|_{\infty}^2 \mu(Q),
\end{split}\end{align}
where the second line uses the pointwise bound $|\mathbf{p}|^2\Psi_\delta^2 \leq \mathbbm{1}_{\Omega_{\eta/8,Q,\delta}}(1 + |\nabla_x(I- P_{\eta t}^*)\varphi|^2)$ and estimate \eqref{eq:gradtauest}, the third line uses the parabolic version of the degenerate Moser-type estimate in \eqref{eq:M} (see Theorem B in \cite{F}), noting that $v:=\partial_t(e^{-tL_{||,\mu}^*}\varphi)$ solves $\partial_t v = - L_{||,\mu}^*v$ whilst $|\partial_tP_{\eta t}^*\varphi(x)| \lesssim |t\, v(x,\eta^2 t^2)|$, and the final line uses the vertical square function estimate from~\eqref{eq:VSF1} with $\mathcal{L}_\mu = L_{||,\mu}^*$ and the $W^{1,2}_{0,\mu}(5Q)$-Hodge estimate for $\varphi$ from \eqref{HodEst1}.

To estimate $\mathbf{I}_2$, we again use the bound $J\eqsim 1$ on $\supp \Psi_\delta \subseteq \Omega_{\eta/8,Q,\delta}$ from \eqref{334}, and then recall the definition $\mathbf{p}:= (\nabla_x(P_{\eta t}^*-I)\varphi, -1)$ from \eqref{eq:pdef} to obtain
\begin{equation}\label{366}
| \mathbf{I}_2|
\lesssim \iint_{\RR^{n+1}_+} u_1^2 |\nabla_x \partial_t P_{\eta t}^*\varphi|^2\Psi_\delta^2\ td\mu dt
+ \iint_{\RR^{n+1}_+} u_1^2 |\partial_t P_{\eta t}^*\varphi|^2|\mathbf{p}|^2\Psi_\delta^2 \ d\mu\frac{dt}{t}.
\end{equation}
The first integral in \eqref{366} is the same as $\mathbf{S}_{12}'$ from \eqref{355} whilst the second integral is the same as $\mathbf{I}_1''$ from \eqref{365}, hence $| \mathbf{I}_2 | \lesssim \|u\|_{\infty}^2\mu(Q)$.

To estimate $\mathbf{I}_3$, we use the bound $|\partial_tP_{\eta t}^*\varphi|<1/8$ guaranteed by \eqref{eq:CoV1} to deduce that $|\partial_t(J^{-2})| = |\partial_t(1 +\partial_t P_{\eta t}^* \varphi)^{-2}| \lesssim |\partial_t^2P_{\eta t}^*\varphi |$ on $\supp \Psi_\delta \subseteq \Omega_{\eta/8,Q,\delta}$ and write
 \begin{equation*}
|\mathbf{I}_3|
\lesssim \iint_{\RR^{n+1}_+} u_1^2|\partial_tP_{\eta t}^*\varphi|^2|\mathbf{p}|^2\Psi_\delta^2\ d\mu\frac{dt}{t}
+ \iint_{\RR^{n+1}_+} u_1^2 |\partial_t^2P_{\eta t}^*\varphi |^2|\mathbf{p}|^2\Psi_\delta^2\  td\mu dt
=: \mathbf{I}_3' + \mathbf{I}_3''
\end{equation*} 

To estimate $\mathbf{I}_3'$, we note that it is the same as $\mathbf{I}_1''$ from \eqref{365}, thus $\mathbf{I}_3'\lesssim \|u\|_{\infty}^2\mu(Q)$.

To estimate $\mathbf{I}_3''$, we follow the estimates and justification provided for \eqref{eq:partialPtEst}, noting in addition that $\partial_tv=\partial_t^2(e^{-tL_{||,\mu}^*}\varphi)$ solves $\partial_t (\partial_tv) = - L_{||,\mu}^*(\partial_tv)$, to obtain
\begin{align*}
\mathbf{I}_3''
&\lesssim \|u\|_{\infty}^2 \sum_{k\in\ZZ}\sum_{Q' \in \mathbb{D}_{k}^{\eta}} \|t\partial_t^2P_{\eta t}^*\varphi\|^2_{L^\infty(Q'\times[2^{-k},2^{-k+1}])} \int_{2^{-k}}^{2^{-k+1}}\int_{Q'} |\mathbf{p}|^2\Psi_\delta^2\ d\mu \frac{dt}{t}\\
&\lesssim \|u\|_{\infty}^2 \sum_{k\in\ZZ}\sum_{Q' \in \mathbb{D}_{k}^{\eta}} \mu(Q') \|\, |tL_{||,\mu}^*P_{\eta t}^*\varphi| + |t^2\partial_t(L_{||,\mu}^*P_{\eta t}^*\varphi)|\, \|^2_{L^\infty(Q'\times[2^{-k},2^{-k+1}])}\\
&\lesssim \|u\|_{\infty}^2 \sum_{k\in\ZZ}\sum_{Q' \in \mathbb{D}_{k}^{\eta}}
\mu(Q') \fint_{2^{-k-1}}^{2^{-k+2}} \fint_{2Q'} (|tL_{||,\mu}^*P_{\eta t}^*\varphi|^2 + |t^2\partial_t(L_{||,\mu}^*P_{\eta t}^*\varphi)|^2)\ d\mu dt \\
&\lesssim \|u\|_{\infty}^2\iint_{\RR^{n+1}_+}|tL_{||,\mu}^*P_{\eta t}^*\varphi|^2\ d\mu\frac{dt}{t}
+ \|u\|_{\infty}^2\iint_{\RR^{n+1}_+}|t^2 \nabla_{x,t}(L_{||,\mu}^*P_{\eta t}^*\varphi)|^2\ d\mu\frac{dt}{t} \\
&\lesssim  \|u\|_{\infty}^2\|\nabla \varphi\|_{L^{2}_{\mu}(\RR^n,\RR^n)}^2 
\lesssim \|u\|_{\infty}^2\mu(Q),
\end{align*}
where the second line uses $|\partial_t^2P^{*}_{\eta t}\varphi| \lesssim |\partial_t(tL_{||,\mu}^*P_{\eta t}^*\varphi)| \lesssim |L_{||,\mu}^*P_{\eta t}^*\varphi| + |t\partial_t(L_{||,\mu}^*P_{\eta t}^*\varphi)|$, the third line uses $|L_{||,\mu}^*P_{\eta t}^*\varphi(x)|=|v(x,\eta^2 t^2)|$ and $|\partial_t(L_{||,\mu}^*P_{\eta t}^*\varphi)(x)|\lesssim |t (\partial_t v)(x,\eta^2 t^2)|$, and the final line uses the vertical square function estimates from \eqref{eq:VSF1} and \eqref{eq:VSF2} with $\mathcal{L}_\mu = L_{||,\mu}^*$, hence $| \mathbf{I}_3 | \lesssim \|u\|_{\infty}^2\mu(Q)$

To estimate $\mathbf{I}_4$, we use $|\partial_tP_{\eta t}^*\varphi|\lesssim 1$, $J\eqsim 1$ and $|\mathbf{p}|^2 \leq (1 + |\nabla_x(I- P_{\eta t}^*)\varphi|^2)$, which hold on $\supp \Psi_\delta \subseteq \Omega_{\eta/8,Q,\delta}$ by \eqref{eq:CoV1}, \eqref{334} and \eqref{eq:pdef}, to reduce to the estimate obtained for $\mathbf{E}_1'+\mathbf{E}_2$, hence $|\mathbf{I}_4| \lesssim \|u\|_{\infty}^2\mu(Q)$.

To estimate $\mathbf{II}$, we use the definition $\mathbf{p}:= (\nabla_x(P_{\eta t}^*-I)\varphi, -1)$ from \eqref{eq:pdef} to note that $\partial_t \mathbf{p} = (\nabla_x\partial_tP_{\eta t}^*\varphi, 0 )$ and use the Hodge decomposition from~\eqref{Hodgeforvarphi} to write
\begin{align}\label{367}
\begin{split}
\langle \partial_t\mathbf{p}, A^*\mathbf{p}\rangle
&= \langle \nabla_x\partial_tP_{\eta t}^*\varphi, A_{||}^*\nabla_x(P_{\eta t}^*\!-\!I)\varphi - \mathbf{c}\rangle 
= \langle \nabla_x\partial_tP_{\eta t}^*\varphi, A_{||}^*\nabla_xP_{\eta t}^*\varphi - \mu\mathbf{h}\rangle
\end{split}
\end{align}
for all $x \in 5Q$ and $t>0$. Using this and recalling that $\div_\mu\mathbf{h}=0$, it follows that
\begin{align}\begin{split}\label{eq:II}
\mathbf{II} &= \iint_{\RR^{n+1}_+} (u_1^2/J)\langle \nabla_x\partial_tP_{\eta t}^*\varphi, A_{||}^*\nabla_xP_{\eta t}^*\varphi - \mu\mathbf{h}\rangle \Psi_\delta^2\ dx dt \\
&= \iint_{\RR^{n+1}_+ } (u_1^2/J)(\partial_t P_{\eta t}^*\varphi) (L_{||,\mu}^* P_{\eta t}^*\varphi)\Psi_\delta^2\ d\mu dt\\
&\quad - \iint_{\RR^{n+1}_+ }\partial_t P_{\eta t}^*\varphi \langle\nabla_x(u_1^2/J), A_{||}^*\nabla_x P_{\eta t}^*\varphi - \mu\mathbf{h}\rangle\Psi_\delta^2\ dx dt\\
&\quad -\iint_{\RR^{n+1}_+ } (u_1^2/J)\partial_t P_{\eta t}^*\varphi \langle\nabla_x(\Psi_\delta^2), A_{||}^*\nabla_x P_{\eta t}^*\varphi - \mu\mathbf{h}\rangle\ dx dt\\
&=: \mathbf{II}_1 + \mathbf{II}_2 + \mathbf{II}_3,
\end{split}\end{align}
where the integration by parts in $x$, with respect to the measure $\mu$, is justified by the definition of the operator $L_{||,\mu}^*$ (recall~\eqref{eq:Lbdrydef}, \eqref{eq:divbdrydef} and \eqref{eq:L||mudef}). 

To estimate $\mathbf{II}_1$, we use $J\eqsim 1$ and $L_{||,\mu}^*P_{\eta t}^*\varphi = -(2\eta^2 t)^{-1}\partial_t P_{\eta t}^*\varphi$ to show that it can be treated the same way as $\mathbf{I}_1''$ in \eqref{365}, without $|\mathbf{p}|^2$, hence $|\mathbf{II}_1| \lesssim \|u\|_{\infty}^2 \mu(Q)$.

To estimate $\mathbf{II}_2$, we use $J\eqsim 1$, $|\nabla_x (J^{-1})| = |\nabla_x(1 +\partial_t P_{\eta t}^* \varphi)^{-1}| \lesssim |\nabla_x \partial_t P_{\eta t}^*\varphi|$ and apply Cauchy's inequality inequality with $\sigma$  to obtain
\begin{align}\begin{split}\label{368}
|\mathbf{II}_2|
&\lesssim \sigma \iint_{\RR^{n+1}_+} |\nabla_x u_1|^2\Psi_\delta^2\ td\mu dt
+  \iint_{\RR^{n+1}_+} u_1^2|\nabla_x \partial_tP_{\eta t}^*\varphi|^2 \Psi_\delta^2\ t d\mu dt \\
&\quad +(\sigma^{-1}+1)\iint_{\RR^{n+1}_+} u_1^2|\partial_t P_{\eta t}^*\varphi|^2(|\nabla_xP_{\eta t}^*\varphi|^2 + |\mathbf{h}|^2)\Psi_\delta^2\ d\mu\frac{dt}{t}.
\end{split}\end{align}
The first integral is the same as $\mathbf{S}_{11}'$ from \eqref{355} whilst the remaining two integrals are the same as those that bound $\mathbf{I}_2$ in \eqref{366}, except $(|\nabla_x P_{\eta t}^*\varphi|^2+|\mathbf{h}|^2)$ replaces $|\mathbf{p}|^2$. This factor is controlled in the same way, however, since the Hodge decomposition in~\eqref{Hodgeforvarphi} implies that $|\mathbf{h}|^2 = |\tfrac{1}{\mu}\mathbf{c}\mathbbm{1}_{5Q} + \tfrac{1}{\mu}A_{||}^*\nabla_x \varphi|^2 \lesssim 1+|\nabla_x\varphi|^2$, so by \eqref{eq:gradtauest} we obtain
\[
|\mathbf{II}_2|
\lesssim \sigma \iint_{\RR^{n+1}_+} \langle A_1\nabla u_1, \nabla u_1\rangle \Psi_\delta^2\ t dx dt + \sigma^{-1} \|u\|^2_{\infty} \mu(Q).
\]

To estimate $\mathbf{II}_3$, we use $J \eqsim 1$ and Cauchy's inequality to write
\begin{equation*}
|\mathbf{II}_3| \lesssim \iint_{\RR^{n+1}_+} u_1^2 |\nabla_x \Psi_\delta|^2\ t d\mu dt
+ \iint_{\RR^{n+1}_+} u_1^2 |\partial_t P_{\eta t}^*\varphi|^2 (|\nabla_xP_{\eta t}^*\varphi|^2 + |\mathbf{h}|^2) \Psi_\delta^2\ d\mu\frac{dt}{t}
\end{equation*}
The first term above is the same as $\mathbf{S}_{13}'$ in \eqref{355} whilst the remaining term is the same as the last integral in \eqref{368}, hence $|\mathbf{II}_3| \lesssim \|u\|_{\infty}^2 \mu(Q)$.

To estimate $\mathbf{III}$, we observe by analogy with \eqref{367} that
\begin{align*}
\langle A\mathbf{p},\partial_t\mathbf{p}\rangle
&= \langle A_{||}\nabla_x(P_{\eta t}^*-I)\varphi - \mathbf{b},\nabla_x\partial_tP_{\eta t}^*\varphi\rangle \\
&= \langle A_{||}\nabla_x(P_{\eta t}^*\varphi-\varphi) + A_{||}\nabla_x \widetilde{\varphi} - \mu\widetilde{\mathbf{h}},\nabla_x\partial_tP_{\eta t}^*\varphi\rangle \\
&= \langle A_{||}\nabla_x[(P_{\eta t}^*\varphi-\varphi)-(P_{\eta t}\widetilde{\varphi}-\widetilde{\varphi})] +  A_{||}\nabla_xP_{\eta t}\widetilde{\varphi} - \mu\widetilde{\mathbf{h}},\nabla_x\partial_tP_{\eta t}^*\varphi\rangle
\end{align*}
for all $x \in 5Q$ and $t>0$ and then write 
\begin{align*}
\begin{split}
\mathbf{III} &= \iint_{\RR^{n+1}_+} (u_1^2/J)\langle \nabla_x[(P_{\eta t}^*\varphi-\varphi)-(P_{\eta t}\widetilde{\varphi}-\widetilde{\varphi})], A_{||}^*\nabla_x\partial_t P_{\eta t}^*\varphi\rangle \Psi_\delta^2\ dx dt \\
&\quad + \iint_{\RR^{n+1}_+} (u_1^2/J)\langle A_{||}\nabla_xP_{\eta t}\widetilde{\varphi} - \mu\widetilde{\mathbf{h}}, \nabla_x\partial_t P_{\eta t}^*\varphi\rangle \Psi_\delta^2\ dx dt
=: \mathbf{III}_1 + \mathbf{III}_2.
\end{split}
\end{align*}

To estimate $\mathbf{III}_1$, we integrate by parts in $x$ with respect to the measure $\mu$ to write
\begin{align*}
\mathbf{III}_1 &= \iint_{\RR^{n+1}_+} (u_1^2/J)[(P_{\eta t}^*\varphi-\varphi)-(P_{\eta t}\widetilde{\varphi}-\widetilde{\varphi})] (L_{||,\mu}^*\partial_t P_{\eta t}^* \varphi)\Psi_\delta^2\ d\mu dt \\
&\quad - \iint_{\RR^{n+1}_+} [(P_{\eta t}^*\varphi-\varphi)-(P_{\eta t}\widetilde{\varphi}-\widetilde{\varphi})]\langle \nabla_x(u_1^2\Psi_\delta^2/J), A_{||}^*\nabla_x\partial_t P_{\eta t}^*\varphi\rangle\ dx dt \\
&=: \mathbf{III}_1' + \mathbf{III}_1'',
\end{align*}
which is justified by the definition of $L_{||,\mu}^*$ (recall~\eqref{eq:Lbdrydef}, \eqref{eq:divbdrydef} and \eqref{eq:L||mudef}). 

To estimate $\mathbf{III}_1'$, we use Hardy's inequality (see, for instance, page 272 in \cite{St}) to observe, for the semigroups $\mathcal{P}_t \in \{e^{-t^2L_{||,\mu}^*},e^{-t^2L_{||,\mu}}\}$, the estimate
\begin{align*}
\int_0^\infty |\mathcal{P}_{\eta t}f - f|^2\ \frac{dt}{t^3}
\leq\int_0^\infty\!\left(\int_{0}^{\eta t}|\partial_s\mathcal{P}_sf|\ ds\right)^2 \frac{dt}{t^3}
\lesssim \int_0^\infty |\partial_t\mathcal{P}_tf|^2\ \frac{dt}{t}
\ \forall f \in L^2_\mu(\RR^n).
\end{align*} 
We then recall that $\|\mathbbm{1}_{\Omega_{\eta,Q}}u_1\|_\infty\lesssim\|u\|_\infty$ and $J\eqsim 1$ on $\supp \Psi_\delta \subseteq \Omega_{\eta/8,Q,\delta}$ to obtain
\begin{align*}
&|\mathbf{III}_1'|
\lesssim \|u\|_{\infty}^2 \int_{\RR^n} \int_0^\infty (|P_{\eta t}^*\varphi - \varphi| + |P_{\eta t}\widetilde{\varphi}-\widetilde{\varphi}|)\, |L_{||,\mu}^*\partial_t P_{\eta t}^* \varphi|\ dt d\mu \\
&\lesssim\|u\|_{\infty}^2\int_{\RR^n}\left(\int_0^{\infty} |P_{\eta t}^*\varphi - \varphi|^2 + |P_{\eta t}\widetilde{\varphi} - \widetilde{\varphi}|^2\ \frac{dt}{t^3}\right)^{1/2} \!
\left(\int_0^{\infty} |t^2 L_{||,\mu}^* \partial_t P_{\eta t}^* \varphi |^2\ \frac{dt}{t}\right)^{1/2}\! d\mu\\
&\lesssim\|u\|_{\infty}^2\left(\iint_{\RR^{n+1}_+} |\partial_tP_t^*\varphi|^2 + |\partial_tP_t\widetilde{\varphi}|^2\ d\mu\frac{dt}{t}\right)^{1/2} \!
\left(\iint_{\RR^{n+1}_+} |t^2 \partial_t L_{||,\mu}^* P_{\eta t}^* \varphi |^2\ d\mu\frac{dt}{t}\right)^{1/2}\\
&\lesssim \|u\|_{\infty}^2 (\|\nabla \varphi\|_{L^2_{\mu}(\RR^n,\RR^n)}^2 + \|\nabla \widetilde{\varphi}\|_{L^2_{\mu}(\RR^n,\RR^n)}^2)^{1/2} \|\nabla \varphi\|_{L^2_{\mu}(\RR^n,\RR^n)}
\lesssim \|u\|_{\infty}^2\mu(Q),
\end{align*}
where the final line uses the vertical square function estimates from~\eqref{eq:VSF1}-\eqref{eq:VSF2} for $\mathcal{L}_\mu\in \{L_{||,\mu}^*$, $L_{||,\mu}\}$ and the $W^{1,2}_{0,\mu}(5Q)$-Hodge estimates for $\varphi$, $\widetilde{\varphi}$ from \eqref{HodEst1}-\eqref{HodEst2}.

To estimate $\mathbf{III}_1''$, recall that $|P_{\eta t}^*\varphi - \varphi| \lesssim t$ and $|P_{\eta t}\widetilde{\varphi} - \widetilde{\varphi}| \lesssim t$ on $\supp \Psi_\delta \subseteq \Omega_{\eta/8,Q,\delta}$ by \eqref{eq:CoV2}, whilst $J\eqsim 1$ and $|\nabla_x (J^{-1})|  \lesssim |\nabla_x \partial_t P_{\eta t}^*\varphi|$, so distributing $\nabla_x$ over $u_1^2$, $\Psi_\delta^2$ and $1/J$ yields terms that can be controlled in the same way $\mathbf{J}$, $\mathbf{K}$ and $\mathbf{S}_{12}'$ in \eqref{355}.

To estimate $\mathbf{III}_2$, note that the estimates used to control $\varphi$ and $P_{\eta t}\varphi$ also hold for $\widetilde{\varphi}$ and $P_{\eta t}\widetilde{\varphi}$ by \eqref{HodEst1}-\eqref{HodEst2} and \eqref{eq:CoV1}-\eqref{eq:CoV2}, whilst $\div_\mu\mathbf{h} = \div_\mu\widetilde{\mathbf{h}}=0$  by \eqref{Hodgeforvarphi}, hence $\mathbf{III}_2$ can be estimated in the same way as $\mathbf{II}$ in \eqref{eq:II}.

This gives $|\mathbf{III}_1''|+ |\mathbf{III}_2| \lesssim \sigma \iint_{\RR^{n+1}_+} \langle A_1\nabla u_1, \nabla u_1\rangle \Psi_\delta^2\ t dx dt +\sigma^{-1} \|u\|^2_{\infty} \mu(Q)$, hence
\begin{equation}\label{eq:S2}
\mathbf{S}_2
\lesssim \sigma \iint_{\RR^{n+1}_+} \langle A_1\nabla u_1, \nabla u_1\rangle \Psi_\delta^2\ t dx dt 
+\sigma^{-1} \|u\|^2_{\infty} \mu(Q) \quad\forall\sigma\in(0,1).
\end{equation}

We combine \eqref{eq:E1E2}, \eqref{eq:S1} and \eqref{eq:S2} to obtain \eqref{eq:SEwithsigma}, as required.
\end{proof}

\section{Solvability of the Dirichlet Problem}\label{sec:Dp}
This section is dedicated to the proof of Theorem~\ref{thm:introAinfinity+Dp}. We first consider the construction and properties of a degenerate elliptic measure $\omega^X$ for degenerate elliptic equations $\div (A \nabla u) = 0$ in the upper half-space, where $X=(x,t)\in\RR^{n+1}_+$ and $n\geq2$. The $t$-independent coefficient matrix $A$ is assumed throughout to satisfy the degenerate bound and ellipticity in~\eqref{eq:deg.def} for some constants $0<\lambda\leq\Lambda<\infty$ and an $A_2$-weight $\mu$ on $\RR^n$. This is necessary as the literature only seems to treat bounded domains whilst the passage to unbounded domains in the uniformly elliptic case (see Section 10 in \cite{LSW} and \cite{HK}) relies on a global version of the Sobolev embedding in~\eqref{wSob}, which is not known for $A_2$-weights in general. The degenerate elliptic measure is then shown to be in the $A_{\infty}$-class with respect to $\mu$ on the boundary $\RR^n$ in Theorem~\ref{thm:Ainftymain} and the solvability of the Dirichlet problem follows in Theorem~\ref{thm:DpSolvmain}. These results together prove Theorem~\ref{thm:introAinfinity+Dp}.

\subsection{Boundary estimates for solutions}
We require some estimates for solutions near the boundary $\partial\Sigma$ of a bounded Lipschitz domain $\Sigma \subset \RR^n$ (see Section~2 of \cite{CFMS} for the standard definition). These estimates require some regularity on the domain boundary but no attempt is made here to obtain the minimal such regularity, as the focus is to define and analyse a degenerate elliptic measure on $\RR^n$.

The Lipschitz regularity of the boundary $\partial\Sigma$ ensures that the smooth class $C^\infty(\overline{\Sigma})$ and the Lipschitz class $C^{0,1}(\overline\Sigma)$ are both dense in $W^{1,2}_\mu(\Sigma)$ (see Theorem~3.4.1 in~\cite{Mo} and page 29 in \cite{KS}). This allows the usual definition, for $E\subseteq \partial\Sigma$ and $u \in W^{1,2}_\mu(\Sigma)$, whereby $u\geq 0$ on $E$ in the $W^{1,2}_\mu(\Sigma)$-sense means there exists a sequence $u_j$ in $C^{0,1}(\overline\Sigma)$ that converges to $u$ in $W^{1,2}_\mu(\Sigma)$ with $u_j(x)\geq 0$ for all $x\in E$. This induces definitions for inequalities $\leq$, $\geq$ and $=$, between functions and/or constants, on $E$ in the $W^{1,2}_\mu(\Sigma)$-sense (see, for instance, Definition 5.1 in \cite{KS}). Moreover, with $\sup_{\partial\Sigma} u := \inf\{k\in\RR : u\leq k \textrm{ on ${\partial\Sigma} $ in the $W^{1,2}_\mu(\Sigma)$-sense}\}$ and $\inf_{\partial\Sigma} :=-\sup_{\partial\Sigma} (-u)$, the weak maximum principle holds (see Theorem 2.2.2 in \cite{FKS}), and the strong version follows by the Harnack inequality in \eqref{eq:H} (see Corollary 2.3.10 in \cite{FKS}).

We can now state a H\"{o}lder continuity estimate and a Harnack inequality for certain solutions near the boundary. For a cube $Q\subset \RR^n$, recall the corkscrew point $X_Q:=(x_Q,\ell(Q))$ and denote the Carleson box in $\RR^{n+1}_+$ by $T_Q:=Q\times (0,\ell(Q))$. Also, recall that $\mu(x,t):=\mu(x)$, so $d\mu(x,t)=\mu(x)dxdt$, for $(x,t)\in \RR^{n+1}$. If $u\in W^{1,2}_{\mu}(T_{2Q})$ is a solution of $\div (A \nabla u) = 0$ in $T_{2Q}$, and $u=0$ on $2Q$ in the $W^{1,2}_\mu(T_{2Q})$-sense, then
\begin{equation}\label{eq:bdrydGN}
|u(x,t)| \lesssim \left(\frac{t}{\ell(Q)}\right)^{\alpha} \left(\fint_{T_{2Q}} |u|^2\ d\mu\right)^{1/2}
\qquad \forall (x,t) \in T_Q,
\end{equation}
and if, in addition, $u\geq0$ almost everywhere on $T_{2Q}$, then
\begin{equation}\label{eq:bdryH}
u(X) \lesssim u(X_Q)
\qquad \forall X \in T_Q,
\end{equation}
where $\alpha$ is from \eqref{eq:dGN} and the implicit constants depend only on $n$, $\lambda$, $\Lambda$ and $[\mu]_{A_2}$. Estimate \eqref{eq:bdrydGN} follows from standard reflection arguments and the interior H\"{o}lder continuity estimate in \eqref{eq:dGN}, as observed on page 102 in \cite{FKS}. Estimate \eqref{eq:bdryH} can then be deduced from \eqref{eq:bdrydGN} and the interior Harnack inequality in \eqref{eq:H}, as in the uniformly elliptic case (see the proof of Theorem 1.1 in \cite{CFMS}, which does not use the assumption therein that $A$ is symmetric).

\subsection{Definition and properties of degenerate elliptic measure}\label{subsection:defellipticmeasure}
For $X\in\RR^{n+1}$, $x\in\RR^n$ and $r>0$, we use $B(X,r):=\{Y\in\RR^{n+1} : |Y-X| <r\}$ to denote balls in $\RR^{n+1}$ and $\Delta(x,r):=\{y\in\RR^{n} : |x|<r\}$ to denote balls in $\RR^n$, where $\Delta(x,r)$ is identified with the surface ball $B((x,0),r)\cap \partial \RR^{n+1}_+$ in $\RR^{n+1}$. For each~$R>0$, consider the bounded Lipschitz domain $\Sigma_R := B(0,R)\cap \RR^{n+1}_{+}$ with Lipschitz constant at most~1. For each $X\in\Sigma_R$, the degenerate elliptic measure $\omega_R^X$ is the measure on $\partial\Sigma_R$, as defined on page 583 in \cite{FJK2}, such that $u(X) = \int_{\partial\Sigma_R} h\ d\omega_R^X$ solves the Dirichlet problem for continuous boundary data $h\in C(\partial\Sigma_R)$ in the sense that $\div(A\nabla u)=0$ in $\Sigma_R$ and $u\in C(\overline{\Sigma_R})$ with $u|_{\partial\Sigma_R}=h$.

We now define the degenerate elliptic measure on $\RR^n$. If $f\in C_c(\RR^n)$, fix $R_0>0$ such that $\supp f \subseteq \Delta(0,R_0)$ and set $f$ equal to zero on $\RR^{n+1}_+$, so then $f^\pm\in C(\partial\Sigma_R)$ for all $R\geq R_0$, where $f^\pm(X):=\max\{\pm f(X),0\}$, thus
\[
u_R^\pm(X): = \int_{\partial\Sigma_R}f^\pm\ d\omega^{X}_R
\qquad \forall X \in \Sigma_R
\]
solve the Dirichlet problem as above in $\Sigma_R$ for all $R\geq R_0$. The maximum principle then implies that $u_{R_1}^\pm(X)\leq u_{R_2}^\pm(X)$, whenever $R_0\leq R_1\leq R_2$ and $X\in\Sigma_{R_1}$, and that $\sup_{R>0}\|u_R^\pm\|_{\infty} \leq \|f\|_{\infty}$. This allows us to define
\begin{equation}\label{eq:wkconvdef}
u(X) := \lim_{R\rightarrow\infty}[u_R^+(X) - u_R^-(X)]
\qquad \forall X \in \RR^{n+1}_+
\end{equation}
and since the mapping $f\mapsto u(X)$ is a positive linear functional on $C_c(\RR^n)$, the Riesz Representation Theorem implies that there exists a regular Borel probability measure (the degenerate elliptic measure) $\omega^{X}$ on $\RR^n$ such that $u(X) = \int_{\RR^n} f\ d\omega^{X}$.

The function $u$ from \eqref{eq:wkconvdef} solves $\div(A\nabla u)=0$ in $\RR^{n+1}_+$. To prove this, note that $\|u\|_{\infty} \leq \|f\|_{\infty}$, so for each compact set $K\subset \RR^{n+1}_+$, the H\"{o}lder continuity of solutions in \eqref{eq:dGN} ensures the equicontinuity required to apply the Arzel\`{a}--Ascoli Theorem and extract a subsequence $u_{R_j}$ that converges to $u$ uniformly on $K$. This combined with Caccioppoli's inequality shows that $u_{R_j}$ converges to $u$ in $W^{1,2}_{\mu}(K)$, hence $u \in W^{1,2}_{\mu,\loc}(\RR^{n+1}_+)$. Moreover, if $\varphi \in C_c^{\infty}(\RR^{n+1}_+)$ and $K=\supp\varphi \subset \Sigma_R$, then
\begin{equation}\label{eq:solcheck}
\left|\int_{K}\langle A\nabla (u - u_R), \nabla \varphi \rangle\right| \leq \Lambda \|\nabla\varphi\|_\infty \mu(K)^{1/2} \|u - u_R\|_{W^{1,2}_{\mu}(K)},
\end{equation}
from which it follows that $\int_{\RR^{n+1}_+}\langle A\nabla u,\nabla\varphi\rangle = 0$, as required.

We note by \eqref{eq:wkconvdef} that, when restricted to any bounded Borel subset of $\RR^n$, the measures $\omega^X_R$ converge weakly to $\omega^X$, so Theorem 1 on page 54 of \cite{EG} shows that
\begin{equation}\label{eq:KU}
\omega^X(U) \leq \liminf_{R\rightarrow \infty} \omega^X_R(U),
\quad
\omega^X(K) \geq \limsup_{R\rightarrow \infty} \omega^X_R(K),
\quad
\omega^X(\mathcal{B}) = \lim_{R\rightarrow \infty} \omega^X_R(\mathcal{B})
\!\!
\end{equation}
for all bounded open sets $U\subset\RR^n$, all compact sets $K\subset\RR^n$, and all bounded Borel sets $\mathcal{B}\subset\RR^n$ such that $\omega^X(\partial\mathcal{B})=0$. This construction of the degenerate elliptic measure also provides for the following expected properties. 
\begin{lemma}\label{lem:abs.cty}
If $X_0,X_1 \in \RR^{n+1}_+$ and $E\subseteq\RR^n$ is a Borel set, then $\omega^{X_0}(E)=0$ if and only if $\omega^{X_1}(E)=0$. Moreover, the non-negative function $u(X):=\omega^X(E)$ is a solution of $\div(A\nabla u)=0$ in $\RR^{n+1}_+$ and the boundary H\"{o}lder continuity estimate 
\begin{equation}\label{eq:newbdryH}
|u(x,t)| \lesssim \left(\frac{t}{\ell(Q)}\right)^{\alpha} u(X_Q)
\qquad \forall (x,t) \in T_Q
\end{equation}
holds on all cubes $Q$ such that $2Q\subseteq \RR^n \setminus E$, where $\alpha$ is from \eqref{eq:dGN} and the implicit constants depend only on $n$, $\lambda$, $\Lambda$ and $[\mu]_{A_2}$, 
\end{lemma}
\begin{proof}
The proof follows that of Lemma 1.2.7 in \cite{K}, except we must account for the fact that the solution to the Dirichlet problem in $\RR^{n+1}_+$ defined by \eqref{eq:wkconvdef} requires boundary data to have compact support, which is easily done as we now show. Suppose that $\omega^{X_0}(E)=0$ and that $K\subseteq E$ is a compact set. The regularity of the measure implies that $\omega^{X_0}(K)=0$ and, for each $\epsilon>0$, there exists a  bounded open set $U\supset K$ such that ${\omega^{X_0}(U)<\epsilon}$. In particular, we may assume that $U$ is \textit{bounded} because $K$ is compact, so by Urysohn's Lemma there exists $g\in C_c(\RR^n)$ such that $g(x)=1$ on $K$, $0\leq g(x)\leq1$ on $U$, and $\supp g\subset U$. It follows that $u(X)=\int_{\RR^n} g\ d\omega^X$ is the solution to the Dirichlet problem in $\RR^{n+1}_+$ defined by \eqref{eq:wkconvdef} with boundary data $g$. Applying the Harnack inequality from \eqref{eq:H} and connecting $X_0$ with $X_1$ via a Harnack chain then shows that there exists $C>0$, depending on $X_0$ and $X_1$, such that
\[
\omega^{X_1}(K)
\leq u(X_1)
\leq Cu(X_0)
\leq C \omega^{X_0}(U)
\leq C \epsilon
\qquad \forall \epsilon>0,
\]
hence $\omega^{X_1}(K)=0$ for all compact sets $K\subseteq E$, and so $\omega^{X_1}(E)=0$ by regularity.

The proof that $u(X):=\omega^X(E)$ is a solution of $\div(A\nabla u)=0$ in $\RR^{n+1}_+$ also follows that of Lemma 1.2.7 in \cite{K}. It remains to prove that the boundary H\"older continuity estimate holds on all cubes $Q$ such that  $2Q\subseteq \RR^n \setminus E$. We first consider when $E$ is bounded. In that case, let $U_\delta$ denote the open $\delta$-neighbourhood of $E$ and set $\chi_{\epsilon,\delta}:= \vp_{\epsilon} * \mathbbm{1}_{U_\delta}$ for all $\delta>\epsilon>0$, where $\vp_\epsilon(x) := \epsilon^{-n}\vp(x/\epsilon)$ and $\vp \in C_c^\infty (\Delta(0,1))$ is a fixed non-negative function with $\int_{\rn}\vp = 1$. In particular, since $U_\delta$ is open, we have $\mathbbm{1}_E\leq \mathbbm{1}_{U_\delta} \leq \liminf_{\epsilon\to 0} \chi_{\epsilon,\delta}$. Consequently, if $X=(x,t)\in\reu$, then
\begin{equation}\label{eq1}
u(X) = \hm^X(E)\leq \hm^X(U_\delta)
\leq \int_{\RR^n}\liminf_{\epsilon\to 0}\chi_{\epsilon,\delta} \ d\hm^X
\leq \liminf_{\epsilon\to 0} \int_{\RR^n} \chi_{\epsilon,\delta} \ d\hm^X.
\end{equation}
The function $\chi_{\epsilon,\delta}$ belongs to $C_c^\infty(\rn)$ and thus extends to a function in $C_c^\infty(\RR^{n+1})$. The construction of the degenerate elliptic measure (see pages 580--583 in \cite{FJK2}, which was the starting point for our extension to the upper half-space above) thus implies that $\ve(X):= \int_{\RR^n} \chi_{\epsilon,\delta} \,d\hm^X$ is in $W^{1,2} (T_{\frac{3}{2}Q})$ and vanishes on $\frac{3}{2}Q$ whenever $0<\epsilon<\delta<\ell(Q)/4$, so estimate \eqref{eq1} combined with the boundary H\"older continuity estimate in \eqref{eq:bdrydGN} and the boundary Harnack inequality in \eqref{eq:bdryH} shows that
\begin{equation}\label{eq2} 
u(x,t) \leq \liminf_{\epsilon\to 0}\ve(x,t)
\lesssim \left(\frac{t}{\ell(Q)}\right)^\alpha \liminf_{\epsilon\to 0}\ve(X_Q)
\quad\forall (x,t) \in T_Q.
\end{equation}
We now let $U_{\delta,\epsilon}$ denote the open $\epsilon$-neighbourhood of $U_\delta$, in which case $\chi_{\epsilon,\delta} \leq \mathbbm{1}_{U_{\delta,\epsilon}}$ and $\ve(X) \leq \hm^X(U_{\delta,\epsilon})$, so by \eqref{eq2} and the regularity of the degenerate elliptic measure we have
\[
u(x,t) \lesssim \left(\frac{t}{\ell(Q)}\right)^\alpha \liminf_{\epsilon\to 0} \hm^{X_Q}(U_{\delta,\epsilon})
\lesssim \left(\frac{t}{\ell(Q)}\right)^\alpha \hm^{X_Q}(U_\delta)
\quad\forall (x,t) \in T_Q.
\]
This proves \eqref{eq:newbdryH} if $E$ is bounded, since the regularity of the measure also implies that $\hm^{X_Q}(U_\delta)$ approaches $\hm^{X_Q}(E)=u(X_Q)$ as $\delta$ approaches 0. If $E$ is not bounded, then applying \eqref{eq:newbdryH} on the bounded sets $E_k:= \mathbbm{1}_{2^{k+1}Q\setminus 2^kQ} E$, for $k\in\mathbb{N}$, shows that
\[
u(x,t) = \sum_{k=1}^\infty \omega^X(E_k)
\lesssim \sum_{k=1}^\infty \left(\frac{t}{\ell(Q)}\right)^\alpha  \omega^{X_Q}(E_k)
= \left(\frac{t}{\ell(Q)}\right)^\alpha \omega^{X_Q}(E)
\quad\forall (x,t) \in T_Q,
\]
as required.
\end{proof}

\subsection{Preliminary estimates for degenerate elliptic measure}\label{Sec:PreEst}
In the uniformly elliptic case, there is a rich theory for the Green's function on bounded domains, and specifically, estimates and connections with elliptic measure (see, for instance, Theorem 1.2.8 and Corollary 1.3.6 in \cite{K}). This theory also extends to unbounded domains (see Section 10 in \cite{LSW} and \cite{HK}). In the degenerate elliptic case, the theory was developed on bounded domains in \cite{FJK1}, \cite{FJK2} and \cite{FKS}, but it is not clear if there is always such a Green's function on unbounded domains. In particular, the construction in \cite{HK} for the uniformly elliptic case relies on the (unweighted) global version of the Sobolev embedding in~\eqref{wSob}, which is not known for a general $A_2$-weight. In what follows, we combine the properties of the Green's function on the bounded domain $\Sigma_R := B(0,R)\cap \RR^{n+1}_{+}$ with the limit properties in \eqref{eq:KU} to deduce estimates for degenerate elliptic measure on $\RR^n$. These will be used to prove Lemma~\ref{28} and ultimately Theorem~\ref{thm:Ainftymain}.

For each $R>0$, the Green's function $g_R: \overline{\Sigma}_R \times \overline{\Sigma}_R \mapsto [0,\infty]$ is constructed by following Proposition 2.4 in \cite{FJK1}. In particular, for each $Y\in\Sigma_R$, the mapping $X\mapsto g_R(X,Y)$ is the H\"{o}lder continuous function in $\overline{\Sigma}_R\setminus\{Y\}$ that vanishes on $\partial\Sigma_R$ and satisfies $\int_{\Sigma_R} \langle A\nabla g_R(\cdot, Y), \nabla \Phi \rangle = \Phi(Y)$ for all $\Phi\in C^\infty_c(\Sigma_R)$. As explained on page 583 in \cite{FJK2}, these properties are valid on any NTA domain, hence \textit{a fortiori} on $\Sigma_R$. The proofs do not rely on the assumption therein that $A$ is symmetric, although the symmetry property $``g_R(X,Y)=g_R(Y,X)"$ is no longer guaranteed, as $g^*_R(X,Y):=g_R(Y,X)$ is the Green's function for the adjoint operator $-\div(A^*\nabla)$. We will rely on the following two lemmas, which are immediate from Theorem~4 and Lemma 3 in \cite{FJK2}, respectively, to estimate the Green's function $g_R$ and the degenerate elliptic measure $\omega_R$ on $\Sigma_R$.
\begin{lemma}\label{Greensize}
If $X,Y\in\Sigma_R$ and $|X-Y|<\dist(Y,\partial\Sigma_R)/2$, then
\[
g_R(X,Y) \eqsim \int_{|X-Y|}^{\dist(Y,\partial\Sigma_R)} \frac{s^2}{\mu(B(Y,s))}\frac{ds}{s},
\]
where the implicit constants depend only on  $n$, $\lambda$, $\Lambda$ and $[\mu]_{A_2}$.
\end{lemma} 

\begin{lemma}\label{grhar}
If $R>0$ and $Q$ is a cube in $\RR^n$ such that $T_{2Q} \subset \Sigma_R$, then
\begin{equation*}
\frac{g_R(X_Q, Y)}{\ell(Q)} \eqsim \omega_R^{Y}(Q) \frac{\ell(Q)}{\mu(T_Q)} = \frac{\omega_R^{Y}(Q)}{\mu(Q)} 
\qquad\forall Y\in \Sigma_R\setminus T_{2Q},
\end{equation*}
where the implicit constants depend only on  $n$, $\lambda$, $\Lambda$ and $[\mu]_{A_2}$.
\end{lemma}

The degenerate elliptic measure $\omega_R^{X}$ satisfies the doubling property $\omega_R^{X}(2Q) \leq C_0 \omega_R^{X}(Q)$ for all cubes $Q$ in $\RR^n$ such that $T_{2Q} \subset \Sigma_R$ and all $X\in \Sigma_R \setminus T_{2Q}$, where the doubling constant $C_0>0$ depends only on $n$, $\lambda$, $\Lambda$ and $[\mu]_{A_2}$. This is proved in Lemma 1 on page 584 of \cite{FJK2} by using the estimates in Lemma \ref{grhar}, the Harnack inequality in \eqref{eq:H}, and the doubling property of $\mu$. The doubling constant $C_0$ does not depend on $R$, which allows us to use the inequalities in \eqref{eq:KU} to show that the degenerate elliptic measure $\omega^X$ is locally doubling on $\RR^n$, in the sense that
\begin{equation}\label{localdoubling}
\omega^{X}(2Q) 
\leq\liminf_{R \rightarrow \infty} \omega^{X}_R(2Q) 
\lesssim \liminf_{R \rightarrow \infty} \omega^{X}_R(\tfrac{1}{2}Q) 
\leq\limsup_{R \rightarrow \infty} \omega^{X}_R(\tfrac{1}{2}\overline{Q}) 
\leq \omega^{X}(Q)
\end{equation}
for all cubes $Q\subset\RR^n$ and all $X\in\RR^{n+1}_+\setminus T_{2Q}$, where the implicit constant is $C_0^2$. In particular, the doubling property implies that $\omega^X(\partial Q)=0$ for all cubes $Q\subset\RR^n$ (see page 403 in \cite{GR} or Proposition~6.3 in \cite{HM}), so \eqref{localdoubling} actually improves to $\omega^{X}(2Q) \leq C_0 \omega^X(Q)$, since by the equality in \eqref{eq:KU} we now have
\begin{equation}\label{eq:Qinfsup}
\omega^X(Q) = \lim_{R\rightarrow \infty} \omega^X_R(Q)
\end{equation}
for all cubes $Q\subset\RR^n$ and all $X\in\RR^{n+1}_+\setminus T_{2Q}$. This provides the following estimate for degenerate elliptic measure. 

\begin{lemma}\label{Bo:est}
If $Q$ is a cube in $\RR^n$, then $\omega^{X_Q}(Q) \gtrsim 1$, where the implicit constant depends only on  $n$, $\lambda$, $\Lambda$ and $[\mu]_{A_2}$.
\end{lemma}

\begin{proof}
Let $Q$ denote a cube in $\RR^n$ and fix $R_0>0$ such that $T_{2Q} \subset \Sigma_{R_0}$. The H\"{o}lder continuity at the boundary in \eqref{eq:bdrydGN} and the Harnack inequality in \eqref{eq:H} imply (see the proof of Lemma 3 on page 585 in \cite{FJK2}) that
\begin{equation*}
\omega_R^{X_Q}(Q) \gtrsim 1 \qquad \forall R\geq R_0,
\end{equation*}
where the implicit constant depends only on $n$, $\lambda$, $\Lambda$ and $[\mu]_{A_2}$, and so does not depend on $R$. The result follows by using Harnack's inequality to shift the pole (from $X_{2Q}$ to $X_Q$) in \eqref{localdoubling}-\eqref{eq:Qinfsup} to obtain $\omega^{X_Q}(Q) = \lim_{R\rightarrow \infty} \omega^{X_Q}_R(Q)\gtrsim 1$.
\end{proof}

The estimates in Lemma \ref{grhar} also imply the following Comparison Principle. The result is stated on page 585 in \cite{FJK2} and the proof is the same as in the uniformly elliptic case (see Theorem 1.4 in \cite{CFMS} or Lemma 1.3.7 in \cite{K}, neither of which use the assumption therein that $A$ is symmetric).

\begin{lemma}(Comparison Principle) \label{Lem:Comparison}
Let $Q$ denote a cube in $\RR^n$ and suppose that $u,v \in W^{1,2}_\mu(T_{2Q}) \cap C(\overline{T_{2Q}})$ with ${u,v\geq0}$ on $T_{2Q}$ . If $\div (A \nabla u) = \div (A \nabla v) = 0$ in $T_{2Q}$ and $u=v=0$ on $2Q$, then
\[
\frac{u(X)}{v(X)} \eqsim \frac{u(X_Q)}{v(X_Q)}\qquad \forall X \in T_Q,
\]
where the implicit constants depend only on $n$, $\lambda$, $\Lambda$ and $[\mu]_{A_2}$. 
\end{lemma}

The following corollary of these preliminaries will be used in Lemma~\ref{Ker:sup_bound} to estimate Radon--Nikodym derivatives of the degenerate elliptic measure.

\begin{lemma}\label{138}
If $Q_0$ and $Q$ are cubes in $\RR^n$ such that $Q\subseteq Q_0$, then
\[
\omega^{X_{Q_0}}(Q) \eqsim \frac{\omega^{X}(Q)}{\omega^{X}(Q_0)}
\qquad \forall X\in \RR^{n+1}_+\setminus T_{2Q_0},
\]
where the implicit constants depend only on $n$, $\lambda$, $\Lambda$ and $[\mu]_{A_2}$.
\end{lemma}
\begin{proof}
Let $Q\subset Q_0$ be cubes in $\RR^n$, suppose that $X\in \RR^{n+1}_+\setminus T_{2Q_0}$ and consider $R>0$ large enough so that $X\in\Sigma_R$ and $T_{4Q_0}\subset \Sigma_R$. Lemma \ref{grhar} shows that
\begin{align*}
\omega_R^{X}(Q_0)\, \ell(Q_0) &\eqsim \mu(Q_0)\, g_R(X_{Q_0},X) ,\\
\omega_R^{X}(Q)\, \ell(Q) &\eqsim \mu(Q)\, g_R(X_Q,X) \\
\omega_R^{X_{3Q_0}}(Q)\, \ell(Q)  & \eqsim \mu(Q)\, g_R(X_Q, X_{3Q_0}).
\end{align*}
If $u(Y) = g_R(Y,X)$ and $v(Y) = g_R(Y,X_{3Q_0})$, then $\div(A\nabla u) = \div(A\nabla u) v = 0$ in $T_{2Q_0}$ and $u=v=0$ on $2Q_0$, so the Comparison Principle in Lemma~\ref{Lem:Comparison} shows that
\begin{equation*}
\frac{g_R(X_Q,X)}{g_R(X_Q, X_{3Q_0})} = \frac{u(X_Q)}{v(X_Q)} \eqsim \frac{u(X_{Q_0})}{v(X_{Q_0})} = \frac{g_R(X_{Q_0},X)}{g_R(X_{Q_0},X_{3Q_0})}.
\end{equation*}
Also, Lemma~\ref{Greensize} shows that $g_R(X_{Q_0}, X_{3Q_0}) \eqsim \ell(Q_0)/ \mu(Q_0)$, so together we obtain
\begin{align*}
\frac{\omega_R^{X}(Q)}{\omega_R^{X}(Q_0)} 
\eqsim \frac{g_R(X_Q,X)}{g_R(X_{Q_0},X)}
\frac{\mu(Q)}{\ell(Q)}\frac{\ell(Q_0)}{\mu(Q_0)}
\eqsim \frac{g_R(X_Q, X_{3Q_0})}{g_R(X_{Q_0},X_{3Q_0})} \frac{\mu(Q)}{\ell(Q)}\frac{\ell(Q_0)}{\mu(Q_0)}
\eqsim \omega_R^{X_{3Q_0}}(Q).
\end{align*}
The Harnack inequality from \eqref{eq:H} then shows that $\omega_R^{X}(Q) \eqsim \omega_R^{X}(Q_0) \omega_R^{X_{Q_0}}(Q)$ and the result follows by using \eqref{eq:Qinfsup} to estimate the limit as $R$ approaches infinity.
\end{proof}

If $X,X_0\in\RR^{n+1}_+$, then Lemma~\ref{lem:abs.cty} shows that $\omega^{X}$ and $\omega^{X_0}$ are mutually absolutely continuous, so the Lebesgue differentiation theorem for the locally doubling measure $\omega^{X_0}$ implies that the Radon--Nikodym derivative of $\omega^{X}$ satisfies
\begin{equation}\label{eq:Klim}
K(X_0,X,y) := \frac{d\omega^{X}}{d\omega^{X_0}}(y) = \lim_{s \rightarrow 0}\frac{\omega^X(Q(y,s))}{\omega^{X_0}(Q(y,s))}
\qquad \omega^{X_0}\text{-a.e. }y\in \RR^n,
\end{equation}
where $Q(y,s)$ denotes the cube in $\RR^n$ with centre $y$ and side length $s$. The following decay estimate for the kernel function $K$ extends Lemma 2 on page 584 in \cite{FJK2}. It is the final property of degenerate elliptic measure needed to prove Lemma~\ref{28}.

\begin{proposition}\label{Ker:sup_bound}
If $Q_0$ and $Q$ are cubes in $\RR^n$ such that $Q\subseteq Q_0$, then
\begin{equation*}
K(X_{Q_0},X_Q,y) \lesssim \frac{1}{\omega^{X_{Q_0}}(Q)}\max\left\lbrace\frac{|y - x_Q|}{\ell(Q)}, 1\right\rbrace^{-\alpha}
\qquad \omega^{X_{Q_0}}\text{-a.e. }y\in Q_0,
\end{equation*}
where $\alpha\!>\!0$ from \eqref{eq:dGN} and the implicit constant depend only on  $n$, $\lambda$, $\Lambda$ and $[\mu]_{A_2}$.
\end{proposition}
\begin{proof}
Let $Q\subseteq Q_0$ denote cubes in $\RR^n$ and fix $J\in\N$ such that $2^{J-1}Q\subseteq Q_0 \subseteq 2^{J}Q$. If $y\in Q$, then Lemma~\ref{138} and the Harnack inequality in \eqref{eq:H} show that
\[
\omega^{X_Q}(Q(y,s)) \eqsim \frac{\omega^{X_{2Q_0}}(Q(y,s))}{\omega^{X_{2Q_0}}(Q)} \eqsim \frac{\omega^{X_{Q_0}}(Q(y,s))}{\omega^{X_{Q_0}}(Q)}
\]
whenever $0<s<\dist(y,\RR^n\setminus Q)$. If $y\in 2^j Q \setminus 2^{j-1} Q$ for some $j\in\{1,\ldots,J\}$, then the boundary H\"{o}lder continuity estimate in \eqref{eq:newbdryH} combined with Lemma~\ref{138} and the Harnack inequality in \eqref{eq:H} show that
\[
\omega^{X_Q}(Q(y,s)) 
\lesssim \left(\frac{\ell(Q)}{2^{j-2}\ell(Q)}\right)^{\alpha}\omega^{X_{2^{j-2}Q}}(Q(y,s))
\eqsim \left(\frac{\ell(Q)}{|y-x_Q|}\right)^{\alpha} \frac{\omega^{X_{Q_0}}(Q(y,s))}{\omega^{X_{Q_0}}(2^j{Q})}
\]
whenever $0<s<\dist(y,\RR^n\setminus (2^jQ\setminus2^{j-2}Q))$, where $\alpha>0$ from \eqref{eq:dGN} and the implicit constants depend only on  $n$, $\lambda$, $\Lambda$ and $[\mu]_{A_2}$. The result follows by using these two  estimates to bound the limit as $s$ approaches zero in \eqref{eq:Klim}.
\end{proof}

\subsection{The $A_\infty$-estimate for degenerate elliptic measure}\label{sect:dyadiccubes}
We now combine the properties of degenerate elliptic measure with good $\epsilon_0$-coverings for sets, as introduced in \cite{KKoPT} and defined below (see also \cite{KKiPT}), to construct bounded solutions that satisfy the truncated square function estimate in Lemma~\ref{28}. This result, combined with the Carleson measure estimate from Theorem~\ref{thm:introCME}, allows us to prove the $A_\infty$-estimate for the degenerate elliptic measure in Theorem~\ref{thm:Ainftymain}. This avoids the need to apply the method of $\epsilon$-approximability, as was done in \cite{HKMP1}, and so simplifies the proof in the uniformly elliptic case. 

Let $\DD(\RR^n)$ denote the standard collection $\{2^k(j+[0,1]^n):k\in\ZZ,j\in\ZZ^n\}$ of all  \textit{closed} dyadic cubes $S$ in $\RR^n$. For each $S\in\DD(\RR^n)$ and $\eta=2^{-K}$, where $K\in\mathbb{N}$, define $\DD(S):=\{S'\in\DD(\RR^n) : S'\subseteq S\}$ and 
\begin{equation}\label{eq:DetaDef}
\DD^\eta(S):=\{S'\in \DD(S) : \ell(S') = 2^{-K} \ell(S)\},
\end{equation}
so $\DD^\eta(S)$ is precisely the set of all \textit{dyadic descendants} of $S$ at scale $2^{-K}\ell(S)$.

\begin{definition}\label{def:epscover} Suppose that $Q_0$ is a cube in $\RR^n$. If $\epsilon_0>0$, $k\in\N$, $Q\subseteq Q_0$ is a cube and $E \subseteq Q$, then a \textit{good $\epsilon_0$-cover of $E$ of length $k$ in $Q$} is a collection $\{O_l\}_{l=1}^{k}$ of nested open sets that satisfy $E\subseteq O_k \subseteq O_{k-1} \subseteq \ldots \subseteq O_1\subseteq Q$ and each of which has a decomposition $O_l=\cup_{i=1}^\infty S_i^l$ given by a collection $\{S_i^{l}\}_{i\in\mathbb{N}}\subseteq\mathbb{D}(\RR^n)$ of dyadic cubes with pairwise disjoint interiors such that
\begin{equation}\label{cover1}
\omega^{X_{2Q_0}}(O_l\cap S_i^{l-1}) \leq \epsilon_0\, \omega^{X_{2Q_0}}(S_i^{l-1})
\qquad \forall i\in\N, \ \forall l\in\{2,\ldots,k\}.
\end{equation}
\end{definition}

Let us record a few important consequences of this definition that will be needed. It is proved on page 243 in \cite{KKoPT} that for each $i\in\N$ and $l\in\{2,\ldots,k\}$, there exists a unique $j\in\N$ such that $S_i^l$ is a \textit{proper} subset of $S_j^{l-1}$, thus $\ell(S_i^l)\leq \tfrac{1}{2}\ell(S_j^{l-1})$. Also, for $m\in\{2,\ldots,k\}$, iterating \eqref{cover1} as in Lemma~2.5 of \cite{KKoPT} shows that
\begin{equation}\label{cover2}
\omega^{X_{2Q_0}}(O_l\cap S_{i}^{m}) \leq \epsilon_0^{l-m}\omega^{X_{2Q_0}}(S_i^{m})
\qquad \forall i\in\N,\ \forall l\in\{m,\ldots,k\}.
\end{equation}

In the uniformly elliptic case, the following result is Lemma 2.3 from \cite{KKiPT}. The proof extends to the degenerate elliptic case, since it only relies on the fact that the degenerate elliptic measure $\omega^{X_{2Q_0}}$ is doubling when restricted to the cube $Q_0$.

\begin{lemma}\label{lemma25} Suppose that $Q_0$ is a cube in $\RR^n$. If $\epsilon_0 >0$, then there exists $\delta_0>0$, depending only on $\epsilon_0$, $n$, $\lambda$, $\Lambda$ and $[\mu]_{A_2}$, such that the following property holds:\\ \indent  If $Q\subseteq Q_0$ is a cube and $E\subseteq Q_0$ such that $\omega^{X_{2Q_0}}(E) \leq \delta_0$, then there exists a good $\epsilon_0$-cover of $E$ of length $k$ in $Q$ for some natural number $k \eqsim {\log(\omega^{X_{2Q_0}}(E))}/{\log\epsilon_0}$, where the implicit constants depend only on $n$, $\lambda$, $\Lambda$ and $[\mu]_{A_2}$.
\end{lemma}

We can now prove the following lemma by adapting the proof in \cite{KKiPT} to the degenerate elliptic case. The original argument has also been somewhat modified.
\begin{lemma}\label{28}
Suppose that $Q_0$ is a cube in $\RR^n$. If $M \geq 1$, then there exists $\delta_M>0$, depending only on $M$,  $n$, $\lambda$, $\Lambda$ and $[\mu]_{A_2}$, such that the following property holds:\\ \indent If $Q\subseteq Q_0$ is a cube and $E \subseteq Q$ and $\omega^{X_{2Q_0}}(E) \leq \delta_M$, then there is a Borel subset $\mathcal{B}$ of $\RR^n$ such that the solution $u(X) :=\omega^{X}(\mathcal{B})$ of $\div (A \nabla u) = 0$ in $\RR^{n+1}_+$ satisfies
\begin{equation*}
M \leq \int_0^{\gamma \ell(Q)}\int_{\Delta(x,\gamma t)}|t\nabla u(y,t)|^2\frac{d\mu(y)}{\mu(\Delta(x,t))}\frac{dt}{t}
\qquad \forall x\in E,
\end{equation*}
where $\gamma>0$ is a constant that depends only on $n$, $\lambda$, $\Lambda$ and $[\mu]_{A_2}$.
\end{lemma}
\begin{proof}
We introduce three constants $\epsilon_0,\delta,\eta\in(0,1)$ that will be chosen with $\delta\leq\delta_0$, where $\delta_0$ is determined by $\epsilon_0$ as in Lemma \ref{lemma25}, and $\eta=2^{-K}$ for some $K\in\N$. Therefore, if $E \subseteq Q \subseteq Q_0$ and $\omega^{X_{2Q_0}}(E) \leq \delta$, then there exists a good $\epsilon_0$-cover of $E$ of length $k$ in $Q$ such that $k \eqsim \log(\omega^{X_{2Q_0}}(E))/\log \epsilon_0$. This cover is denoted by $\{O_l\}_{l=1}^{k}$ with $O_l=\cup_{i=1}^\infty S_i^{l}$ as in Definition~\ref{def:epscover}, and for each such cube ~$S_i^l$, a dyadic descendant $\widetilde{S}_i^l$ in $\DD^\eta(S_i^l)$ that contains the centre of $S_i^l$ is now fixed and
\begin{equation}\label{eqOtildeDef}
\widetilde{O}_l := \cup_{i=1}^\infty \widetilde{S}_{i}^{l},
\end{equation}
where we note that $\ell(\widetilde{S}_i^l)=\eta\ell(S_i^l)$ in accordance with \eqref{eq:DetaDef}.

We claim that there exists a Borel subset $\mathcal{B}$ of $\RR^n$ such that $\mathbbm{1}_\mathcal{B} = \sum_{m=2}^k \mathbbm{1}_{\widetilde{O}_{j-1}\setminus O_j}$. To see this, suppose that $\sum_{m=2}^k \mathbbm{1}_{\widetilde{O}_{j-1}\setminus O_j}(x) \neq 0$ and let $l_0$ denote the smallest integer $l\in[2,k]$ such that $\mathbbm{1}_{\widetilde{O}_{l-1}\setminus O_l}(x) =1$. It must hold that $x \in \widetilde{O}_{l_0-1}\setminus O_{l_0}$, so then $x \notin O_{l_0}$, which implies that $x \notin O_l$ and $x \notin \widetilde{O}_l$ for all $l \geq l_0$, hence $ \mathbbm{1}_{\widetilde{O}_{l-1}\setminus O_l}(x) =0$ for all $l > l_0$ and the claim follows.

We now aim to choose $\epsilon_0,\eta\in(0,1)$ such that $u(X) :=\omega^X(\mathcal{B})$ on $\RR^{n+1}_+$ satisfies 
\begin{equation}\label{214}
|u(X_{\eta S^l_i}) - u(X_{\eta \widehat{S}^l_i})| \gtrsim 1
\qquad \forall \widehat{S}^l_i\in\DD^\eta(S_i^l),\ \forall i\in\N,\ \forall l \in \{1,\ldots,k\},
\end{equation}
where the implicit constant depends only on the allowed constants $n, \lambda, \Lambda$ and $[\mu]_{A_2}$, and if $x^l_i$ and $\hat{x}^l_i$ denote the centres of $S^l_i$ and $\widehat{S}^l_i$, then the relevant corkscrew points are precisely $X_{\eta S^l_i}=(x^l_i,\eta\ell(S^l_i))$ and $X_{\eta \widehat{S}^l_i}=(\hat{x}^l_i,\eta^2\ell(S^l_i))$. To this end, we proceed to obtain estimates for $u(X_{\eta S^l_i})$ and $u(X_{\eta \widehat{S}^l_i})$.

To estimate $u(X_{\eta S^l_i})$, write
\[
u(X_{\eta S^l_i})
= \int_{\RR^n \setminus S^l_i} \mathbbm{1}_\mathcal{B}\ d\omega^{X_{\eta S_i^l}} + \int_{S^l_i}\mathbbm{1}_\mathcal{B}\ d\omega^{X_{\eta S_i^l}}=: I + II.
\]
The boundary H\"{o}lder continuity in \eqref{eq:newbdryH} shows that $I \leq \omega^{X_{\eta S_i^l}}(\RR^n \setminus S^l_i) \leq C_0 \eta^\alpha$, where $C_0, \alpha>0$ depend only on the allowed constants. To estimate $II$, write
\begin{align*}
II &= \sum_{j=2}^{l}\int_{S^l_i}\mathbbm{1}_{\widetilde{O}_{j-1} \setminus O_j}\ d\omega^{X_{\eta S^l_i}}
+ \sum_{j=l + 2}^{k}\int_{S^l_i}\mathbbm{1}_{\widetilde{O}_{j-1}\setminus O_j}\ d\omega^{X_{\eta S^l_i}}
+ \int_{S^l_i}\mathbbm{1}_{\widetilde{O}_l\setminus O_{l+1}}\ d\omega^{X_{\eta S^l_i}}\\
&=: II_1 + II_2 + II_3.
\end{align*}
First, observe that $II_1 = 0$, since if $m \in\{2,\ldots,l\}$, then $S^l_i \subseteq O_l \subseteq O_j$ and so $(\widetilde{O}_{j-1}\setminus O_j) \cap S^l_i = \emptyset$. To estimate $II_2$, the kernel function representation in \eqref{eq:Klim} and estimates in Proposition~\ref{Ker:sup_bound}, the local doubling property of the degenerate elliptic measure in \eqref{localdoubling} and property \eqref{cover2} of the good $\epsilon_0$-covering, show that
\begin{align*}
II_2 &= \sum_{j=l + 2}^{k} \int_{(\widetilde{O}_{j-1}\setminus O_j) \cap S^l_i} K(X_{2Q_0},X_{\eta S^l_i},y)\ d\omega^{X_{2Q_0}}(y) \\
&\leq \frac{C_\eta}{\omega^{X_{2Q_0}}(S^l_i)}\sum_{j=l + 2}^{k} \omega^{X_{2Q_0}}\left((\widetilde{O}_{j-1}\setminus O_j) \cap S^l_i \right)\\
&\leq \frac{C_\eta}{\omega^{X_{2Q_0}}(S^l_i)}\sum_{j=l + 2}^{k} \omega^{X_{2Q_0}}(O_{j-1}\cap S^l_i)\\
&\leq \frac{C_\eta}{\omega^{X_{2Q_0}}(S^l_i)}\sum_{j=l + 2}^{k} \epsilon_0^{j-1-l}\omega^{X_{2Q_0}}(S_i^l)
\leq C_\eta\epsilon_0/(1-\epsilon_0),
\end{align*}
where the constant $C_\eta>0$ depends only on $\eta$ and the allowed constants.

To estimate $II_3$, observe that $S_i^l \cap \widetilde{O}_l = \widetilde{S}^l_i$ by the definition of $\widetilde{O}_l$ in \eqref{eqOtildeDef}, hence
\begin{equation*}
II_3 = \int_{\widetilde{S}^l_i} d\omega^{X_{\eta S^l_i}} - \int_{\widetilde{S}^l_i \cap O_{l+1}}d\omega^{X_{\eta S^l_i}}=: II_3' - II_3''.
\end{equation*}
The term $II_3''$ is estimated in the same way as $II_2$ above to show that
\begin{equation*}
II_3'' \leq \frac{C_\eta}{\omega^{X_{2Q_0}}(S^l_i)}\omega^{X_{2Q_0}}(O_{l+1} \cap \widetilde{S}^l_i) \leq \frac{C_\eta}{\omega^{X_{2Q_0}}(S^l_i)}\omega^{X_{2Q_0}}( O_{l+1} \cap S^l_i)\leq C_\eta \epsilon_0.
\end{equation*}
We estimate $II_3'$ from above and below. First, note that $X_{\eta S^l_i}=(x^l_i,\eta\ell(S^l_i))$, $x_i^l\in \widetilde{S}_i^l$ and $\ell(\widetilde{S}_i^l)=\eta\ell(S_i^l)$, so $\omega^{X_{\eta S^l_i}}(\widetilde{S}_i^l) \eqsim \omega^{X_{\widetilde{S}^l_i}}(\widetilde{S}_i^l)$ by the Harnack inequality in \eqref{eq:H}, whilst $\omega^{X_{\widetilde{S}^l_i}}(\widetilde{S}_i^l)\gtrsim1$ by Lemma~\ref{Bo:est}. Thus, there exists $c_0\in(0,1)$ depending only on the allowed constants such that $II_3' = \omega^{X_{\eta S^l_i}}(\widetilde{S}_i^l) \geq c_0$. Next, choose a different dyadic descendant $\utilde{S}_i^l \neq \widetilde{S}_i^l $ in $\DD^\eta(S_i^l)$ that contains the centre of $S_i^l$. The preceding argument shows that $\omega^{X_{\eta S_i^l}}(\utilde{S}_i^l) \geq c_0$, whilst $\omega^{X_{\eta S^l_i}}(\utilde{S}_i^l\cap \widetilde{S}_i^l)\leq\omega^{X_{\eta S^l_i}}(\partial \widetilde{S}_i^l)=0$, hence
\begin{equation*}
c_0 \leq II_3' = \omega^{X_{\eta S^l_i}}(\widetilde{S}_i^l) = 1 - \omega^{X_{\eta S^l_i}}(\RR^n\setminus \widetilde{S}_i^l) \leq 1 - \omega^{X_{\eta S^l_i}}(\utilde{S}_i^l) \leq 1 - c_0.
\end{equation*}
The above estimates together show that if $\epsilon_0\in(0,1/2)$, then 
\begin{equation}\label{eq:est1}
c_0 \leq u(X_{\eta S_i^l}) \leq C_0\eta^{\alpha} + 3C_{\eta}\epsilon_0 + 1 - c_0.
\end{equation}

To estimate $u(X_{\eta\widehat{S}^l_i})$, write
\begin{equation*}
 u(X_{\eta\widehat{S}^l_i}) = \int_{\RR^n \setminus \widehat{S}^l_i} \mathbbm{1}_\mathcal{B}\ d\omega^{X_{\eta\widehat{S}^l_i}} + \int_{\widehat{S}^l_i} \mathbbm{1}_\mathcal{B}\ d\omega^{X_{\eta\widehat{S}^l_i}}=: \widehat{I} + \widehat{II}
\end{equation*}
as well as
\begin{align*}
\widehat{II} &= \sum_{j=2}^{l}\int_{\widehat{S}^l_i}\mathbbm{1}_{\widetilde{O}_{j-1}\setminus O_j}\ d\omega^{X_{\eta \widehat{S}^l_i}}
+ \sum_{j=l + 2}^{k}\int_{\widehat{S}^l_i}\mathbbm{1}_{\widetilde{O}_{j-1}\setminus O_j}\ d\omega^{X_{\eta\widehat{S}^l_i}}
+ \int_{\widehat{S}^l_i}\mathbbm{1}_{\widetilde{O}_l\setminus O_{l+1}}\ d\omega^{X_{\eta\widehat{S}^l_i}}\\
&=: \widehat{II}_1 + \widehat{II}_2 + \widehat{II}_3.
\end{align*}
The arguments used to estimate $I$, $II_1$ and $II_2$ show that $\widehat{I} \leq \omega^{X_{\eta \widehat{S}_i^l}}(\RR^n \setminus \widehat{S}^l_i) \leq C_0 \eta^{\alpha}$, $\widehat{II}_1=0$ and $\widehat{II}_2 \leq C_{\eta}\epsilon_0/(1-\epsilon_0)$. To estimate $\widehat{II}_3$, observe that
\[
\widehat{S}^l_i \cap (\widetilde{O}_l \setminus O_{l+1}) = (\widehat{S}^l_i \cap\widetilde{S}^l_i)\setminus O_{l+1},
\]
where either $\omega^{X_{\eta\widehat{S}^l_i}}(\widehat{S}^l_i \cap\widetilde{S}^l_i) = 0$ and $\widehat{II}_3=0$, or $\widehat{S}^l_i =\widetilde{S}^l_i$ and
\begin{equation*}
\widehat{II}_3 = \int_{\widehat{S}^l_i}\ d\omega^{X_{\eta\widehat{S}^l_i}} - \int_{\widehat{S}^l_i \cap O_{l+1}}\ d\omega^{X_{\eta\widehat{S}^l_i}}=: \widehat{II}_3' - \widehat{II}_3''.
\end{equation*}
The boundary H\"{o}lder continuity estimate in \eqref{eq:newbdryH} shows that
\[
\widehat{II}_3' = \omega^{X_{\eta\widehat{S}^l_i}}(\widehat{S}^l_i) 
= 1 - \omega^{X_{\eta\widehat{S}^l_i}}(\RR^n \setminus \widehat{S}^l_i )
\geq 1 - C_0\eta^{\alpha}
\]
whilst repeating the arguments used to estimate $II_3''$ shows that
\begin{equation*}
\widehat{II}_3'' \leq \frac{C_{\eta}}{\omega^{X_{2Q_0}}(\widehat{S}^l_i)} \omega^{X_{2Q_0}}(O_{l+1} \cap \widehat{S}^l_i) 
\leq \frac{C_{\eta}}{\omega^{X_{2Q_0}}(S^l_i)} \omega^{X_{2Q_0}}(O_{l+1} \cap S^l_i)
\leq C_{\eta}\epsilon_0.
\end{equation*}
These estimates together show that if $\epsilon_0\in(0,1/2)$, then either
\begin{equation}\label{216}
 0 \leq u( X_{\eta\widehat{S}^l_i}) \leq C_0\eta^{\alpha} + 3C_{\eta}\epsilon_0
\quad\textrm{or}\quad
u( X_{\eta\widehat{S}^l_i}) \geq 1 - \left(C_0\eta^{\alpha} + C_{\eta}\epsilon_0\right).
\end{equation}

The estimates \eqref{eq:est1} and \eqref{216} together imply that
\[
|u( X_{\eta S^l_i}) - u( X_{\eta\widehat{S}^l_i})| \geq c_0 - 2C_0\eta^{\alpha} - 4C_{\eta}\epsilon_0.
\]
We thus obtain \eqref{214} by first choosing $\eta\in(0,1)$ so that $2C_0\eta^{\alpha}\leq c_0/4$ and then choosing $\epsilon_0\in(0,1/2)$ (depending on $\eta$) so that $4C_{\eta}\epsilon_0 \leq c_0/4$. These choices of $\eta$ and $\epsilon_0$, which depend only on the allowed constants, are now fixed.

To complete the proof, suppose that $M\geq 1$ and $x\in E$, and recall that $\delta\in(0,\delta_0)$ remains to be chosen, where $\delta_0$ is now fixed by our choice of $\epsilon_0$ as in Lemma \ref{lemma25}. First, fix a cube $S^k$ in $\{S^k_i\}_{i\in\mathbb{N}}$ such that $x\in S^k$. The remarks after Definition~\ref{def:epscover} then imply that for each $l\in\{1,\ldots,k-1\}$, there exists a unique cube $S^l$ in $\{S^l_i\}_{i\in\mathbb{N}}$ such that $x\in S^l$ and $S^{l+1}\subset S^l$, thus $\ell(S^{l+1})\leq \tfrac{1}{2}\ell(S^l)$. Next, for each $l\in\{1,\ldots,k\}$, fix a dyadic descendant $\widehat{S}^l$ in $\DD^\eta(S^l)$ such that $x\in \widehat{S}^l$.

Observe that, for some $\tau\in(0,1)$ sufficiently close to $1$ and depending only on $\eta$, the corkscrew points $X_{\eta S^l}$ and $X_{\eta\widehat{S}^l}$ both belong to the dilate $\tau Q^l_\eta$ of the cube
\[
Q^l_\eta:=\{(y,t)\in\RR^{n+1}_+ : |y-x|_\infty < (\tfrac{1}{2}+\tfrac{\eta^2}{4})\ell(S^l),\ \tfrac{\eta^2}{2}\ell(S^l) < t < (1+\eta^2)\ell(S^l)\}
\]
with $\ell(Q^l_\eta)=(1+\tfrac{\eta^2}{2})\ell(S^l)$. Therefore, if $c^l:=\fint_{Q^l_\eta} u$, then the Moser-type estimate in \eqref{eq:M}, the Poincar\'{e} inequality in \eqref{wPoinc} and the doubling property of $\mu$ show that
\begin{align}\begin{split}\label{eqMPD}
|u( X_{\eta S^l}) - u( X_{\eta\widehat{S}^l})|^2 
&\lesssim |u( X_{\eta S^l})-c^l|^2 + |u( X_{\eta\widehat{S}^l})-c^l|^2 \\ 
&\lesssim \|u - c^l\|_{L^\infty(\tau Q^l_\eta)}^2 \\
&\lesssim_\eta \fint_{Q^l_\eta} |u - c^l|^2\ d\mu \\
&\lesssim \ell(Q^l_\eta)^2\fint_{Q^l_\eta} |\nabla u|^2\ d\mu \\
&\lesssim \frac{\ell(S^l)}{\mu\big(\Delta(x,(1+\tfrac{\eta^2}{2})\ell(S^l))\big)}\int_{Q^l_\eta} |\nabla u|^2\ d\mu\\
&\lesssim \iint_{Q^l_\eta} |t\nabla u(y,t)|^2\ \frac{d\mu(y)}{\mu(\Delta(x,t))}\frac{dt}{t}.
\end{split}\end{align}

Iterating the bound $\ell(S^{l+1})\leq \tfrac{1}{2}\ell(S^l)$ shows that $\ell(S^{l'})\leq 2^{l-l'}\ell(S^{l})$ when $l'\geq l$. This implies that the collection $\{Q^1_\eta,\ldots,Q^k_\eta\}$ has the bounded intersection property whereby for each $l\in\{1,\ldots,k\}$, there are at most $3+2\log_2(\tfrac{1}{\eta^2}+1))$ such cubes $Q^{l'}_\eta$ satisfying $Q^{l'}_\eta \cap Q^l_\eta \neq \emptyset$. This allows us to sum estimate \eqref{eqMPD} over $l\in\{1,\ldots,k\}$ and then apply \eqref{214} to obtain
\[
k \lesssim_\eta \iint_{\cup_{l=1}^k Q^l_\eta} |t\nabla u(y,t)|^2\ \frac{d\mu(y)}{\mu(\Delta(x,t))}\frac{dt}{t} 
\lesssim \int_0^{\gamma\ell(Q)}\int_{\Delta(x,\gamma t)}|t\nabla u(y,t)|^2\frac{d\mu(y)}{\mu(\Delta(x,t))}\frac{dt}{t}
\]
for some $\gamma>0$ that depends only on $\eta>0$ and thus only on the allowed constants.

To conclude, recall that $k \eqsim \log(\omega^{X_{2Q_0}}(E)^{-1})/\log(1/\epsilon_0)\geq \log(1/\delta)/\log(1/\epsilon_0)$, since $\omega^{X_{2Q_0}}(E)\leq\delta<1$. Therefore, the result follows by choosing $\delta\in(0,\delta_0]$ such that $M\leq\log(1/\delta)$, since $\delta_M\!:=\delta$ depends only on $M$ and the allowed constants.
\end{proof}

We now combine the above technical lemma with the Carleson measure estimate from Theorem~\ref{thm:introCME} to prove the main $A_\infty$-estimate for degenerate elliptic measure.

\begin{theorem}\label{thm:Ainftymain}
Suppose that $Q_0$ is a cube in $\RR^n$. If $X\in \RR^{n+1}_+\setminus T_{Q_0}$ and $\omega:=\omega^X\lfloor{Q_0}$ denotes the degenerate elliptic measure restricted to $Q_0$, then $\omega\in A_\infty(\mu)$ and the following equivalent properties hold:
\begin{enumerate}
\item For each $\epsilon\in(0,1)$, there exists $\delta\in(0,1)$, depending only on $\epsilon$, $n$, $\lambda$, $\Lambda$ and $[\mu]_{A_2}$, such that the following property holds: If $Q\subseteq Q_0$ is a cube and $E \subseteq Q$ such that $\omega(E) \leq \delta \omega(Q)$, then $\mu(E) \leq \epsilon \mu(Q)$.
\item The measure $\omega$ is absolutely continuous with respect to $\mu$ and there exists $q\in(1,\infty)$  such that the Radon--Nikodym derivative $k := d\omega/d\mu$ satisfies, on all surface balls $\Delta \subseteq Q_0$, the reverse H\"{o}lder estimate 
\[
\left(\fint_{\Delta} k^q\ d\mu\right)^{1/q}
\lesssim \fint_{\Delta} k\ d\mu,
\]
where $q$ and the implicit constant depend only on $n$, $\lambda$, $\Lambda$ and $[\mu]_{A_2}$.
\item There exist $C, \theta>0$, depending only on $n$, $\lambda$, $\Lambda$ and $[\mu]_{A_2}$, such that
\[
\omega(E) \leq C \left(\frac{\mu(E)}{\mu(Q)}\right)^\theta \omega(Q)
\]
for all cubes $Q\subseteq Q_0$ and all Borel sets $E\subseteq Q$.
\end{enumerate}
\end{theorem}
\begin{proof}
It is well-known that (1)--(3) are equivalent (see Theorem 1.4.13 in \cite{K}). Moreover, by Lemma~\ref{138}, it suffices to prove (1) when $X=X_{2Q_0}$.  In that case, by Lemma~\ref{28}, the Carleson measure estimate in Theorem~ \ref{thm:introCME}, Fubini's Theorem and the doubling property of $\mu$, it follows that for each $M\geq 1$, there exists $\delta_M>0$, depending only on $M$ and the allowed constants, such that the following property holds: If $Q\subseteq Q_0$ is a cube and $E\subseteq Q$ such that $\omega(E) \leq \delta_M \omega(Q)$, then there exists a solution $u$ of the equation $\div (A \nabla u)=0$ in $\RR^{n+1}_+$ with $\|u\|_\infty\leq1$ such that
\begin{align*}
M\mu(E) 
&\leq \int_{E} \int_0^{\gamma \ell(Q)}\int_{\Delta(x,\gamma t)}|t\nabla u(y,t)|^2\ \frac{d\mu(y)}{\mu(\Delta(x,t))}\frac{dt}{t}\ d\mu(x) \\
&\lesssim \int_0^{\tilde{\gamma}\ell(Q)}\int_{\tilde{\gamma}Q} |t\nabla u(y,t)|^2\ d\mu(y)\frac{dt}{t} 
\lesssim \mu(Q),
\end{align*}
where the implicit constants and $\tilde\gamma>\gamma>0$ depend only on the allowed constants. Therefore, if $\epsilon\in(0,1)$, we choose $M(\epsilon) \geq 1$ and thus $\delta_{M(\epsilon)} \in(0,1)$, depending only on $\epsilon$ and the allowed constants, such that $\mu(E) \leq \epsilon \mu(Q)$, as required.
\end{proof}

\subsection{The square function and non-tangential maximal function estimates}

The $L^p_\mu(\RR^n)$-norm equivalence between the square function $Su$ and the non-tangential maximal function $N_*u$ of solutions $u$ in Theorem~\ref{thm:SNS} is now a corollary of the main $A_\infty$-estimate for the degenerate elliptic measure in Theorem~\ref{thm:Ainftymain}. This was proved by Dahlberg, Jerison and Kenig in Theorem~1 of~\cite{DJK}, which actually provides the more general result in Theorem~\ref{thm:DJKSNS} below. In particular, the degenerate elliptic case is treated on page 106 of~\cite{DJK}, noting that the normalisation $u(X_0)=0$ assumed therein is actually only required for the so-called $N\lesssim S$-estimate.

\begin{theorem}\label{thm:DJKSNS}
Suppose that $\Phi: [0,\infty) \!\rightarrow\! [0,\infty)$ is an unbounded, non-decreasing, continuous function with $\Phi(0)=0$ and $\Phi(2t) \leq C \Phi(t)$ for all $t>0$ and some $C>0$. If $\div( A\nabla u) = 0$ in $\RR^{n+1}_+$, then
\[
\int_{\RR^n} \Phi(Su)\ d\mu \lesssim \int_{\RR^n} \Phi(N_*u)\ d\mu,
\]
and if, in addition, $u(X_0)=0$ for some $X_0\in\RR^{n+1}_+$, then
\[
\int_{\RR^n} \Phi(N_*u)\ d\mu \lesssim \int_{\RR^n} \Phi(Su)\ d\mu,
\]
where the implicit constants depend only on $X_0$, $\Phi$, $n$, $\lambda$, $\Lambda$ and $[\mu]_{A_2}$.
\end{theorem}

The next result is also a consequence of the main $A_\infty$-estimate in Theorem~\ref{thm:Ainftymain}. It will allow us to construct solutions to the Dirichlet problem $(D)_{p,\mu}$ as integrals of $L^p_\mu(\RR^n)$-boundary data with respect to degenerate elliptic measure.

\begin{lemma}\label{lem13}
Suppose that $\tfrac{1}{p}+\tfrac{1}{q}=1$, where $q\in (1,\infty)$ is the reverse H\"{o}lder exponent from Theorem~\ref{thm:Ainftymain}. If $X=(x,t)\in\RR^{n+1}_+$, then the Radon--Nikodym derivative $k(X,\cdot) := d\omega^X/d\mu$ is in $L^q_\mu(\RR^n)$ and
\[
\int_{\RR^n} k((x,t),y)^q\ d\mu(y) \lesssim \mu(\Delta(x,t))^{1-q}.
\]
Moreover, if $f\in L^p_\mu(\RR^n)$ and $u(X) := \int_{\RR^n} f(y) \ d\omega^{X}$, then
$
\|N_* u\|_{L^p_\mu(\RR^n)} \lesssim \|f\|_{L^p_\mu(\RR^n)}.
$
The implicit constant in each estimate depends only on $n$, $\lambda$, $\Lambda$ and $[\mu]_{A_2}$.
\end{lemma}

\begin{proof}
Suppose that $X=(x,t)\in\RR^{n+1}_+$. The proof of Proposition~\ref{Ker:sup_bound} shows that
\[
k((x,t),y) \lesssim 2^{-j\alpha} \frac{k((x,2^j t),y)}{\omega^{(x,2^j t)}(\Delta(x,2^jt))}
\qquad \forall y\in \Delta(x, 2^jt) \setminus \Delta(x, 2^{j-1}t),\ \forall j\in\mathbb{N}.
\]
Applying the reverse H\"{o}lder estimate from Theorem~\ref{thm:Ainftymain} then shows that
\begin{align*}
\int_{\RR^n} &k((x,t),y)^q\ d\mu(y) \\
&= \int_{\Delta(x,t)} k((x,t),y)^q\ d\mu(y)
+ \sum_{j=1}^\infty \int_{\Delta(x, 2^{j}t) \setminus \Delta(x, 2^{j-1}t)} k((x,t),y)^q\ d\mu(y) \\
&\lesssim \mu(\Delta(x,t))^{1-q} + \sum_{j=1}^\infty 2^{-j\alpha q}  \mu(\Delta(x, 2^jt))^{1-q}
\lesssim \mu(\Delta(x,t))^{1-q}.
\end{align*}

To obtain the non-tangential maximal function estimate, it suffices to consider the case when $f\geq 0$, since in general we may then decompose $f=f^+-f^-$ into its positive and negative parts $f^+,f^-\geq0$. To this end, suppose that $x_0 \in \RR^n$ and that $X=(x,t) \in \RR^{n+1}_+$ in order to write
\[
f = f \mathbbm{1}_{\Delta(x_0,2t)} +\sum_{j=1}^{\infty} f \mathbbm{1}_{\Delta(x_0, 2^{j+1}t) \setminus \Delta(x_0, 2^jt)} =: \sum_{j=0}^{\infty}f_j
\]
and define
\[
u_j(X) := \int_{\RR^n} f_j(y)\ d\omega^X(y) = \int_{\RR^n} f_j(y)\, k(X,y)\ d\mu(y).
\]
The self-improvement property of the reverse H\"{o}lder estimate from Theorem~\ref{thm:Ainftymain} (see Theorem 1.4.13 in \cite{K}) implies that there exists an exponent $r>q$ such that 
\begin{equation}\label{eq:knewest}
\left(\fint_\Delta k((x,t),y)^r\ d\mu(y)\right)^{1/r}
\lesssim \fint_\Delta k((x,t),y)\ d\mu(y)
\leq \frac{1}{\mu(\Delta)}
\end{equation}
for all surface balls $\Delta \subseteq \Delta(x,t/2)$.

Now suppose that $X=(x,t) \in \Gamma(x_0)$. To estimate $u_0$, we apply the interior Harnack inequality in \eqref{eq:H} followed by H\"{o}lder's inequality and \eqref{eq:knewest} to obtain
\begin{align*}
u_0(x,t) \eqsim u_0(x, 6t) &\leq \int_{\Delta(x_0,2t)} f(y)\, k((x, 6t),y)\ d\mu(y)\\
&\leq \left(\int_{\Delta(x_0,2t)} |k((x, 6t),y)|^{r}\ d\mu(y) \right)^{1/r}
\left( \int_{\Delta(x_0,2t)}f(y)^{r'}\ d\mu(y)\right)^{1/r'} \\
&\lesssim \mu(\Delta(x_0, 2t))^{-1/r'} \left(\int_{\Delta(x_0,2t)}f(y)^{r'}\ d\mu(y)\right)^{1/r'} \\
&\leq [M_{\mu}(f^{r'})(x_0)]^{1/r'}.
\end{align*}
To estimate $u_j$ when $j\in\N$, we apply the boundary H\"{o}lder continuity estimate from~\eqref{eq:newbdryH} and then proceed as in the estimate above to obtain
\begin{align*}
u_j(x,t) &\lesssim  \left(\frac{t}{2^{j}t}\right)^{\alpha} u_j(x_0, 2^jt) 
\eqsim   2^{-j \alpha} u_j(x_0, 2^{j+2} t) \\
&\leq 2^{-j \alpha} \int_{\Delta(x_0, 2^{j+1}t)} f(y)\, k((x_0, 2^{j+2}t),y)\ d\mu(y)\\
&\leq 2^{-j \alpha}\left(\int_{\Delta(x_0, 2^{j+1}t)} k((x_0, 2^{j+2}t),y)^r\ d\mu(y)\right)^{1/r}
\left( \int_{\Delta(x_0, 2^{j+1}t)}f(y)^{r'}\ d\mu(y)\right)^{1/r'}\\
&\lesssim 2^{-j \alpha}\left(\fint_{\Delta(x_0, 2^{j+1}t)}f(y)^{r'}  d\mu(y)\right)^{1/r'}\\
&\leq 2^{-j \alpha} [M_{\mu}(f^{r'})(x_0)]^{1/r'}.
\end{align*}
The above estimates together show that $N_{*}u(x_0) \lesssim [M_{\mu}(f^{r'})(x_0)]^{1/r'}$ for all $x_0\in \RR^n$, and since $r'<q'=p$, it follows that $\|N_*u\|_{L^p_{\mu}} \lesssim \|f\|_{L^p_{\mu}}$, as required.
\end{proof}

We conclude the paper by using the preceding lemma to obtain solvability of the Dirichlet problem $(D)_{p,\mu}$. A uniqueness result is also obtained but only for solutions that converge uniformly to 0 at infinity. This restriction does not appear in the uniformly elliptic case (see Theorem~1.7.7 in \cite{K}). It arises here because of the absence of a Green's function for degenerate elliptic equations on unbounded domains (see Section~\ref{Sec:PreEst}) and it is not clear to us whether this can be improved.

\begin{theorem}\label{thm:DpSolvmain}
Suppose that $\tfrac{1}{p}+\tfrac{1}{q}=1$, where $q\in (1,\infty)$ is the reverse H\"{o}lder exponent from Theorem~\ref{thm:Ainftymain}. The Dirichlet problem for $L^p_\mu(\RR^n)$-boundary data is solvable in the sense that for each $f\in L^p_\mu(\RR^n)$, there exists a solution $u$ such that
\begin{equation}\tag*{$(D)_{p,\mu}$}
\begin{cases}
\div (A \nabla u)=0 \textrm{ in } \RR^{n+1}_+,\\
N_* u \in L^p_{\mu}(\RR^n),\\
\lim_{t\rightarrow 0} u(\cdot,t) = f,
\end{cases}
\end{equation}
where the limit converges in $L^p_\mu(\RR^n)$-norm and in the non-tangential sense whereby $\lim_{\Gamma(x)\ni(y,t)\rightarrow (x,0)} u(y,t) = f(x)$ for almost every $x\in\RR^n$. Moreover, if $f$ has compact support, then there is a unique solution $u$ of $(D)_{p,\mu}$ that converges uniformly to 0 at infinity in the sense that $\lim_{R\rightarrow\infty} \|u\|_{L^\infty(\RR^{n+1}_+\setminus B(0,R))} = 0$.
\end{theorem}

\begin{proof}
Suppose that $f\in L^p_\mu(\RR^n)$ and define $u(X):=\int_{\RR^n} f\ d\omega^X$ for all $X\in\RR^{n+1}_+$. We first prove that $\div( A \nabla u)=0$ in $\RR^{n+1}_+$. Let $(f_j)_j$ denote a sequence in $C_c(\RR^n)$ that converges to $f$ in $L^p_\mu(\RR^n)$ and consider the  solutions $u_j(X):=\int_{\RR^n} f_j\ d\omega^X$. The $L^q(\RR^n)$-estimate for the Radon--Nikodym derivative $d\omega^X/d\mu$ from Lemma~\ref{lem13} and the doubling property of $\mu$ show that $\|u_j-u\|_{L^\infty(K)} \lesssim_{\mu,K} \|f_j-f\|_{L^p_\mu(\RR^n)}$ for all $j\in\N$ and any compact set $K\subset \RR^{n+1}_+$, so $u_j$ converges to $u$ in $L^2_{\mu,\loc}(\RR^n)$. Moreover, Cacioppoli's inequality and the arguments preceding \eqref{eq:solcheck} show that $u_j$ converges to a solution $v$ in $W^{1,2}_{\mu,\loc}(\RR^n)$, so then $u=v$ is a solution in $\RR^{n+1}_+$ as required .

The non-tangential maximal function estimate $\|N_* u\|_{L^p_{\mu}(\RR^n)} \lesssim \|f\|_{L^p_{\mu}(\RR^n)}$ is given by Lemma~\ref{lem13}. To prove the non-tangential convergence to the boundary datum, first recall that $u_j \in C(\overline{\RR^{n+1}_+})$ with $u_j|_{\RR^n}:=f_j$, so $\lim_{\Gamma(x)\ni(y,t)\rightarrow (x,0)} u_j(y,t) = f_j(x)$ (see Section \ref{subsection:defellipticmeasure}). We combine this fact with the bound 
\[
|u(y,t) - f(x)| \leq  |u(y,t) - u_j(y,t)| + |u_j(y,t) - f_j(x)| + |(f_j-f)(x)|
\]
to obtain
\[
\limsup_{\Gamma(x)\ni(y,t)\rightarrow (x,0)} |u(y,t) - f(x)|
\leq |N_*(u-u_j)(x)| + |(f-f_j)(x)|
\]
for all $x\in\RR^n$. For any $\eta >0$, we then apply Chebyshev's inequality and the non-tangential maximal function estimate from Lemma~\ref{lem13}, to show that
\begin{align*}
\mu\Big(\Big\lbrace x &\in \RR^n: \limsup_{\Gamma(x)\ni(y,t)\rightarrow (x,0)} |u(y,t) - f(x)| > \eta \Big\rbrace\Big) \\
&\leq\mu(\lbrace x \in \RR^n:  N_*(u-u_j)(x) > \eta/2 \rbrace)
+ \mu(\lbrace x \in \RR^n: |(f-f_j)(x)| > \eta/2 \rbrace) \\
&\lesssim \eta^{-p} \left(\|N_*(u-u_j)\|_{L^p_\mu(\RR^n)}^p + \|f-f_j\|_{L^p_\mu(\RR^n)}^p\right) \\
&\lesssim \eta^{-p} \|f-f_j\|_{L^p_\mu(\RR^n)}^p.
\end{align*}
It follows, since $f_j$ converges to $f$ in $L^p_\mu(\RR^n)$, that $\lim_{\Gamma(x)\ni(y,t)\rightarrow (x,0)} u(y,t) = f(x)$ for almost every $x\in\RR^n$, as required. The norm convergence $\lim_{t\rightarrow 0} \|u(\cdot,t)-f\|_{L^p_\mu(\RR^n)}$ then follows by Lebesgue's dominated convergence theorem.

It remains to prove that $u$ is the unique solution satisfying $\lim_{|X|\rightarrow\infty} \|u(X)\|_\infty = 0$ when $f$ has compact support. In that case, fix $R_0>0$ such that $f$ is supported in the surface ball $\Delta(0,R_0)$. If $X \in \RR^{n+1}_+$ and $|X|> 2R_0$, then the reverse H\"{o}lder estimate in Theorem~\ref{thm:Ainftymain} shows that
\begin{align*}
|u(X)| &\leq \int_{\Delta(0,R_0)} |f(y)|\, k(X,y)\ d\mu(y)\\
&\leq \|f\|_{L^p_\mu(\RR^n)} \left(\int_{\Delta(0,|X|/2)}k(X,y)^q\ d\mu(y)\right)^{1/q}\\
&\lesssim \|f\|_{L^p_\mu(\RR^n)} \mu(\Delta(0,|X|/2))^{1/q} \fint_{\Delta(0,|X|/2)}k(X,y)\ d\mu(y) \\
&\leq \|f\|_{L^p_\mu(\RR^n)} \mu(\Delta(0,|X|/2))^{-1/p},
\end{align*}
whilst $\lim_{R\rightarrow\infty}\mu(\Delta(0,R))=\infty$, since $\mu$ is in the $A_\infty$-class with respect to Lebesgue measure on $\RR^n$, thus $\lim_{R\rightarrow\infty} \|u\|_{L^\infty(\RR^{n+1}_+\setminus B(0,R))}= 0$. The maximum principle allows us to conclude that any solution of $(D)_{p,\mu}$ with this decay must be unique.
\end{proof}

\end{document}